 \documentclass[final,1p,times]{elsarticle}

\usepackage{amssymb}
 \usepackage{amsthm}
\usepackage{amsmath,amssymb,amsopn,amsfonts,mathrsfs,amsbsy,amscd}
\usepackage{longtable}
\usepackage{caption}
\usepackage{multirow}

 \usepackage{lineno}

\newcommand{\g}{\mathfrak{g}}
\newcommand{\p}{{\mathfrak{p} }}
\newcommand{\prs}{\langle\;,\;\rangle}

\newcommand{\too}{\longrightarrow}
\newcommand{\om}{\omega}
\newcommand{\esp}{\quad\mbox{and}\quad}

\def\br{[\;,\;]}

\newcommand{\G}{{\mathfrak{g}}}
\newcommand{\U}{{\mathfrak{U}}}
\newcommand{\A}{{\mathcal{A}}}

\newcommand{\h}{{\mathfrak{h}}}
\newcommand{\ad}{{\mathrm{ad}}}
\newcommand{\tr}{{\mathrm{tr}}}

\newcommand{\B}{{\cal B}}

\newcommand{\Om}{\Omega}

\newcommand{\si}{\sigma}
\newcommand{\al}{\alpha}
\newcommand{\be}{\beta}
\newcommand{\ga}{\gamma}
\newcommand{\Ga}{\Gamma}
\newcommand{\e}{\epsilon}

\newcommand{\la}{\lambda}

\newcommand{\de}{\delta}
\newcommand{\R}{\mathbb{R}}

\newtheorem{theo}{Theorem}[section]
\newtheorem{pr}{Proposition}[section]

\newtheorem{co}{Corollary}[section]

\newtheorem{exem}{Example}
\newtheorem{remark}{Remark}

{%
	\stepcounter{equation}
	\begin{equation*}}{%
	\leqno(\arabic{equation})
	\end{equation*}
}

\begin{document}

\begin{frontmatter}


 

\title{   Left invariant generalized  complex and K\"ahler structures  on simply connected four dimensional Lie groups: classification and invariant cohomologies}

 \author[label1,label2]{Mohamed Boucetta, Mohammed Wadia Mansouri}
 \address[label1]{Universit\'e Cadi-Ayyad\\
 	Facult\'e des sciences et techniques\\
 	BP 549 Marrakech Maroc\\e-mail: m.boucetta@uca.ac.ma
 }
 \address[label2]{Universit\'e Ibno Tofail\\ Facult\'e des Sciences\\e-mail: mansouriwadia@gmail.com}
 



\begin{abstract}  We give a complete classification of  left
	invariant generalized  complex structures of type 1  on four dimensional simply connected  Lie
groups and we compute for each class its invariant generalized  Dolbeault cohomology, its invariant generalized Bott-Chern cohomology and its invariant generalized Aeppli cohomology. We classify also left invariant generalized K\"ahler structures on four dimensional simply connected  Lie
groups.
	
\end{abstract}


\end{frontmatter}







\section{Introduction} \label{section1}
Generalized complex geometry, recently introduced by   Hitchin \cite{H} and developed by Gualtieri \cite{Ga}, unifies complex and symplectic geometries and shares many properties with them. Several aspects of this geometry and its possible applications in string theory have been studied by many authors (see for instance \cite{an, C, string}). On the other hand, in \cite{alex} some classes of left invariant generalized complex structures on semi-simple Lie groups have been described and in \cite{cav} all  6-dimensional nilmanifolds having invariant generalized complex structures were classified.

The  purpose of this paper is to give a complete classification of  left
invariant generalized  complex structures of type 1 on  simply connected four dimensional  Lie
groups.  Also, for each obtained class, we compute its invariant  generalized  Dolbeault cohomology, its invariant generalized Bott-Chern cohomology and its invariant generalized Aeppli cohomology. As an application, we give also a complete classification of left
invariant generalized  K\"ahler structures  on  simply connected four dimensional  Lie
groups.
Generalized complex structures of type 2 and 0 are equivalent, respectively, to  left invariant complex structures and left invariant symplectic structures and, in dimension four, were   classified in \cite{O1, O2, Snow}. Classical left invariant K\"ahler structures on four dimensional Lie groups were also classified in \cite{O1}.

Let $M$ be a differentiable manifold and let
\[\mathcal{TM}=TM\oplus T^*M.\]
The Courant bracket and the scalar product of two sections $X+\al$, $Y+\be$ of $\mathcal{TM}$ are given by
\[ [X+\al,Y+\be]_c=[X,Y]+\mathcal{L}_X\be-\mathcal{L}_Y\al-
\frac12d\left(\prec\be,X\succ-\prec\al,Y\succ \right) \esp
\langle X+\al,Y+\be\rangle=\frac12(\prec\al,Y\succ+\prec\be,X\succ). \]

We recall that  a generalized complex structure on  $M$ is an endomorphism field $\mathcal{J}\in \mathrm{End}(\mathcal{TM})$ such that, for any $a,b\in\Ga(\mathcal{TM})$, 
\[\mathcal{J}^2=-\mathrm{Id},\; \langle \mathcal{J}a,b\rangle+\langle \mathcal{J}b,a\rangle=0\esp N_{\mathcal{J}}=0, \]
where $N_{\mathcal{J}}$ is the Nijenhuis torsion of $\mathcal{J}$ with respect to the Courant bracket, i.e., $$
N_{\mathcal{J}}(a,b)=[\mathcal{J}a,\mathcal{J}b]_c-\mathcal{J}[\mathcal{J}a,b]_c-\mathcal{J}[a,\mathcal{J}b]_c+\mathcal{J}^2[a,b]_c.$$

The study of left invariant generalized complex structures on a Lie group can be carried out at the level of its Lie algebra. More precisely, 
let $G$ be a simply connected Lie group and $(\G,\br)$ its Lie algebra. The neutral metric $\prs$ and  the Courant bracket    when restricted to $\Phi(\G)=\G\oplus\G^*$ define a neutral metric denoted similarly by $\prs$ and a   bracket given, for any $u,v\in\G,\al,\be\in\G^*$, by
\[ [u+\al,v+\be]^\rhd=[u,v]+\ad_u^t\be-\ad_v^t\al, \]where
$\prec\ad_u^t\al,v\succ=-\prec\al,[u,v]\succ$.  This bracket satisfies the Jacobi identity even if in general the Courant bracket doesn't, moreover $\prs$ is bi-invariant.

A left invariant generalized complex structure on $G$ is equivalent to  a complex structure on $(\Phi(\G),\br^\rhd)$ which is skew-symmetric with respect to the bi-invariant symmetric 2-form $\prs$. That is   an endomorphism $K:\Phi(\G)\too\Phi(\G)$ such that, $K^2=-\mathrm{Id}_{\Phi(\G)}$ and $N_K=0$. With respect to the splitting $\Phi(\G)=\G\oplus\G^*$, $K$ can be written
\begin{equation}\label{k} K=\left(\begin{array}{cc}J&R\\ \si&-J^*\end{array} \right) \end{equation}
where $J:\G\too\G$ is an endomorphism, $J^*:\G^*\too\G^*$ its dual and $R:\G^*\too\G$,  $\si:\G\too\G^*$ are skew-symmetric linear maps. When $K$ defines a left invariant generalized complex structure on $G$, the triple $(J,R,\sigma)$ will be called generalized complex structure on $\G$. This is the main object of this paper. There are two ways of characterizing generalized complex structures on Lie algebras. Let us give the first one which is an adaptation of the  result   established in \cite[Proposition 2.2]{C}. 

\begin{pr}\label{pr2} Let $(\G,J,R,\si)$ be a Lie algebra endowed with a triple as above. Then $(J,R,\si)$ is a generalized complex structure if and only if, for any $u,v,w\in\G$, $\al,\be\in\G^*$,
	\begin{enumerate}
		\item[$(C0)$] $J^2+{R} \circ \si=-\mathrm{Id}_{\G}$,
		$J\circ\mathrm{R}={R}\circ J^*$ and $\si\circ J=J^*\circ\si$,
		\item[$(C1)$] $[{R},{R}]=0$,
		\item[$(C2)$] $J^*[\al,\be]_{R}=\ad_{{R}(\al)}^tJ^*\be-\ad_{{R}(\be)}^tJ^*\al$,
		\item[$(C3)$] $N_{J}(u,v)={R}(i_{u\wedge v}\mathrm{d}\si^\flat)$, where $\si^\flat(u,v)=\prec\si(u),v\succ$,
		\item[$(C4)$] $\mathrm{d}\si_{J}(u,v,w)=\mathrm{d}\si^\flat(Ju,v,w)+\mathrm{d}\si^\flat(u,Jv,w)+\mathrm{d}\si^\flat(u,v,Jw)$,
		where $\si_{J}(u,v)=\si^\flat(Ju,v),$
	\end{enumerate}where $[R,R]\in\wedge^3\G$ is the Schouten bracket given by
	\[ [R,R](\al,\be,\ga)=\prec\al,[R(\be),R(\ga)]\succ+\prec\be,[R(\ga),R(\al)]\succ+\prec\ga,[R(\al),R(\be)]\succ,
 \] and $\br_R$ is the bracket on $\G^*$ given by
	\[ [\al,\be]_R=\ad_{R(\al)}^t\be-\ad_{R(\be)}^t\al. \]
	\end{pr}
	
	Generalized complex structures can be also characterized by using the spinors (see \cite{Ga}). Let $(\G,J,R,\si)$ be a generalized complex structure.  Consider $\mathrm{L}$  the maximal isotropic subspace of $\Phi(\G)\otimes\mathbb{C}$  given by
	\[ \mathrm{L}=\left\{ X+\xi-i(JX+R(\xi)+\si(X)-J^*\xi),X\in\G,\xi\in\G^*     \right\}. \] According to \cite{Ga}, there exists a spinor $\rho\in\wedge^\bullet\G^*\otimes\mathbb{C}$ such that
	\[ \mathrm{L}=\mathrm{Ann}(\rho)=\left\{ u\in\Phi(\G)\otimes\mathbb{C}, u.\rho=0    \right\} \]	where $(X+\xi).\rho=i_X\rho+\xi\wedge\rho$. Moreover, $(\G,J,R,\si)$ satisfies the condition of Proposition \ref{pr2} if and only if there exists $X+\xi\in\Phi(\G)$ such that
	\begin{equation}
	\mathrm{d}\rho=(X+\xi).\rho.
	\end{equation}The  form $\rho$ is called the pure spinor associated to  $(\G,J,R,\si)$.
	Note that if $\om$ is a symplectic 2-form on $\G$ and $J:\G\too\G$ is a complex isomorphism with $N_J=0$ then
	\[ K^\om=\left(\begin{array}{cc}0&-\om^{-1}\\\om&0\end{array} \right)\quad\mbox{respectively}\quad K^J=\left(\begin{array}{cc}J&0\\0&-J^*\end{array} \right) \]	are generalized complex structures on $\G$.

	According to \cite[Definition 3.5]{Ga}, the type of a generalized complex structure $(\G,J,R,\si)$ is the integer $$k=\frac12\dim_\R(\G^*\cap K(\G^*))=\frac12\dim_\R(\h^0),$$ where $\h=\mathrm{Im\;}R$ and $\h^0$ its annulator. 
	In the case of a four dimensional Lie algebra, we have three types: 0,1 or 2. The type 0 corresponds to $R$ invertible, $1$ to $R$ of rank 2 and 2 to $R=0$.

	In this paper, we classify generalized complex structures and we will give now a precise definition of the equivalence relation under which the classification will be carried on.
	
	An automorphism of a Lie algebra $\G$ is an isomorphism which preserves the Lie bracket. A 2-cocycle of $\G$ is an endomorphism $B:\G\too\G^*$ which is skew-symmetric and satisfies, for any $u,v,w\in\G$,
	\[ \prec B([u,v]),w\succ+\prec B([v,w]),u\succ+\prec B([w,u]),v\succ=0. \]
	Let $(\G,J,R,\si)$ be a generalized complex structure on $\G$, $A:\G\too\G$ an automorphism of $\G$ and $B:\G\too\G^*$ a 2-cocycle.  Then
	\begin{eqnarray} K_A&=&\left(\begin{array}{cc}A&0\\0&(A^{-1})^*\end{array} \right)\left(\begin{array}{cc}J&R\\\si&-J^*\end{array} \right)\left(\begin{array}{cc}A^{-1}&0\\0&A^*\end{array} \right)=
	\left(\begin{array}{cc}AJA^{-1}&ARA^*\\(A^{-1})^*\si A^{-1}&-(A^{-1})^*J^*A^*\end{array} \right),\nonumber\\
	K_B&=&\left(\begin{array}{cc}\mathrm{Id}_\G&0\\B&\mathrm{Id}_{\G^*}\end{array} \right)\left(\begin{array}{cc}J&R\\\si&-J^*\end{array} \right)\left(\begin{array}{cc}\mathrm{Id}_\G&0\\-B&\mathrm{Id}_{\G^*}\end{array} \right)=\left(\begin{array}{cc}J-RB&R\\BJ+\si-BRB+J^*B&BR-J^*\end{array} \right),\nonumber\\ \label{t} \end{eqnarray}
	are also generalized complex structures on $\G$. It is a consequence of the fact that the transformations $\phi(A)=\left(\begin{array}{cc}A&0\\0&(A^{-1})^*\end{array} \right)$ and $\exp(B)=\left(\begin{array}{cc}\mathrm{Id}_\G&0\\B&\mathrm{Id}_{\G^*}\end{array} \right)$ preserve the Courant bracket and the metric on $\Phi(\G)$. 
	The transformation $\exp(B)$ is known as a $B$-transformation. We refer to $\phi(A)$ as an automorphism transformation. 
	As a consequence, we adopt the following relation of equivalence between generalized complex structures.
	Through this paper, two generalized complex structures on $\G$ will be called equivalent if they are conjugate by a product of transformations of the form $\phi(A)$ and $\exp(B)$.

It is well-known \cite{Ga} that a generalized complex structure of type 2 is equivalent to a $K^J$ and a generalized complex structure of type 0 is equivalent to a $K^\om$. So in this paper, we deal mostly with generalized complex structures of type 1.
	 
Let us state our main results now.

\begin{enumerate}
	\item The classification of real four dimensional Lie algebras was obtained by M. Mubarakzyanov \cite{M}.  We use the notations of \cite{patera} and present in two tables  the list of  all four dimensional real Lie algebras.\footnote{Actually, there is a redundancy in the lists of Mubarakzyanov and Patera: the Lie algebras $A_{4,5}^{-1,-1}$ and $A_{4,5}^{-1,1}$ are isomorphic, so we drop the first one from our list.} Table \ref{1} contains the nonunimodular ones and Table \ref{2} the unimodular ones. For each Lie algebra in these tables  we precise whether or not  it admits a generalized complex structure of type 0,1 or 2. Since our classification completes the classification of invariant complex and symplectic structures given in \cite{O1, O2, Snow}, for each Lie algebra in our lists we give also the corresponding Lie algebra of the list used in \cite{O1, O2}. It is important here to mention that our study, combined with the results in \cite{O1, O2, Snow} show
	that the  four dimensional Lie algebras which have no generalized complex structure are: $A_{3,2}\oplus A_1$, $A_{3,5}^\al\oplus A_1$ with $0<|\al|<1$, $A_{4,4}$, $A_{4,2}^{\al}$ with $|\al|\not=1$ and  $A_{4,5}^{\al,\be}$ with  $-1<\al<\be<1$ and $\al+\be\not=0$. Moreover,   $A_{3,7}^\al$ with $\al>0$ and $A_{4,5}^{1,1}$ don't have neither a complex structure nor a symplectic structure but carry a generalized complex structure of type 1.

	\item For each Lie algebra $\G$ in Table \ref{1} and Table \ref{2}, we give explicitly the classes of  generalized complex structures $(J,R,\sigma)$ of type 1 carried by $\G$ with their pure spinors (see Tables \ref{8} and \ref{9}).
	
	\item Generalized complex structures with $\si=0$ are in correspondence with holomorphic Poisson tensors. We show that any left invariant holomorphic Poisson tensor in a four dimensional Lie group is invertible and defines an holomorphic symplectic form. There are four simply connected  Lie groups of dimension four which carry a holomorphic Poisson tensor, namely, those associated to the Lie algebras $A_{3,1}\oplus A_1$, $A_{4,5}^{-1,1}$, $A_{4,9}^{-\frac12}$ and $A_{4,12}$   (See Theorem \ref{holom}).
	\item Recall that a generalized K\"ahler structure on a Lie algebra $\G$ is a couple of commuting generalized complex structures $(\mathcal{J}_1,\mathcal{J}_2)$ such that $G=\langle \mathcal{J}_1\mathcal{J}_2.,.\rangle$ is  positive definite. We show that in dimension four either $\mathcal{J}_1$ is of type 2 and $\mathcal{J}_2$ of type 0 and $(\mathcal{J}_1,\mathcal{J}_2)$ is equivalent to a classical K\"ahler structure or both $\mathcal{J}_1$ and  $\mathcal{J}_2$ have type 1 and we classify such couples (see Theorem \ref{type11}). It is worth mentioning that there are only  four  Lie algebras of dimension four which carry a non classical generalized K\"ahler structure, namely, $A_{3,6}\oplus A_1$, $A_2\oplus 2A_1$,  $2A_1$ and $A_{4,6}^{\al,0}$ with $\al\not=0$ (see Theorem \ref{type11}). 
	The author can consult \cite{GaK} for an introduction to generalized K\"ahler geometry.
	
	\item Among the main invariants of generalized complex structures one can mention their different cohomologies (see \cite{danielle, Ga}). For any class of generalized complex structure in Table 8 and 9, we compute explicitly its invariant  Dolbeault cohomology, its invariant Bott-Chern cohomology and its invariant Aeppli cohomology.

\end{enumerate}

\paragraph{Organization of the paper}  Section \ref{section2} contains the basic results which are essential to the classification.  Proposition \ref{step1},  Theorems \ref{main1}, \ref{main2} and \ref{main3} are the key steps in this classification. We give also the pure spinor associated to any generalized complex structure of type 1 in a four dimensional Lie algebra and the condition for this structure to be Calabi-Yau (See Propositions \ref{spinor}-\ref{cy}). 
Sections \ref{section3} is devoted to the classification of generalized complex structures of type 1. At the end of this section, we give the classes of left invariant holomorphic Poisson tensors on four dimensional  simply connected Lie groups. Section \ref{section6} is devoted to the classification of generalized K\"ahler structures in dimension 4. Section \ref{section4} is  devoted to the computation of the different cohomologies. Section \ref{section5} contains all the  tables and Section \ref{section7} is an Appendix containing the details of the computations needed in the proof of Theorem \ref{type1}.

\paragraph{ Notations } For  $\mathbb{B}=(e_1,\ldots,e_n)$ a basis of a real vector space $V$, we denote by $\mathbb{B}^*=(e^1,\ldots,e^n)$ it dual basis. The elements of the bases of $\wedge^k\G$ and $\wedge^k\G^*$ will be denoted by $e_{ij}=e_i\wedge e_j$,  $e^{ij}=e^i\wedge e^j$, $e^{ijk}=e^i\wedge e^j\wedge e^k$ etc.. and $E_{ij}$ is the endomorphism which sends $e_j$ to $e_i$ and vanishes on $e_k$ for $k\not=j$. For any $\om\in\wedge^2 V^*$, we denote by $\om_\#:V\too V^*$ the endomorphism, $u\mapsto i_u\om$ and for any $\pi\in\wedge^2 V$, we denote by $\pi^\#:V^*\too V$ given by $\prec\be,\pi^\#(\al)\succ=\pi(\al,\be)$. One must be careful that
\[ M(\pi,\mathbb{B}^*)=-M(\pi^\#,\mathbb{B}^*,\mathbb{B})\esp
M(\om,\mathbb{B})=-M(\om_\#,\mathbb{B},\mathbb{B}^*). \]
For instance, in dimension 2,
\[ M(E_{12},\mathbb{B})=\left(\begin{array}{cc}0&1\\0&0   \end{array} \right),
M(e_{12}^\#,\mathbb{B}^*,\mathbb{B})=\left(\begin{array}{cc}0&-1\\1&0   \end{array} \right)\esp M(e^{12}_\#,\mathbb{B},\mathbb{B}^*)=\left(\begin{array}{cc}0&-1\\1&0   \end{array} \right). \]
The task of classification  involved a huge amount of computations  performed with  the computation software Maple.

The authors wish to thank Oumaima Tibssirte for having read carefully the first version of the paper and pointing out many typing and style mistakes.

\section{Preliminaries}\label{section2}

In this section, we give all the preliminaries  needed  to reach the purpose of this paper.
We start by deriving some important properties of generalized complex structures of type 1 on four dimensional Lie algebras.

Let $\G$ be a four dimensional Lie algebra and $(J,R,\si)$  a generalized complex structure on $\G$ of type 1, i.e., $(J,R,\si)$ satisfy the properties $(C0)-(C4)$ of Proposition \ref{pr2} and $\dim\mathrm{Im\;}R=2$.

The property $C0$  doesn't involve the Lie algebra structure but it has a crucial consequence in dimension 4. Rigorously speaking,
since $J$ commutes with $R$, the vector subspace $\h=\mathrm{Im\;}R$ is invariant by $J$ and the formula $$\om(R(\be),R(\ga))=\prec\be, R(\ga)\succ,\quad\be,\ga\in\G^*$$ defines a symplectic form on $\h$ satisfying $\om(Ju,v)=\om(u,Jv)$ for any $u,v\in\h$. This combined with $\dim\h=2$ imply that the restriction $J_1$ of $J$ to $\h$ satisfies $J_1=\la\;\mathrm{Id}_\h$ with $\la\in\R$. 
 
 On the other hand, since $J$ commutes with $\si$, the vector subspace $\p=\si^{-1}(\h^0)$, where $\h^0$ is the annulator of $\h$, contains $\ker\si$ and is invariant by $J$.  From the relation $J^2+R\circ\si=-\mathrm{Id}_\G$ we deduce that the restriction $L$ of $J$ to $\p$ satisfies $L^2=-\mathrm{Id}_\p$ and hence $\h\cap\p=\{0\}$. But $\p\not=\{0\}$ otherwise $\ker\si=\{0\}$ which is impossible so $\dim\p=2$ and we get $\G=\h\oplus\p$. 
 
 The restrictions of $J$ to $\p$ and  of $J^*$ to  $\h^0$  induce on these vector spaces a  1-dimensional complex structure and   $\si_2=\si_{|\p}:\p\too\h^0$ preserves the complex structures. But $\h^0=\p^*$ and $\si_2$ is skew-symmetric so $\si_2=0$. Thus $\p=\ker\si$. Moreover, $\si_1=\si_{|\h}:\h\too\p^0$ is invertible and we have $\si_1=-(1+\la^2)R_1^{-1}=-(1+\la^2)\om$ where $R_1:\p^0\too\h$ is the restriction of $R$ to $\p^0$.

Next we deal with $C1-C4$. Note that $C1$ is equivalent to $\h$ being a Lie subalgebra of $\G$ and we have seen that  $\G=\h\oplus\p$. Furthermore,  let $(e_1,e_2)$ be an arbitrary basis of $\h$ and for any $e_3\in\p$ let $e_4=Je_3$ then  $(e_1,e_2,e_3,e_4)$ is a basis of $\G$ and
\[ Je_1=\la e_1, Je_2=\la e_2, Je_3=e_4, Je_4=-e_3, R=a e_{12}^\#\esp \si=a^{-1}(1+\la^2)e^{12}_\#, \]where $(e^1,e^2,e^3,e^4)$ is the dual basis. Write
\[ [e_1,e_2]=a_1e_1+a_2e_2,[e_3,e_4]=b_1e_1+b_2e_2+b_3e_3+b_4e_4, [u,v]=\phi_\h(u)(v)-\phi_\p(v)(u), \] where $u\in\h$, $v\in\p$, $\phi_\h(u)\in\mathrm{End}(\p)$ and $\phi_\h(v)\in\mathrm{End}(\h)$.

The relation $C2$ is equivalent to
\begin{equation}\label{c2} \prec\ga,[R(\be),Ju]-J[R(\be),u]\succ=\prec\be,[R(\ga),Ju]-J[R(\ga),u]\succ, \end{equation} for any $u\in\G$ and any $\ga,\be\in\G^*$.
Since $\h$ is a Lie subalgebra and  $J_{|\h}=\la\mathrm{Id}_\h$ this relation is obviously true for any $u\in\h$. The relation is also clearly true for $\ga,\be\in\h^0$. 
For $\ga\in\h^0$,  $\be\in\p^0$ and $u\in\p$ the relation \eqref{c2} gives
\[ \prec\ga,[\phi_\h(R(\be)),J](u)\succ=0. \]
Thus, for any $v\in\h$, $\phi_\h(v)$ preserves the complex structure of $\p$ induced by $J$.
For $\ga\in\p^0$,  $\be\in\p^0$ and $u\in\p$ the relation \eqref{c2} gives
\[ R\circ\phi_\p^*(Ju)+\phi_\p(Ju)\circ R=\la\left( R\circ\phi_\p^*(u)+\phi_\p(u)\circ R \right).\]Since $J^2u=-u$ for any $u\in\p$ this relation is equivalent to
\[ R\circ\phi_\p^*(u)+\phi_\p(u)\circ R=0 \]which is equivalent to $\tr(\phi_\p^*(u))=0$ for any $u\in\p$.

We consider now $C3$,
\[N_{J}(u,v)={R}(i_{u\wedge v}\mathrm{d}\si^\flat).\]
This relation is obviously true if $u,v\in\h$ since $N_J(u,v)=0$ and $\mathrm{d}\si^\flat_{|h}=0$. For $u\in\h,v\in\p$, having the relation $\phi_\h(u)\circ J=J\circ\phi_\h(u)$, we get
\begin{align*}
N_J(u,v)&=[Ju,Jv]-J[u,Jv]-J[Ju,v]+J^2[u,v]\\
&=\la\phi_\h(u)(Jv)-\la \phi_\p(Jv)(u)-J\phi_\h(u)(Jv)+J\phi_\p(Jv)(u)
-\la J\phi_\h(u)(v)+\la J\phi_\p(v)(u)+J^2\phi_\h(u)(v)-J^2\phi_\p(v)(u)\\
&=\la J\phi_\h(u)(v)-\la \phi_\p(Jv)(u)+\phi_\h(u)(v)+\la \phi_\p(Jv)(u)-
\la J\phi_\h(u)(v)+\la^2 \phi_\p(v)(u)-\phi_\h(u)(v)-\la^2\phi_\p(v)(u)\\
&=0.
\end{align*}On the other hand, for any $\be\in\G^*$,
\[ \mathrm{d}\si^\flat(u,v,R(\be))=\si^\flat(\phi_\p(v)(u),R(\be))-\si^\flat(\phi_\p(v)(R(\be)),u)=0 \]since $\tr(\phi_\p(v))=0$. So $C3$ is satisfied for $u\in\h$ and $v\in\p$. Now, a direct computation gives
\[ N_J(e_3,e_4)=(1+\la^2)(a_1e_1+a_2e_2). \]
For any $\be\in\p^0$, we have
\[ \mathrm{d}\si^\flat(e_3,e_4,R(\be))=-\si^\flat([e_3,e_4],R(\ga))=-\si^\flat(a_1e_1+a_2e_2,R(\ga)). \]
So $C3$ holds. Let us show now that $C4$ holds also.
Note first that $\si_J=\la\si^\flat$ and we have seen that $\mathrm{d}\si^\flat(u,v,w)=0$ whenever $u,v\in\h$. Now it is obvious that $C4$ is true for $(u,e_3,e_4)$ for any $u\in\h$. So far, we have shown the following result. 

\begin{pr}\label{step1} Let $(\G,J,R,\si)$ be a four dimensional Lie algebra  endowed with a generalized complex structure of type 1. Put $\h=\mathrm{Im\;}R$ and $\p=\ker\si$. Then $\G=\h\oplus\p$ and for any basis $(e_1,e_2)$ of $\h$ and any $e_3\in\p$, if $e_4=Je_3$ then the following assertions hold:
	\begin{enumerate}
		\item $J=\la(E_{11}+E_{22})+E_{34}-E_{43} ,\; R=a e_{12}^\#\esp\si=a^{-1}(1+\la^2)e^{12}_\#,\;\la\in\R,a\not=0.$
		 
		\item The Lie brackets are given by
		\[ \begin{cases}[e_1,e_2]=a_1e_1+a_2e_2,\;[e_3,e_4]=b_1e_1+b_2e_2+b_3e_3+b_4e_4,\\
		\;[e_1,e_3]=x_1e_3-y_1e_4-p_1e_1-r_1e_2,\;[e_1,e_4]=y_1e_3+x_1e_4-p_2e_1-r_2e_2,\\
\;		[e_2,e_3]=x_2e_3-y_2e_4-q_1e_1+p_1e_2,\;[e_2,e_4]=y_2e_3+x_2e_4-q_2e_1+p_2e_2.
\end{cases} \]
\item The Jacobi identity is equivalent to

\begin{equation}\tag{S}\label{S} \begin{cases}
a_{{1}}x_{{1}}+a_{{2}}x_{{2}}=0,\\
a_{{1}}y_{{1}}+a_{{2}}y_{{2}}=0,\\
a_{{1}}p_{{1}}+a_{{2}}q_{{1}}-p_{{1}}x_{{2}}+p_{{2}}y_{{2}}+q_{{1}}x_{
	{1}}-q_{{2}}y_{{1}}=0,\\
a_{{1}}r_{{1}}-a_{{2}}p_{{1}}-p_{{1}}x_{{1}}+p_{{2}}y_{{1}}-r_{{1}}x_{
	{2}}+r_{{2}}y_{{2}}=0,\\
a_{{1}}p_{{2}}+a_{{2}}q_{{2}}-p_{{1}}y_{{2}}-p_{{2}}x_{{2}}+q_{{1}}y_{
	{1}}+q_{{2}}x_{{1}}=0,\\
a_{{1}}r_{{2}}-a_{{2}}p_{{2}}-p_{{1}}y_{{1}}-p_{{2}}x_{{1}}-r_{{1}}y_{
	{2}}-r_{{2}}x_{{2}}=0,\\
a_{{1}}b_{{2}}-2\,b_{{1}}x_{{1}}-b_{{3}}p_{{1}}-b_{{4}}p_{{2}}+q_{{1}}
r_{{2}}-q_{{2}}r_{{1}}=0,\end{cases}\begin{cases}
a_{{2}}b_{{2}}-2\,b_{{2}}x_{{1}}-b_{{3}}r_{{1}}-b_{{4}}r_{{2}}-2\,p_{{
		1}}r_{{2}}+2\,p_{{2}}r_{{1}}=0,\\
-b_{{3}}x_{{1}}+b_{{4}}y_{{1}}+p_{{1}}y_{{1}}-p_{{2}}x_{{1}}+r_{{1}}y_
{{2}}-r_{{2}}x_{{2}}=0,\\
-b_{{3}}y_{{1}}-b_{{4}}x_{{1}}+p_{{1}}x_{{1}}+p_{{2}}y_{{1}}+r_{{1}}x_
{{2}}+r_{{2}}y_{{2}}=0,\\
-a_{{1}}b_{{1}}-2\,b_{{1}}x_{{2}}-b_{{3}}q_{{1}}-b_{{4}}q_{{2}}+2\,p_{
	{1}}q_{{2}}-2\,p_{{2}}q_{{1}}=0,\\
-a_{{2}}b_{{1}}-2\,b_{{2}}x_{{2}}+b_{{3}}p_{{1}}+b_{{4}}p_{{2}}-q_{{1}
}r_{{2}}+q_{{2}}r_{{1}}=0,\\
-b_{{3}}x_{{2}}+b_{{4}}y_{{2}}-p_{{1}}y_{{2}}+p_{{2}}x_{{2}}+q_{{1}}y_
{{1}}-q_{{2}}x_{{1}}=0,\\
-b_{{3}}y_{{2}}-b_{{4}}x_{{2}}-p_{{1}}x_{{2}}-p_{{2}}y_{{2}}+q_{{1}}x_
{{1}}+q_{{2}}y_{{1}}=0.
\end{cases} \end{equation}

		\end{enumerate}
	
	\end{pr}
	
	\begin{remark}\label{rem1} The Lie algebra $\G$ in the last proposition is unimodular if and only if 
		\begin{equation}\label{uni} a_{{2}}+2\,x_{{1}}=a_{{1}}-2\,x_{{2}}=b_{{4}}=b_{{3}}=0.
		\end{equation} Moreover, the matrix in the basis $(e_1,e_2)$ of the restriction of the Killing form ${Q}$ of $\G$ to $\h$ is given by
		\begin{equation}\label{K} Q_\h=\left( \begin {array}{cc} {a_{{2}}}^{2}+2\,{x_{{1}}}^{2}-2\,{y_{{1}}}
		^{2}&-a_{{1}}a_{{2}}+2\,x_{{1}}x_{{2}}-2\,y_{{1}}y_{{2}}
		\\ \noalign{\medskip}-a_{{1}}a_{{2}}+2\,x_{{1}}x_{{2}}-2\,y_{{1}}y_{{2
			}}&{a_{{1}}}^{2}+2\,{x_{{2}}}^{2}-2\,{y_{{2}}}^{2}\end {array}
			\right) .
			 \end{equation}This formula will play a key role in the proof of Theorem \ref{main3}.
		
		\end{remark}
		
Let us give now the pure spinor associated to a generalized complex structure $(\G,J,R,\si)$ of type 1 on a four dimensional Lie algebra. Consider $\mathrm{L}$  the maximal isotropic subspace of $\Phi(\G)\otimes\mathbb{C}$ associated to the generalized complex structure   and given by
\[ \mathrm{L}=\left\{ X+\xi-i(JX+R(\xi)+\si(X)-J^*\xi),X\in\G,\xi\in\G^*     \right\}. \] According to \cite{Ga}, there exists a spinor $\rho\in\Om^\bullet(\G)\otimes\mathbb{C}$ such that
\[\mathrm{L}=\mathrm{Ann}(\rho)=\left\{ u\in\Phi(\G)\otimes\mathbb{C}, u.\rho=0    \right\} \]	where $(X+\xi).\rho=i_X\rho+\xi\wedge\rho$.

\begin{pr}\label{spinor} Let $(\G,J,R,\si)$ be a generalized complex structure of type 1 on a four dimensional Lie algebra. Then the associated pure spinor is given by
\[ \rho=\exp({(i-\la)\om})(\theta+iJ^*\theta)=(\theta+iJ^*\theta)+(i-\la)\om\wedge(\theta+iJ^*\theta), \]where $\theta\in\h^0\setminus\{0\}$ and $\om\in\wedge^2\G^*$ is given by
\[ \om(R(\xi_1),R(\xi_2))=\prec \xi_1,R(\xi_2)\succ,\; i_X\om=0\quad\mbox{for}\quad X\in\ker\si\esp J_{|\h}=\la\;\mathrm{Id}_\h. \]
\end{pr}

\begin{proof} By using the splitting $\G=\h\oplus\p$ with $\p=\ker\si$, we will check that for any $u\in \mathrm{L}$, $u.\rho=0$. This shows that $\mathrm{L}\subset\mathrm{Ann}(\rho)$ and since $\mathrm{L}$ is maximal we have the equality. Note first that for any $X\in\G$,
	\[ i_{X-iJX}(\theta+iJ^*\theta)=0\esp \si(X)\wedge\om=0. \]
	This with the fact that $i_X\om=0$ for $X\in\p$ implies that, for any $X\in\p$,
	\[ [X-i(JX+\si(X))].\rho=0. \]
	Now take $X=R(\xi)\in\h$ with $\xi\in\p^0$. Then
	\[ [X-i(JX+\si(X))].\rho=-i\si(X)\wedge(\theta+iJ^*\theta)+(\la-i)i_{X-iJX}\om\wedge (\theta+iJ^*\theta).  \]
	But
	\[ \si(X)=\si\circ R(\xi)=-\xi-(J^*)^2\xi=-(1+\la^2)\xi\esp i_{X-iJX}\om=(1-i\la)i_{R(\xi)}\om=(1-i\la)\xi \]
	and hence $[X-i(JX+\si(X))].\rho=0$. On the other hand, for any $\xi\in\h^0$, we have $$(\xi+iJ^*\xi)\wedge(\theta+iJ^*\theta)=0.$$ One can check this by taking $\xi=\theta$ or $\xi=J^*\theta$. We deduce that $(\xi+iJ^*\xi).\rho=0$.
	
	Let $\xi\in\p^0$. We have $i_{R(\xi)}(\theta+iJ^*\theta)=0$,  $\xi\wedge\om=0$ and $J^*\xi=\la\xi$ hence
	\[ \xi-i(R(\xi)-J^*\xi).\rho=(1+i\la)\xi\wedge(\theta+iJ^*\theta)-i(\la-i)i_{R(\xi)}\om\wedge (\theta+iJ^*\theta)=0\]since $i_{R(\xi)}\om=\xi$. This completes the proof.
	\end{proof}

By using Proposition \ref{step1}, we will prove the following result.
\begin{pr}\label{cy}Let $(\G,J,R,\si)$ be a generalized complex structure of type 1 on a four dimensional Lie algebra. Then it is Calabi-Yau, i.e., $\mathrm{d}\rho=0$ if and only if $[\G,\G]\subset\mathrm{Im\;}R$.
	
	\end{pr}

	\begin{proof} Choose a basis $(e_1,e_2)$ of $\mathrm{Im\;}R$ and $(e_3,e_4=Je_3)$ a basis of $\ker\si$. Then
		\[ \rho=e^3+ie^4+(i-\la) e^{12}\wedge(e^3+ie^4). \]
		So
		\[ \rho=e^3-\la e^{123}-e^{124}+i( e^{4}+e^{123}
		-\la e^{124}). \]
		 The relation $\mathrm{d}\rho=0$
		is equivalent to
		\begin{equation*} \begin{cases}
		\mathrm{d}e^3=\mathrm{d}e^4=0,\\
		\mathrm{d}(e^{12}\wedge (\la e^3+e^4))=0,\\
		\mathrm{d}(e^{12}\wedge (e^3-\la e^4))=0.
		\end{cases}\label{laska} \end{equation*}But, from Proposition \ref{step1},
		\[ \begin{cases}\mathrm{d}e^3=-b_3e^{34}-x_1e^{13}-y_1e^{14}-x_2e^{23} -y_2e^{24} ,\\
		\mathrm{d}e^4=-b_4e^{34}+y_1e^{13}-x_1e^{14}+y_2e^{23}-x_2e^{24},\\
		\mathrm{d}e^1=-a_1e^{12}-b_1e^{34}+p_1e^{13}+p_2e^{14}+q_1e^{23}+q_2e^{24},\\
		\mathrm{d}e^2=-a_2e^{12}-b_2e^{34}+r_1e^{13}+r_2e^{14}-p_1e^{23}-p_2e^{24}.
		\end{cases} \]On the other hand,
		\begin{eqnarray*}
			\mathrm{d}(e^{12})&=&-b_1e^{234}
			+b_2e^{134},\;
			\mathrm{d}(e^{123})=-b_3e^{1234},\;
			\mathrm{d}(e^{124})=-b_4e^{1234}.
		\end{eqnarray*}In conclusion, $\mathrm{d}\rho=0$ if and only if $b_3=b_4=x_1=x_2=y_1=y_2=0$ which completes the proof.
		\end{proof}

The next step is to solve the system \eqref{S}. We distinguish three cases: $(i)$ $\G$ is unimodular, $(ii)$ $\h$ is non abelian and $\G$ is nonunimodular and $(iii)$ $\h$ is  abelian and $\G$ is nonunimodular.

\begin{theo}\label{main1} Let $(\g,J,R,\sigma)$ be a generalized complex structure of type 1 on  a four dimensional unimodular  Lie algebra.   Then there exists a basis $(e_1,e_2,e_3,e_4)$ of $\g$ such that
	\[ J=\la(E_{11}+E_{22})+E_{34}-E_{43} ,\; R= e_{12}^\#\esp\si=(1+\la^2)e^{12}_\#, \]
	and the non vanishing Lie brackets have one of the following forms:
	\begin{enumerate}
		\item[$\U_1:$]    $			[e_1,e_2]=e_1,	
		[e_2,e_3]=\frac12e_3-ye_4-q_1e_1,\;[e_2,e_4]=ye_3+\frac12e_4-q_2e_1.$
		
		\item[$\U_2:$] 	$[e_3,e_4]=b_1e_1+b_2e_2,
		[e_2,e_3]=- ye_4-q_1e_1,\;[e_2,e_4]= ye_3-q_2e_1.$

		\item[$\U_3:$]  $[e_3,e_4]=b_1e_1+b_2e_2,
		\;[e_4,e_1]=pe_1+re_2,
		\;[e_4,e_2]=qe_1-pe_2,	$ with $|p^2+qr|\in\{0,1\}.$ 	
	\end{enumerate}

\end{theo}

\begin{proof} We are in the situation of Proposition \ref{step1}. 
	We distinguish two cases: $\h$ is non abelian and $\h$ is abelian.

	Suppose that $\h$ is non abelian. We can choose $(e_1,e_2)$ such that $a_1=1$ and $a_2=0$. Moreover, since $\G$ is unimodular, according to \eqref{uni},
	 $x_1=0=b_3=b_4=0$, $x_2=\frac12$. So, the system \eqref{S} is equivalent to
	\begin{equation*} \begin{cases}y_1=0,\\
	2p_{{2}}y_{{2}}+\,p_{{1}}=0,\\
	2r_{{2}}y_{{2}}+\,r_{{1}}=0,\\
	-2p_{{1}}y_{{2}}+\,p_{{2}}=0,\\
	-2r_{{1}}y_{{2}}+r_{{2}}=0,
	\end{cases}\esp \begin{cases}
	q_{{1}}r_{{2}}-q_{{2}}r_{{1}}+b_{{2}}=0,\\
	-\,p_{{1}}r_{{2}}+\,p_{{2}}r_{{1}}=0,\\
	\,p_{{1}}q_{{2}}-\,p_{{2}}q_{{1}}-\,b_{{1}}=0,\\
	-q_{{1}}r_{{2}}+q_{{2}}r_{{1}}-b_{{2}}=0.
	\end{cases} \end{equation*}
	This is equivalent to $y_1=p_1=p_2=r_1=r_2=b_1=b_2=0$ and $\g$ is isomorphic to $\U_1$.
	
 Suppose now that $\h$ is abelian.	
		Then $b_3=b_4=x_1=x_2=0$. In this case, the system \eqref{S} is equivalent to
	\begin{equation}\label{ju} \begin{cases}
	p_2y_2-q_2y_1= p_2y_1+r_2y_2=0,\\q_1y_1 -p_1y_2=r_1y_2 +p_1y_1=0,\\          q_1r_2-q_2r_1=r_2p_1-p_2r_1=  q_2p_1-p_2q_1= 0.  
	\end{cases} \end{equation}

We distinguish two cases: 
\begin{itemize}\item[$(i)$] $(y_1,y_2)=0$. Then
	\[ q_1r_2-q_2r_1=r_2p_1-p_2r_1=  q_2p_1-p_2q_1= 0. \]
We have three cases:

$\bullet$ $(q_1,q_2)=(p_1,p_2)=0$. In this case, $\g$ is isomorphic to $\U_2$.

$\bullet$ $(q_1,q_2)\not=(0,0)$ and hence $p_1=pq_1$, $p_2=pq_2$, $r_1=rq_1$ and $r_2=rq_2$.

\[ \begin{cases}
[e_1,e_2]=0,\;[e_3,e_4]=b_1e_1+b_2e_2,\\
[e_1,e_3]=-pq_1e_1-rq_1e_2,\;[e_1,e_4]=-pq_2e_1-rq_2e_2,\\
[e_2,e_3]=-q_1e_1+pq_1e_2,\;[e_2,e_4]=-q_2e_1+pq_2e_2.
\end{cases}
\]Put $f_3=q_2e_3-q_1e_4$ and $f_4=Jf_3=q_1e_3+q_2e_4$. Then
\[ [e_1,f_3]=[e_2,f_3]=0, [e_1,f_4]=-(q_1^2+q_2^2)(pe_1+re_2)\esp [e_2,f_4]=-(q_1^2+q_2^2)(e_1-pe_2).  \]
If the determinant $\de$ of $([e_1,f_4],[e_2,f_4])$ in the basis $(e_1,e_2)$ is non zero then we can replace $(f_3,f_4)$ with $\frac{1}{\sqrt{\de}}(f_3,f_4)$ and we get that
the Lie algebra is  isomorphic to $\U_3$.

$\bullet$ $(p_1,p_2)\not=(0,0)$ and hence $q_1=pp_1$, $q_2=pp_2$, $r_1=rp_1$ and $r_2=rp_2$. This case is similar to the last case.

\item[$(ii)$] {  $(y_1,y_2)\not=0$}. Without loss of generality we can suppose $y_1\not=0$. From \eqref{ju}, we deduce that
\[ (p_1,q_1)=(my_1,my_2),(p_1,r_1)=(ny_2,-ny_1),(p_2,q_2)=(sy_1,sy_2)\esp (p_2,r_2)=(ty_2,-ty_1). \]
	So $m=\frac{ny_2}{y_1}$ and $s=\frac{ty_2}{y_1}$ and one can check that the last equation in \eqref{ju} are satisfied. So
	\[ \begin{cases}
	[e_1,e_2]=0,\;[e_3,e_4]=b_1e_1+b_2e_2,\\
	[e_1,e_3]=-y_1e_4-ny_2e_1+ny_1e_2,\;[e_1,e_4]=y_1e_3-ty_2e_1+ty_1e_2,\\
	[e_2,e_3]=-y_2e_4-\frac{ny_2^2}{y_1}e_1+\frac{ny_2^2}{y_1}e_2,\;
	[e_2,e_4]=y_2e_3-\frac{ty_2^2}{y_1}e_1+\frac{ty_2^2}{y_1}e_2.
	\end{cases}
	\]
	
	If we put $f_2=y_1e_2-y_2e_1$, we get
	\[ [f_2,e_3]=[f_2,e_4]=0, [e_1,e_3]=-y_1e_4+nf_2\esp [e_1,e_4]=y_1e_3+tf_2 \]
	and the Lie algebra is isomorphic to $\U_2$.
	\end{itemize}To conclude, one can make a change of basis of $\h$ without changing the general form of the Lie brackets in order to get $R=e_{12}^\#$.
\end{proof}

In Table \ref{10},  for each class of Lie algebras $\U_i$, $i=1,2,3$ obtained in Theorem \ref{main1}, we build a family of isomorphisms (depending on the values of the parameters) from $\U_i$ onto  an unimodular four dimensional Lie algebra in  Table \ref{2}.

 Let us pursue our study of generalized complex structures of type 1.
\begin{theo}\label{main2}Let $(\g,J,R,\sigma)$ be a generalized complex structure of type 1 on  a four dimensional nonunimodular  Lie algebra where $\mathrm{Im\;}R$  is non abelian.
	  Then there exists a basis $(e_1,e_2,e_3,e_4)$ of $\g$ such that
	\[ J=\la(E_{11}+E_{22})+E_{34}-E_{43} ,\; R= e_{12}^\#\esp\si=(1+\la^2)e^{12}_\#, \]
	and the non vanishing Lie brackets have one of the following forms:
	\begin{enumerate}
		\item[$\B_{1}:$] $[e_1,e_2]=e_1,\; [e_1,e_4]=-e_1,\;{[e_2,e_3]}=q_1e_1+e_3,\; [e_2,e_4]=q_2e_1+e_2+e_4,\;{[e_3,e_4]}=q_1e_1+e_3  $
		
		\item[$\B_{2}:$]   $ [e_1,e_2]=e_1,\;
		{[e_2,e_3]=q_1e_1},\; [e_2,e_4]=q_2e_1,\;
		{[e_3,e_4]}=q_1e_1+e_3 $

		\item[$\B_{3}:$]   $[e_1,e_2]=e_1,\; [e_2,e_3]=-q_1e_1+xe_3-ye_4,\; [e_2,e_4]=-q_2e_1+ye_3+xe_4\quad x\not=\frac{1}{2} $
		
		\item[$\B_{4}:$] $ [e_1,e_2]=e_1,\;
		{[e_2,e_3]}=-q_1e_1-\frac12e_3-ye_4,\; [e_2,e_4]=-q_2e_1+ye_3-\frac12e_4,\; {[e_3,e_4]}=e_1,\;
		$
	\end{enumerate}

\end{theo}

\begin{proof} We are in the situation of Proposition \ref{step1}. We can choose $(e_1,e_2)$ such that $a_1=1$ and $a_2=0$.
	 Moreover, from \eqref{S} we get that $x_1=y_1=0$ and
	\[ \begin{cases}
	(1-x_{{2}})p_{{1}}+p_{{2}}y_{{2}}=0,\\
	-p_{{1}}y_{{2}}+p_{{2}}(1-x_2)=0,\\
	p_{{1}}x_{{2}}+(p_{{2}}+b_3)y_{{2}}=0,\\
	(p_2-b_{{3}})x_{{2}}-p_{{1}}y_{{2}}=0,\\
	-r_{{1}}x_{{2}}+r_{{2}}y_{{2}}+r_{{1}}=0,\\
	r_{{1}}x_{{2}}+r_{{2}}y_{{2}}=0,\\
	-r_{{1}}y_{{2}}-r_{{2}}x_{{2}}+r_{{2}}=0,\\
	r_{{1}}y_{{2}}-r_{{2}}x_{{2}}=0,
	\end{cases}\esp \begin{cases}
	-b_{{3}}p_{{1}}+q_{{1}}r_{{2}}-q_{{2}}r_{{1}}+b_{{2}}=0,\\
	-b_{{3}}r_{{1}}-2\,p_{{1}}r_{{2}}+2\,p_{{2}}r_{{1}}=0,\\
	-2\,b_{{1}}x_{{2}}-b_{{3}}q_{{1}}+2\,p_{{1}}q_{{2}}-2\,p_{{2}}q_{{1}}-
	b_{{1}}=0,\\
	-2\,b_{{2}}x_{{2}}+b_{{3}}p_{{1}}-q_{{1}}r_{{2}}+q_{{2}}r_{{1}}=0.
	\end{cases} \]The system
	\[ 
	r_{{1}}(1-x_{{2}})+r_{{2}}y_{{2}}=
	-r_{{1}}y_{{2}}+r_{{2}}(1-x_2)=
	r_{{1}}x_{{2}}+r_{{2}}y_{{2}}=
	-r_{{2}}x_{{2}}+r_{{1}}y_{{2}}=0,
	 \]Implies that $r_1=r_2=0$. Then our system becomes
	\begin{equation}\label{jnu} \begin{cases}
	(1-x_{{2}})p_{{1}}+p_{{2}}y_{{2}}=0,\\
	-p_{{1}}y_{{2}}+p_{{2}}(1-x_2)=0,\\
	p_{{1}}x_{{2}}+(p_{{2}}+b_3)y_{{2}}=0,\\
	(p_2-b_{{3}})x_{{2}}-p_{{1}}y_{{2}}=0,\\
	\end{cases}\esp \begin{cases}
	-b_{{3}}p_{{1}}+b_{{2}}=0,\\
	-2\,b_{{1}}x_{{2}}-b_{{3}}q_{{1}}+2\,p_{{1}}q_{{2}}-2\,p_{{2}}q_{{1}}-
	b_{{1}}=0,\\
	-2\,b_{{2}}x_{{2}}+b_{{3}}p_{{1}}=0.
	\end{cases} \end{equation}
	We distinguish three cases:
	
	\begin{itemize}
		\item[$(i)$] $(p_1,p_2)=(0,0)$ and $b_3\not=0$. Then the system \eqref{jnu} is equivalent to
		\[ x_2=y_2=b_2=0\esp b_1=-b_3q_1. \]Then
			 the Lie algebra is isomorphic to  $\B_2$.
			\item[$(ii)$] $(p_1,p_2)=(0,0)$ and $b_3=0$. Then the system \eqref{jnu} is equivalent to $b_2=0$ and $b_1(2x_2+1)=0$. In this case $\g$ is isomorphic either to $\B_3$ or  $\B_4$.
			
		\item[$(iii)$] $(p_1,p_2)\not=(0,0)$. Then  \eqref{jnu} is equivalent to 
		\[ y_2=0, x_2=1, p_1=0, p_2=b_3, b_2=0, b_1=-p_2q_1 \]	and the Lie algebra is isomorphic to  $\B_1$.
	\end{itemize}To conclude, one can make a change of basis of $\h$ without changing the general form of the Lie brackets in order to get $R=e_{12}^\#$.
	\end{proof}

In Table \ref{12},  for each class of Lie algebras $\B_i$, $i=1,2,3,4$ obtained in Theorem \ref{main2}, we build a family of isomorphisms (depending on the values of the parameters) from $\B_i$ onto  a nonunimodular four dimensional Lie algebra in  Table \ref{1}.

\begin{theo}\label{main3} Let $(\g,J,R,\sigma)$ be a generalized complex structure of type 1 on  a four dimensional nonunimodular  Lie algebra where $\mathrm{Im\;}R$  is  abelian. Then there exists a basis $(e_1,e_2,e_3,e_4)$ of $\g$ such that
	\[ J=\la(E_{11}+E_{22})+E_{34}-E_{43} ,\; R= e_{12}^\#\esp\si=(1+\la^2)e^{12}_\#, \]
	and the non vanishing Lie brackets have one of the following forms:
	\begin{enumerate}
		\item[$\A_1:$]
		$
		[e_1,e_3]=e_3-y_1e_4,\;[e_1,e_4]=y_1e_3+e_4,
		[e_2,e_3]=-y_2e_4,\;[e_2,e_4]=y_2e_3$.

		\item[$\A_2:$] $ [e_1,e_3]=xe_3-cxe_4-ae_2,\; [e_1,e_4]=cxe_3+xe_4-be_2, (a,b)\not=(0,0),\; x\not=0.$

		\item[$\A_3:$] $[e_3,e_4]=b_1e_1+b_2e_2+b_3e_3,
		\;[e_1,e_4]=pe_1+re_2,
		\;[e_2,e_4]=qe_1-pe_2,\; b_3\not=0.
		$
		
		\item[$\A_4:$] $[e_3,e_4]=b_1e_1+b_2e_2-2e_3,\;[e_1,e_4]=-e_1,
		[e_2,e_3]=-q_1e_1,\;[e_2,e_4]=-q_2e_1+e_2,\;q_1\not=0.
		$

		\item[$\A_5:$] $[e_3,e_4]=-{p^2}e_1+{p^2}e_2+2pe_3,
		\;[e_1,e_3]=e_3-e_4-{p}e_1,\;[e_1,e_4]=e_3+e_4+pe_2,
		\;		[e_2,e_3]=(e_3+e_4+pe_2),\;[e_2,e_4]=-(e_3-e_4
		-{p}e_1),\; p\not=0.$
		
	\end{enumerate}

\end{theo}

\begin{proof} We are in the situation of Proposition \ref{step1} with $\h$ is abelian and $\G$ is nonunimodular.

	{We suppose $b_3=b_4=0$.} Since $\G$ is nonunimodular then, according to \eqref{uni}, $(x_1,x_2)\not=(0,0)$. The system \eqref{S}  can be written
	
	\begin{equation}\label{s1} \begin{cases}
	-{\it p_1}\,{\it x_2}+{\it p_2}\,{\it y_2}+{
		\it q_1}\,{\it x_1}-{\it q_2}\,{\it y_1}=0,\\
	-{\it p_1}\,{\it x_2}-{
		\it p_2}\,{\it y_2}+{\it q_1}\,{\it x_1}+{\it q_2}\,{\it y_1}=0,\\
	-{\it p_1}\,{\it x_1}+{\it p_2}\,{
		\it y_1}-{\it r_1}\,{\it x_2}+{\it r_2}\,{\it y_2}=0,\\
	{\it p_1}\,{\it x_1}+{\it p_2}\,{
		\it y_1}+{\it r_1}\,{\it x_2}+{\it r_2}\,{\it y_2}=0,\\
	-{\it p_1}\,{\it y_2}-{\it p_2}\,{\it x_2}+{
		\it q_1}\,{\it y_1}+{\it q_2}\,{\it x_1}=0,\\
	-{
		\it p_1}\,{\it y_2}+{\it p_2}\,{\it x_2}+{\it q_1}\,{\it y_1}-{\it q_2}\,{
		\it x_1}=0,\\
	-{\it p_1}\,{\it y_1}-{\it p_2}\,{
		\it x_1}-{\it r_1}\,{\it y_2}-{\it r_2}\,{\it x_2}=0,
	\\
	{\it p_1}
	\,{\it y_1}-{\it p_2}\,{\it x_1}+{\it r_1}\,{\it y_2}-{\it r_2}\,{\it x_2}=0,\end{cases}\esp\begin{cases}
	-2\,{\it b_1}\,{\it x_1}-{\it q_2}\,{\it r_1}+{\it q_1}\,{\it r_2}=0,\\
	-2\,{\it b_2}
	\,{\it x_1}-2\,{\it p_1}\,{\it r_2}
	+2\,{\it p_2}\,{\it r_1}=0,\\
	-2\,{\it x_2}\,{\it b_1}+2\,{\it q_2}\,{\it p_1}-2\,{\it p_2}\,{\it q_1}=0,\\
	-2\,{
		\it x_2}\,{\it b_2}+{\it q_2}\,{
		\it r_1}-{\it q_1}\,{\it r_2}=0.
	\end{cases} \end{equation}
	
	The first set of equations is equivalent
	
	\[ \begin{cases}
	{\it q_1}\,{\it x_1}-{\it p_1}\,{\it x_2}=0,\\
	{\it p_1}\,{\it x_1}+{\it r_1}\,{\it x_2}=0,\\
	{\it q_2}\,{
		\it x_1}-{\it p_2}\,{\it x_2}=0,\\
	{\it p_2}\,{
		\it x_1}+{\it r_2}\,{\it x_2}=0,
	\end{cases}\esp
	\begin{cases}
	{\it q_2}\,{\it y_1}-{\it p_2}\,{\it y_2}=0,\\
	{\it p_2}\,{
		\it y_1}+{\it r_2}\,{\it y_2}=0,\\
	{\it q_1}\,{\it y_1}-{\it p_1}\,{\it y_2}
	=0,\\
	{\it p_1}
	\,{\it y_1}+{\it r_1}\,{\it y_2}=0.\end{cases}\]
	Since $e_1$ and $e_2$ are interchangeable, we can suppose that $x_1\not=0.$ So we get
	\[ p_1=-ax_2, q_1=-a\frac{x_2^2}{x_1}, r_1=ax_1 \]and
	\[ p_2=-bx_2, q_2=-b\frac{x_2^2}{x_1}, r_2=bx_1\esp a(x_1y_2-x_2y_1)=b(x_1y_2-x_2y_1)=0. \]If we replace in
	\[ \begin{cases}
	-2\,{\it b_1}\,{\it x_1}-{\it q_2}\,{\it r_1}+{\it q_1}\,{\it r_2}=0,\\
	-2\,{\it b_2}
	\,{\it x_1}-2\,{\it p_1}\,{\it r_2}
	+2\,{\it p_2}\,{\it r_1}=0,\\
	-2\,{\it x_2}\,{\it b_1}+2\,{\it q_2}\,{\it p_1}-2\,{\it p_2}\,{\it q_1}=0,\\
	-2\,{
		\it x_2}\,{\it b_2}+{\it q_2}\,{
		\it r_1}-{\it q_1}\,{\it r_2}=0.
	\end{cases} \]we get
	$b_1=b_2=0$.
	
	We have two cases
	
	$\bullet$ $a=b=0$. If take change $e_1$ to $\frac1{x_1}e_1$ and $e_2$ to $e_{2}-\frac{x_2}{x_1}e_1$ we get that  $\G$ is isomorphic to $\A_1$.
	
	$\bullet$ $(a,b)\not=(0,0)$ then $y_1=cx_1$ and $y_2=cx_2$ and
	\[  {\begin{cases}
		[e_1,e_2]=0,\;[e_3,e_4]=0,\\
		[e_1,e_3]=x_1e_3-cx_1e_4+ax_2e_1-ax_1e_2,\;
		[e_1,e_4]=cx_1e_3+x_1e_4+bx_2e_1-bx_1e_2,\\
		[e_2,e_3]=x_2e_3-cx_2e_4+a\frac{x_2^2}{x_1}e_1-ax_2e_2,\;
		[e_2,e_4]=cx_2e_3+x_2e_4+b\frac{x_2^2}{x_1}e_1-bx_2e_2.
		\end{cases}
	} \] We take $f_2=x_1e_2-x_2e_1$ and we get
	\[ [f_2,e_3]=[f_2,e_4]=0. \]Therefore, up to a change of parameters, the Lie algebra is isomorphic to $\A_2$.

	{$(b_3,b_4)\not=(0,0)$.  } We can suppose $b_4=0$. Then \eqref{S}
	 is equivalent to
	\[ \begin{cases}
	2{\it q_1}\,{\it x_1}-2{\it p_1}\,{\it x_2}
	-{\it b_3}\,{\it y_2}
	=0,\\
	2{\it q_2}\,{\it y_1}-(2p_2+b_3)y_2=0,\\
	(2{\it p_2}-b_3)\,{
		\it y_1}+2{\it r_2}\,{\it y_2}
	=0,\\
	2{\it p_1}\,{\it x_1}+2{\it r_1}\,{\it x_2}
	-{	\it b_3}\,{\it y_1}=0,\\
	2{
		\it q_1}\,{\it y_1}-2{\it p_1}\,{\it y_2}
	-{\it b_3}\,{\it x_2}=0,\\
	2{\it q_2}\,{\it x_1}+(b_3-2p_2)x_2
	=0,\\
	{\it (2p_2+b_3)}\,{
		\it x_1}+2{\it r_2}\,{\it x_2}
	=0,\\
	2{\it p_1}\,{\it y_1}+2{\it r_1}\,{\it y_2}
	-{\it b_3}\,{\it x_1}=0.
	\end{cases}\esp\begin{cases}
	-2\,{\it b_1}\,{\it x_1}-{\it b_3}\,{\it p_1}
	-{\it q_2}\,{\it r_1}+{\it q_1}\,{\it r_2}=0,\\
	-2\,{\it b_2}
	\,{\it x_1}-{\it b_3}\,{\it r_1}-2\,{\it p_1}\,{\it r_2}
	+2\,{\it p_2}\,{\it r_1}=0,\\
	-2\,{\it x_2}\,{\it b_1}-{\it b_3}\,{\it q_1}
	+2\,{\it q_2}\,{\it p_1}-2\,{\it p_2}\,{\it q_1}=0,\\
	-2\,{
		\it x_2}\,{\it b_2}+{\it b_3}\,{\it p_1}+{\it q_2}\,{
		\it r_1}-{\it q_1}\,{\it r_2}=0.
	\end{cases} \]
	The first system can be written
	\[ \begin{cases}
	2{\it q_1}\,{\it x_1}-2{\it p_1}\,{\it x_2}
	-{\it b_3}\,{\it y_2}
	=0,\\
	2{\it p_1}\,{\it x_1}+2{\it r_1}\,{\it x_2}
	-{	\it b_3}\,{\it y_1}=0,\\
	2{\it q_2}\,{\it x_1}+(b_3-2p_2)x_2
	=0,\\
	{\it (2p_2+b_3)}\,{
		\it x_1}+2{\it r_2}\,{\it x_2}
	=0,\\
	\end{cases}
	\esp
	\begin{cases}
	2{\it q_2}\,{\it y_1}-(2p_2+b_3)y_2=0,\\
	(2{\it p_2}-b_3)\,{
		\it y_1}+2{\it r_2}\,{\it y_2}
	=0,\\
	2{
		\it q_1}\,{\it y_1}-2{\it p_1}\,{\it y_2}
	-{\it b_3}\,{\it x_2}=0,\\
	
	2{\it p_1}\,{\it y_1}+2{\it r_1}\,{\it y_2}
	-{\it b_3}\,{\it x_1}=0.
	\end{cases}\]
	
	\underline{$x_1=x_2=0$}. Then $y_1=y_2=0$ and
	\[\begin{cases}
	(2p_2-{\it b_3})\,{\it r_1}-2\,{\it p_1}\,{\it r_2}
	=0,\\
	(2p_2+{\it b_3})\,{\it q_1}
	-2\,{\it q_2}\,{\it p_1}=0,\\
	{\it b_3}\,{\it p_1}+{\it q_2}\,{
		\it r_1}-{\it q_1}\,{\it r_2}=0.
	\end{cases} \]

	$\bullet$ $p_1=0$, $r_1=0$ and $q_1=0$. Then, up to a change of parameters, $\G$ is isomorphic to $\A_3.$

	$\bullet$ $p_1=0$, $r_1=0$ and $q_1\not=0$. Then $b_3=-2p_2$ and $r_2=0$ and,   up to a change of parameters, $\G$ is isomorphic to $\A_4$.
		
	$\bullet$ $p_1=0$, $r_1\not=0$ then $q_1=0$, $b_3=2p_2$ and $q_2=0$ and, up to a change of parameters, $\G$ is isomorphic to $\A_4$.
	
	$\bullet$ $p_1\not=0$. Then
	\[ r_1=2a p_1,\; r_2=a(2p_2-b_3),\; q_1=2bp_1\esp  q_2=b(2p_2+b_3). \]
	We replace in the last equation and we get
	\[ 0=b_3p_1+ 2ab(2p_2+b_3)p_1-2abp_1(2p_2-b_3)=b_3p_1+4abp_1b_3 \]
	thus
	\[ ab=-\frac14, \]and
	\[ {\begin{cases}
		[e_1,e_2]=0,\;[e_3,e_4]=b_1e_1+b_2e_2+b_3e_3,\\
		[e_1,e_3]=-p_1e_1-2ap_1e_2,\;[e_1,e_4]=-p_2e_1-a(2p_2-b_3)e_2,\\
		[e_2,e_3]=\frac1{2a}p_1e_1+p_1e_2,\;[e_2,e_4]=\frac1{4a}(2p_2+b_3)e_1+p_2e_2. p\not=0.
		\end{cases}
	} \]Put $f_2=e_1+2ae_2$ then
	\[ [f_2,e_3]=0\esp [f_2,e_4]=\frac12b_3f_2. \]
	Therefore, up to an change of parameters, the Lie algebra is isomorphic to $\A_4.$

	{$(x_1,x_2)\not=(0,0)$.} Note first that this condition is independent of the choice of the basis $(e_1,e_2)$, i.e., if it is true for 	a basis then it is true for any other choice of basis. The system \eqref{S}  is equivalent to the three systems
	\begin{equation}\label{h} \begin{cases}
	2{\it q_1}\,{\it x_1}-2{\it p_1}\,{\it x_2}
	-{\it b_3}\,{\it y_2}
	=0,\\
	2{\it p_1}\,{\it x_1}+2{\it r_1}\,{\it x_2}
	-{	\it b_3}\,{\it y_1}=0,\\
	2{\it q_2}\,{\it x_1}+(b_3-2p_2)x_2
	=0,\\
	{\it (2p_2+b_3)}\,{
		\it x_1}+2{\it r_2}\,{\it x_2}
	=0,\\
	\end{cases}
	\begin{cases}
	2{\it q_2}\,{\it y_1}-(2p_2+b_3)y_2=0,\\
	(2{\it p_2}-b_3)\,{
		\it y_1}+2{\it r_2}\,{\it y_2}
	=0,\\
	2{
		\it q_1}\,{\it y_1}-2{\it p_1}\,{\it y_2}
	-{\it b_3}\,{\it x_2}=0,\\
	
	2{\it p_1}\,{\it y_1}+2{\it r_1}\,{\it y_2}
	-{\it b_3}\,{\it x_1}=0.
	\end{cases} 
	\begin{cases}
	2\,{\it b_1}\,{\it x_1}+{\it b_3}\,{\it p_1}
	+{\it q_2}\,{\it r_1}-{\it q_1}\,{\it r_2}=0,\\
	b_1x_1+b_2x_2=0,\\
	(2p_2-{\it b_3})\,{\it r_1}-2\,{\it p_1}\,{\it r_2}
	= 2\,{\it b_2}
	\,{\it x_1},\\
	(2p_2+{\it b_3})\,{\it q_1}
	-2\,{\it q_2}\,{\it p_1}=-2\,{\it x_2}\,{\it b_1}.
	\end{cases} \end{equation}
	Let us first show that the restriction of the Killing form $Q$ of $\G$ to $\h$ is nondegenerate Lorentzian. According to \eqref{K}, the matrix of this restriction in the basis $(e_1,e_2)$ is given by
	\[ Q_\h=\left( \begin {array}{cc} 2\,{x_{{1}}}^{2}-2\,{y_{{1}}}
	^{2}&2\,x_{{1}}x_{{2}}-2\,y_{{1}}y_{{2}}
	\\ \noalign{\medskip}2\,x_{{1}}x_{{2}}-2\,y_{{1}}y_{{2
		}}&2\,{x_{{2}}}^{2}-2\,{y_{{2}}}^{2}\end {array}
		\right) . \]
	From this formula, one can see easily that if $Q_\h$ is nondegenerate then it must be Lorentzian.
	
	Suppose  that $Q_\h$ is degenerate. Then we can choose a basis $(e_1,e_2)$ of $\h$ such that
	\[ x_1^2-y_1^2=x_1x_2-y_1y_2=0. \]Thus $y_1=\e x_1$, $x_1(x_2-\e y_2)=0$ and $\e=\pm1$. If $x_1=0$ then $y_1=0$ and from \eqref{h} we get $b_3=2p$, $r_2=0$, $y_2=0$ and $b_3=0$ which is impossible. If $x_1\not=0$ then $y_2=\e x_2$. From the last equation in the first system in \eqref{h} and the second equation in the second system, we get
	\[ 2r_2=\mu_1x_1=\mu_2y_1, 2p_2+b_3=-\mu_1x_2\esp 2p_2-b_3=-\mu_2y_2. \]
	Since $x_1\not=0$ and $y_i=\e x_i$ we deduce that $b_3=0$ which is impossible.
	
	In conclusion, we get that $Q_\h$ is nondegenerate Lorentzian and then there exists a  basis $(e_1,e_2)$ such that
	\[ x_1^2-y_1^2=x_2^2-y_2^2=0\esp x_1x_2-y_1y_2=2. \]
	So
	\[ x_2=\frac1{x_1},\; y_1=x_1\esp y_2=-x_2. \]
	The system \eqref{h} becomes
	\begin{equation}\label{h1} \begin{cases}
	2{\it q_1}\,{\it x_1}-(2{\it p_1}-b_3)\,{\it x_1^{-1}}
	=0,\\
	(2{\it p_1}-b_3)\,{\it x_1}+2{\it r_1}\,{\it x_1^{-1}}
	=0,\\
	2{\it q_2}\,{\it x_1}+(b_3-2p_2)x_1^{-1}
	=0,\\
	{\it (2p_2+b_3)}\,{
		\it x_1}+2{\it r_2}\,{\it x_1^{-1}}
	=0,\\
	\end{cases}
	\begin{cases}
	2{\it q_2}\,{\it x_1}+(2p_2+b_3)x_1^{-1}=0,\\
	(2{\it p_2}-b_3)\,{
		\it x_1}-2{\it r_2}\,{\it x_1^{-1}}
	=0,\\
	2{
		\it q_1}\,{\it x_1}+(2{\it p_1}-b_3)\,{\it x_1^{-1}}
	=0,\\
	(2{\it p_1}-b_3)\,{\it x_1}-2{\it r_1}\,{\it x_1^{-1}}
	=0.
	\end{cases} 
	\begin{cases}
	2\,{\it b_1}\,{\it x_1}+{\it b_3}\,{\it p_1}
	+{\it q_2}\,{\it r_1}-{\it q_1}\,{\it r_2}=0,\\
	b_1x_1+b_2x_2=0,\\
	(2p_2-{\it b_3})\,{\it r_1}-2\,{\it p_1}\,{\it r_2}
	= 2\,{\it b_2}
	\,{\it x_1},\\
	(2p_2+{\it b_3})\,{\it q_1}
	-2\,{\it q_2}\,{\it p_1}=-2\,{\it x_2}\,{\it b_1}.
	\end{cases} \end{equation}
	We deduce that
	\[ r_1=q_1=p_2=0, b_3=2p_1, q_2=-\frac{p_1}{x_1^2},r_2=-p_1x_1^2, b_1=-\frac{p_1^2}{x_1},b_2={p_1^2}x_1, \]and hence
	\[ \begin{cases}[e_1,e_2]=0,\;[e_3,e_4]=-\frac{p_1^2}{x_1}e_1+{p_1^2}x_1e_2+2p_1e_3,\\
	\;[e_1,e_3]=x_1(e_3-e_4-\frac{p_1}{x_1}e_1),\;[e_1,e_4]=x_1(e_3+e_4+p_1x_1e_2),\\
	\;		[e_2,e_3]=\frac1{x_1}(e_3+e_4+p_1x_1e_2),\;[e_2,e_4]=-\frac1{x_1}(e_3-e_4
	-\frac{p_1}{x_1}e_1).
	\end{cases} \]
	If we change $(e_1,e_2)$ to $(\frac1{x_1}e_1,x_1e_2)$ we get that the Lie algebra is isomorphic to $\A_5$. To conclude, one can make a change of basis of $\h$ without changing the general form of the Lie brackets in order to get $R=e_{12}^\#$.
	\end{proof}
	
	In Table \ref{11},  for each class of Lie algebras $\A_i$, $i=1,\ldots,5$ obtained in Theorem \ref{main3}, we build a family of isomorphisms (depending on the values of the parameters) from $\A_i$ onto  a nonunimodular four dimensional Lie algebra in  Table \ref{1}.

\section{Generalized complex structures of type 1 in four dimensional Lie algebras}\label{section3}

In this section, we give the classification of generalized complex structures of type 1 in four dimensional Lie algebras. Let us describe our method and give our results.
  We proceed as follows:
\begin{enumerate} \item For each family of Lie algebras $\U_i$, $\B_i$ and $\A_i$ obtained in Theorems \ref{main1}-\ref{main3},  we build a family of isomorphisms (depending on the values of the parameters) from this family onto    four dimensional Lie algebras in  Table \ref{1} and Table \ref{2}. 
This step is summarized in Tables \ref{10}-\ref{12}. 
	
	\item Once performed, the first step give us the list $\mathcal{L}$ of Lie algebras in Tables \ref{1} and \ref{2} which have a generalized complex structure of type 1. Each Lie algebra in $\mathcal{L}$ is isomorphic in different ways to some $\U_i$, $\B_i$ or $\A_i$ and hence inherits a family of generalized complex structures of type 1.

	\item The last step is the classification up to automorphisms transformations,  $B$-transformations and homotheties. The groups of automorphisms of four dimensional Lie algebras were given in \cite{P} and Table \ref{13} contains the vector spaces of  2-cocycles of four dimensional Lie algebras. This step involved a huge amount of computation using Maple and the details are given in the Appendix.

\end{enumerate}

This method  leads to the following result. 

\begin{theo}\label{type1} Let $(\G,J,R,\sigma)$ be a four dimensional  Lie algebra endowed with a generalized complex structure of type 1. Then $\G$ is isomorphic to one of the Lie algebras listed in Tables \ref{8} and \ref{9} with the corresponding triple $(J,R,\si)$ and the associated pure spinor.
	
\end{theo}

\begin{proof} The general scheme of the proof is as follows. We take a Lie algebra, say $\G$, from the ones obtained in Theorem \ref{main1}-\ref{main3}. This Lie algebra has a basis $\mathbb{B}_0=(e_1,e_2,e_3,e_4)$ where the Lie brackets depend on some parameters and the generalized complex structure is given by
	\[ J_0=\la(E_{11}+E_{22})+E_{34}-E_{43} ,\; R_0= e_{12}^\#\esp\si_0=(1+\la^2)e^{12}_\#. \]
	Depending on the parameters defining the Lie brackets this Lie algebra is isomorphic to a family of Lie algebras in the Tables \ref{1} or \ref{2}. The list of such isomorphisms are given in Tables \ref{10}-\ref{12}. Suppose that we have an isomorphism from $\G$ to a Lie algebra, say $A$, in Tables \ref{1} or \ref{2}. This isomorphism is given by the passage matrix $P$ from $\mathbb{B}_0$ to $\mathbb{B}=(f_1,f_2,f_3,f_4)$. The image by $P$ of the generalized complex structure $(J_0,R_0,\si_0)$ is given by the matrices of its component in the bases $\mathbb{B}$ and $\mathbb{B}^*$ by
	\begin{equation}\label{m} J_1=P^{-1}J_0P,\; R_1=P^{-1}R_0(P^{-1})^t\esp \si_1=P\si_0P^t. \end{equation}
	In this way we collect all the possible  generalized complex structures  on $A$. Thereafter, we proceed to the classification up to  automorphism transformations and  $B$-transformations.  In \eqref{t}, one can find how such transformations affect a given generalized complex structure. The automorphisms of four dimensional Lie algebras are given in \cite{P} and their 2-cocycles are given in Table \ref{13}.
	
	Let us perform the scheme above for  $A=A_{3,1}\oplus A_1$ with its basis $\mathbb{B}=(f_1,f_2,f_3,f_4)$ where the non vanishing Lie bracket $[f_2,f_3]=f_1$, the others cases are treated in the Appendix. This case has the advantage of using all the  techniques needed in the general case.
	
	The linear maps from $A$ to $A$, $A$ to $A^*$ or $A^*$ to $A$ are given by their matrices in the bases $\mathbb{B}$ and its dual $\mathbb{B}^*$.
	Note first that   the automorphisms  and the 2-cocycles of $A$ have, respectively, the form
	\[ T=\left[ \begin {array}{cccc} uv-xy&p&r&s\\ \noalign{\medskip}0&u&x&0
	\\ \noalign{\medskip}0&y&v&0\\ \noalign{\medskip}0&z&t&w\end {array}
	\right]\esp B=\left[ \begin {array}{cccc} 0&a_{{1,2}}&a_{{1,3}}&0
	\\ \noalign{\medskip}-a_{{1,2}}&0&a_{{2,3}}&a_{{2,4}}
	\\ \noalign{\medskip}-a_{{1,3}}&-a_{{2,3}}&0&a_{{3,4}}
	\\ \noalign{\medskip}0&-a_{{2,4}}&-a_{{3,4}}&0\end {array} \right] .
	 \]Recall that 
	  $\phi(T)=\left(\begin{array}{cc}T&0\\0&(T^{-1})^*\end{array} \right)$ and $\exp(B)=\left(\begin{array}{cc}\mathrm{Id}_A&0\\B&\mathrm{Id}_{A^*}\end{array} \right)$ are, respectively, the automorphism transformation and the $B$-transformation  associated to $T$ and  $B$. 
	  
According to Table \ref{10}, $A$ is obtained 3 times from $\U_2$ and 3 times from $\U_3$. Let us study each of this cases and derive the generalized complex structures obtained.

\begin{enumerate}

\item For $y=0,b_2\not=0,q_1=q_2=0$, the transformation
$$f_1=b_1e_1+b_2e_2,f_2=e_3,f_3=e_4,f_4=e_1$$gives an isomorphism from $\U_2$ onto $A$ and, by virtue of \eqref{m},  the generalized complex structure obtained on $A$ is 
\[ J_1=\la(E_{11}+E_{44})+E_{23}-E_{32},\;R_1=\frac1{b_2}f_{14}^\#\esp
\si_1=b_2(1+\la^2)f^{14}_\#. \]

	\item	For	$y=0,b_2=0,b_1=0,q_1\not=0$, the transformation
			$$f_1=-q_1e_1,f_2=-q_1e_2,f_3=-\frac{1}{q_1}e_3,f_4=
			-\frac{q_2}{q_1}e_3+e_4$$gives an isomorphism from $\U_2$ onto $A$ and the generalized complex structure obtained on $A$ is 
	\[ J_2=\lambda(E_{11}+E_{22})+\frac{q_2}{q_1}(E_{44}-E_{33})+\frac1{q_1}E_{43}-\frac{q_1^2+q_2^2}{q_1}E_{34}
	 ,\; R_2=-\frac1{q_1^2}f_{12}^\#
	\esp\si_2= -q_1^2(1+\la^2)f^{12}_\#. 
	 \]		
			
		\item	For	$y=0,b_2=0,b_1=0,q_2\not=0$, the transformation
		$$f_1=-q_2e_1,f_2=-q_2e_2,f_3=-\frac{1}{q_2}e_4,f_4=-\frac{q_1}{q_2}e_4+e_3$$gives an isomorphism from $\U_2$ onto $A$ and the generalized complex structure obtained on $A$ is 
		\[ J_{2b}=\lambda(E_{11}+E_{22})-\frac{q_1}{q_2}(E_{44}-E_{33})-\frac1{q_2}E_{43}+\frac{q_1^2+q_2^2}{q_2}E_{34}
		,\; R_{2b}=-\frac1{q_2^2}f_{12}^\#
		\esp\si_{2b}= -q_2^2(1+\la^2)f^{12}_\#. 
		\]

		\item For	$y=0,b_2=0,b_1\not=0$, the transformation
			$$f_1=b_1e_1,f_2=e_3,f_3=e_4,f_4=e_2+\frac{q_2}{b_1}e_3-\frac{q_1}{b_1}e_4$$gives an isomorphism from $\U_2$ onto $A$ and the generalized complex structure obtained on $A$ is 
			\[ J_3=\lambda(E_{11}+E_{44})+E_{23}-E_{32}-(q_1+\la q_2) E_{24}+(q_1\la-q_2) E_{34},\; R_3=\frac1{b_1^2}\left( q_2f_{12}^\#-q_1f_{13}^\#-f_{14}^\# \right),\;
			\si_3=-b_1^2(\la^2+1)f_{\#}^{14}. 
			 \]
		\item For	$p^2+qr=0,r\not=0,rb_1-p b_2=0$, put $b_1=\mu p,b_2=\mu r$, the transformation
			$$f_1=-pe_1-re_2,f_2=e_1,f_3=e_4,f_4=\mu e_1+e_3,
			$$gives an isomorphism from $\U_3$ onto $A$ and the generalized complex structure obtained on $A$ is
			\[ J_4=\la(E_{11}+E_{22})-\mu E_{23}+\la\mu E_{24}+E_{43}-E_{34},\; R_4=-\frac1r f_{12}^\#\esp\sigma_4=-r(1+\la^2)\left( f^{12}_\#+\mu f_\#^{14}   \right) . \]
			
		\item For	$p=r=0,q\not=0, b_2=0$ the transformation
			$$f_1=-qe_1,f_2=e_2,f_3=e_4,f_4=e_3+\frac{b_1}qe_2$$gives an isomorphism from $\U_3$ onto $A$ and the generalized complex structure obtained on $A$ is 			\[ J_5=\la(E_{11}+E_{22})+E_{43}-E_{34}
			-\frac{b_1}q(E_{23}-\la E_{24}),\; R_5= \frac1qf_{12}^\#
			 \esp\si_5=q((1+\la^2)f^{12}_\#+\left( {\la}^{2}+1 \right) b_{{1}}f^{14}_\#)
			. 
			 \]
			
		\item	For $p=r=q=0, b_i\not=0$, $i=1$ or $i=2$ the transformation $$f_1=b_1e_1+b_2e_2,f_2=e_3,f_3=e_4,f_4=e_i$$gives an isomorphism from $\U_3$ onto $A$ and the generalized complex structure obtained on $A$ is 
		\[ J_6=\la(E_{11}+E_{44})+E_{23}-E_{32},\; R_6=(-1)^i\frac1{b_i}f_{14}^\#\esp\si_6=(-1)^ib_i(1+\la^2)f^{14}_\#. \]

\end{enumerate} The generalized complex structures $(J_1,R_1,\si_1)$ and $(J_6,R_6,\si_6)$ are equivalent to $\mathcal{J}^\la=(\la(E_{11}+E_{44})+E_{23}-E_{32},\; f_{14}^\#,(1+\la^2)f^{14}_\#))$ via an automorphism of the form $(f_1,f_2;f_3,f_4)\mapsto(f_1,f_2,f_3,af_4)$.

Note that $(J_2,R_2,\si_2)$ and $(J_{2b},R_{2b},\si_{2b})$ are the same, up to a change of parameters. The reduction of $(J_2,R_2,\si_2)$ will be done by using both an automorphism and a $B$-transformation. Indeed, the automorphism $T_1=
\frac1{q_1}(E_{11}-E_{22})+pE_{12}-E_{33}+\frac1{q_1^2+q_2^2}(q_2E_{43}+q_1E_{44})$ and $B_1=-\la f^{12}_\#$ satisfy

\[ \exp(B_1)\phi(T_1^{-1})(J_2,R_2,\si_2)\phi(T_1)\exp(-B_1)=(E_{34}-E_{43},f_{12}^\#,f^{12}_\#). \]
 Thus
 $(J_2,R_2,\si_2)$ is equivalent  to $\mathcal{J}_2=(E_{34}-E_{43},f_{12}^\#,f^{12}_\#)$.

We proceed in a similar way for $(J_4,R_4,\si_4)$. The automorphism and the 2-cocycle
$$T_2=-E_{11}+\frac1r E_{22}-r(E_{33}-E_{44})+r\la\mu E_{23}\esp B_2=-\la f^{12}_\#-(1+\la^2)\mu r^2f^{13}_\# $$
 satisfy
 \[ \exp(B_2)\phi(T_2^{-1})(J_4,R_4,\si_4)\phi(T_2)\exp(-B_2)=(E_{34}-E_{43},f_{12}^\#,f^{12}_\#), \]
 and hence $(J_4,R_4,\si_4)$ is equivalent to $\mathcal{J}_2$.
  
 Similarly, the $B$-transformation $B_3=b_1f^{13}_\#-\la qf^{12}_\#$ and the automorphism $T_3=E_{11}+\frac1qE_{22}-{\la b_1}E_{23}+q(E_{33}-E_{44})$ satisfy
 \[ \phi(T_3^{-1})\exp(B_3)(J_5,R_5,\si_5)\exp(-B_3)\phi(T_3)=\mathcal{J}_2.\]

   Let us deal now with $(J_3,R_3,\si_3)$. If $(q_1,q_2)=(0,0)$ then $(J_3,R_3,\si_3)$ is equivalent to $\mathcal{J}^\la$.
    
    If $(q_1,q_2)\not=(0,0)$, then we consider
    \[ \begin{cases}T_4=-\frac{\rho}{b_1^2}E_{11}-\frac{q_2}{\rho}E_{22}+
    \frac{q_1}{\rho}E_{32}+\frac{1}{\rho}E_{42}+\frac{\la q_2+q_1}{b_1^2}E_{23}
    +\frac{q_2-\la q_1}{b_1^2}E_{33}-\frac{1}{b_1^2}E_{44}-\frac{\la}{b_1^2}E_{43},\quad\rho=\sqrt{q_1^2+q_2^2},\\
    \\B_4=-\la f_{12}^\#+\frac{\rho(1+\la^2)}{b_1^2}f_{13}^\#.\end{cases}
     \]
    Then
    \[ \exp(B_4)\phi(T_4^{-1})(J_3,R_3,\si_3)\phi(T_4)\exp(-B_4)=\mathcal{J}_2. \]

     So far, we have shown that a generalized complex structure of type 1 on $A_{3,1}\oplus A_1$ is equivalent to $\mathcal{J}^\la=(\la(E_{11}+E_{44})+E_{23}-E_{32},\; f_{14}^\#,(1+\la^2)f^{14}_\#))$ or $\mathcal{J}_2=(E_{34}-E_{43},f_{12}^\#,f^{12}_\#)$.
Since the $B$-transformations do not affect $R$ and automorphisms transformations preserve the center which is equal to $\mathrm{Im\;}f_{14}^\#$ we deduce that $\mathcal{J}^\la$ and $\mathcal{J}_2$ are not equivalent. To complete the proof, we will show that for $\la_1\not=\la_2$, $\mathcal{J}^{\la_1}$ and $\mathcal{J}^{\la_2}$ are not equivalent. Denote by $J_\la$ the first component of $\mathcal{J}^\la$. Under an automorphism transformation $T$,  $J_\la$ transforms into $TJ_\la T^{-1}$ and hence its eigenvalues doesn't change. On other hand, for any 2-cocycle $B$, $\exp(B)$ transforms $J_\la$ into $J_\la-f_{14}^\# B$ and one can check that $J_\la-f_{14}^\# B$ has the same eigenvalues as $J_\la$ which completes the proof.
\end{proof}

We end this section by giving all the classes of left invariant holomorphic Poisson tensors on four dimensional simply connected Lie groups.

Recall that a holomorphic Poisson tensor on a complex manifold $(M,J)$ is complex bivector field $\Pi=\pi+iJ\circ \pi$ where $\pi$ is a real bivector field and
\[ \mathcal{J}=\left( \begin{array}{cc}J&\pi\\0&-J^* \end{array}  \right) \]is a generalized complex structure on $M$. Hence a left invariant holomorphic Poisson tensor on a Lie group $G$ is equivalent to a generalized complex structure $(J,R,\si)$ on $\G$ with $\si=0$. In dimension four, we have the following result.

\begin{theo}\label{holom} Let $(J,R,\si)$ be generalized complex structure on a four dimensional Lie algebra $\G$ such that $\si=0$. Then $R$ is invertible and $(\G,J,R)$ is isomorphic to one of the following structures:
	\begin{enumerate}\item $A_{3,1}\oplus A_1$$:$
		$\;J=E_{41}-E_{14}+E_{32}-E_{23}\esp R=f_{12}^\#+f_{34}^\#.$

		\item $A_{4,5}^{-1,1}$$:$
		$\;J=E_{31}-E_{13}+E_{42}-E_{24}\esp R=f_{23}^\#-f_{14}^\#.$

		\item $A_{4,9}^{-1/2}$$:$
		$ J=E_{21}-E_{12}+2E_{43}-\frac12E_{34}\esp R= \cos(\theta)\left( f_{24}^\#-\frac12f_{13}^\# \right)+\sin(\theta)\left( f_{14}^\#+\frac12f_{23}^\#    \right) .$

		\item $A_{4,12}$$:$
		$ J=E_{12}-E_{21}+E_{43}-E_{34}\esp R=f_{23}^\#-f_{14}^\#.$

		\end{enumerate}
	
	\end{theo}

	\begin{proof} Note first that, by virtue of Proposition \ref{step1}, a generalized complex structure $(J,R,\si)$ of type 1 on a four dimensional Lie algebra satisfies $\si\not=0$. So a holomorphic generalized complex structure $(J,R)$ on a four dimensional Lie algebra must have $R$ invertible. According to \cite[Proposition 2.6]{C}, a couple $(J,R)$ with $R$ is invertible defines a generalized complex structure on a Lie algebra if and only if
		\[ J^2=-\mathrm{Id}_\G,\; N_J=0,\; J\circ R=R\circ J^*\esp  \mathrm{d}\om=\mathrm{d}\om_J=0, \]where $\om,\om_J\in\wedge^2\G^*$ ar given by $\om(u,v)=\prec R^{-1}(u),v\succ$ and $\om_J(u,v)=\om(Ju,v)$.
	So the determination of  holomorphic generalized complex structures on a Lie algebra $\G$ is equivalent to the determination of the couple $(J,\om)$ where $J$ is a complex structure, $\om$ is a symplectic 2-form such that $\om_J(u,v)=\om(Ju,v)$ is a also a symplectic 2-form. In dimension 4, the classification of invariant complex and symplectic structures is given in \cite{O1, O2, Snow}. According to this study, there are twelve Lie algebras which carries both a complex and symplectic structure, namely, $A_{3,1}\oplus A_1$, $A_{3,6}\oplus A_1$, $A_2\oplus2A_1$, $2A_2$, $A_{4,5}^{-1,1}$, $A_{4,6}^{\al,0}$, $A_{4,7}$, $A_{4,9}^\be$ $(\be\not=1)$, $A_{4,11}^\al$ and $A_{4,12}$. For each Lie algebra in this list, we take the list of the classes of complex structures on the Lie algebra,  we look for the holomorphic symplectic 2-forms and we proceed to the classification up to  automorphisms of Lie algebra. The result of this study leads to the structures listed in the theorem.	
				\end{proof}

Our study, combined with the results in \cite{O1, O2, Snow} lead to the following result.
\begin{co} The  four dimensional Lie algebras which have no generalized complex structure are: $A_{3,2}\oplus A_1$, $A_{3,5}^\al\oplus A_1$ with $0<|\al|<1$, $A_{4,4}$, $A_{4,2}^{\al}$ with $|\al|\not=1$ and  $A_{4,5}^{\al,\be}$ with  $-1<\al<\be<1$ and $\al+\be\not=0$. 
	\end{co}

	\section{Generalized K\"ahler structures  in four dimensional Lie algebras}\label{section6}
	
	In this section, we classify generalized K\"ahler structures in four dimensional Lie algebras as an important consequence of Theorem \ref{type1}. These structures are in correspondence with left invariant generalized K\"ahler structures in four dimensional simply connected Lie groups. One can see \cite{GaK} for a detailed introduction to generalized K\"ahler geometry.

	A generalized K\"ahler structure on a Lie algebra $\G$ consists of a pair of commuting generalized complex structures $\mathcal{J}_1$ and $\mathcal{J}_2$ such that the symmetric nondegenerate 2-form $G$ given by 
	\[ G(u,v)=\langle \mathcal{J}_1\mathcal{J}_2u,v\rangle=-\langle \mathcal{J}_1u,\mathcal{J}_2v\rangle  \]
	is positive definite, where $\prs$ is the neutral metric on $\G\oplus\G^*$.
If $\mathcal{J}_1=(J_1,R_1,\si_1)$ and $\mathcal{J}_2=(J_2,R_2,\si_2)$ then  $G$ is positive definite if and only if, for any $(X,\xi)\in\G\times\G^*\setminus\{(0,0)\}$,
\begin{equation}\label{G} \prec J_2^*\xi-\si_2(X),J_1X+R_1(\xi)\succ+\prec J_1^*\xi-\si_1(X),J_2X+R_2(\xi)\succ<0. \end{equation}
If $\mathcal{J}_1$ is of type $n=\frac12\dim\G$, i.e., $R_1=0$ then if we take $X=0$ in \eqref{G} we get that, for any $\xi\not=0$
\[ \prec J_1^*\xi,R_2(\xi)\succ<0 \]and hence $R_2$ mus be invertible thus $\mathcal{J}_2$ is of type 0. By applying a $B$-transformation to $\mathcal{J}_1$ and $\mathcal{J}_2$ we can suppose that $J_2=0$. But the fact that $\mathcal{J}_1$ and $\mathcal{J}_2$ commute and $R_2$ invertible will imply that $\si_1=0$. So, up to a $B$-transformation, a generalized K\"ahler structure with one generalized complex structure is of type $n$ is a classical K\"ahler structure. Classical left invariant K\"ahler structure on dimension 4 were classified in \cite{O1}.

	 In dimension four they are some restrictions on the types of $\mathcal{J}_1$ and $\mathcal{J}_2$.
	\begin{pr}\label{pr} Let $\mathcal{J}_1=(J_1,R_1,\si_1)$ and $\mathcal{J}_2=(J_2,R_2,\si_2)$ be a generalized K\"ahler structure on a four dimensional Lie algebra $\G$ with the types of $\mathcal{J}_1$ and $\mathcal{J}_2$ are both different of 2. Then  $\mathcal{J}_1$ and $\mathcal{J}_2$ are either both of type 0  or are both of type 1 and in this case
		\[ \mathrm{Im\;}R_1\cap \mathrm{Im\;}R_2=\{0\}.\]
		
		\end{pr}
	
	\begin{proof} Suppose that $\mathcal{J}_1$ and $\mathcal{J}_2$ are not both of type 0.  
		We distinguish two cases.
		
		\begin{enumerate}\item $\mathcal{J}_1$ and $\mathcal{J}_2$ are both of type 1. Then if $X\in\mathrm{Im\;}R_1\cap \mathrm{Im\;}R_2$ then, according to Proposition \ref{step1}, $J_1X=\la_1 X$ and $J_2X=\la_2X$ and if we replace in \eqref{G} with $\xi=0$, we get
			\[ -\prec\si_2(X),J_1X\succ-\prec\si_1(X),J_2X\succ=0 \] and hence $X=0$ and $\mathrm{Im\;}R_1\cap \mathrm{Im\;}R_2=\{0\}$.
			\item $\mathcal{J}_1$ of type 1 and $\mathcal{J}_2$ of type 0. By applying a $B$-transformation, we can suppose that $J_2=0$ and $\si_2=-R_2^{-1}$. If we take $X\in\mathrm{Im\;}R_1$ and $\xi=0$ and we replace in \eqref{G}, by virtue of Proposition \ref{step1} $J_1X=\la_1X$, we get
			\[ -\prec \si_2(X),J_1X\succ=-\la_1\prec \si_2(X),X\succ=0<0 \]
			which is a contradiction.\qedhere
			\end{enumerate}
		
		\end{proof}
		
			The following theorem gives a complete classification of generalized K\"ahler structures of type (1,1) on four dimensional Lie algebras.
			\begin{theo}\label{type11} Let $(\mathcal{J}_1,\mathcal{J}_2)$ be a generalized K\"ahler structure on a four dimensional Lie algebra such that both $\mathcal{J}_1$ and $\mathcal{J}_2$ are of type 1. Then $(\G,\mathcal{J}_1,\mathcal{J}_2)$ is isomorphic to one of the following Lie algebra with the given generalized K\"ahler structure:
				\begin{enumerate} \item $A_{3,6}\oplus A_1$:
					\[ \mathcal{J}_1=(E_{12}-E_{21},-f_{34}^\#,-f^{34}_\#)\esp
					\mathcal{J}_2=(\rho E_{43}-\frac1\rho E_{34},f_{12}^\#,f^{12}_\#),\quad\rho>0.  \]
					
					\item  $A_{2}\oplus 2A_1$:
					\[ \mathcal{J}_1=(E_{34}-E_{43},-f_{12}^\#,-f^{12}_\#)\esp\mathcal{J}_2=(\rho E_{21}-\frac1\rho E_{12},f_{34}^\#,f^{34}_\#)
					,\quad\rho>0.  \]
					
					\item  $2A_{2}$:
					\[ \mathcal{J}_1=(E_{34}-E_{43},-f_{12}^\#,-f^{12}_\#)\esp \mathcal{J}_2=(\rho E_{21}-\frac1\rho E_{12},rf_{34}^\#,\frac1rf^{34}_\#)
					,\quad\rho,r>0.  \]
					
					\item  $A_{4,6}^{\al,0}$:
					\[ \mathcal{J}_1=(E_{23}-E_{32},-f_{14}^\#,-f^{14}_\#)\esp \mathcal{J}_2=(\rho E_{41}-\frac1\rho E_{14},f_{23}^\#,f^{23}_\#)
					,\quad\rho>0.  \]

					\end{enumerate}
				
				\end{theo}
				
				\begin{proof} We scroll  Tables \ref{8} and \ref{9} and proceed by a case by case approach. There are two situations where we can say that the corresponding Lie algebra has no generalized K\"ahler structure of type (1,1).    
				\begin{enumerate}	
					\item $\G$ has only Calabi-Yau generalized complex structures which is equivalent, according to Proposition \ref{cy}, to $[\G,\G]\subset\mathrm{Im\;}R$. This implies that, for any couple of generalized complex structures $(J_1,R_1,\si_1)$ and $(J_2,R_2,\si_2)$ on $\G$, $\mathrm{Im\;}R_1\cap\mathrm{Im\;}R_2\not=\{0\}$ and hence, by virtue of Proposition \ref{pr}, $\G$ has no generalized K\"ahler structure of type $(1,1)$. In this way, we show that $A_{3,1}\oplus A_1$, $A_{3,4}\oplus A_1$ and $A_{4,1}$ have no generalized K\"ahler structure of type $(1,1)$.
					\item $\G$ has a unique class of generalized complex structure $(J_0,R_0,\si_0)$. So if $(\mathcal{J}_1,\mathcal{J}_2)$ is a generalized K\"ahler structure on $\G$ then we can take  $\mathcal{J}_1=(J_0,R_0,\si_0)$ and $\mathcal{J}_2$ is conjugate to $(J_0,R_0,\si_0)$ by a sequence of $B$-transformations and automorphism transformations. Since the $B$-transformations doesn't affect $R$ we get that $\mathrm{Im\;}R_2=T(\mathrm{Im\;}R_0)$ with $T$ is an automorphism of $\G$ and hence, by virtue of Proposition \ref{pr}, we must have 
					$T(\mathrm{Im\;}R_0)\cap \mathrm{Im\;}R_0=\{0\}$. For instance, if take $\G=A_{33}\oplus A$, $\mathrm{Im\;}R_0=\mathrm{span}\{e_3,e_4\}$ and one can check in \cite{P} that any automorphism $T$ of $\G$ satisfies $T(e_4)=a_6 e_4$ and hence $T(\mathrm{Im\;}R_0)\cap \mathrm{Im\;}R_0\not=\{0\}$. By looking at the effect on the automorphisms on $\mathrm{Im\;}R_0$, we can show that, apart form $A_{3,6}\oplus A_1$, $A_2\oplus 2A_1$, $2A_2$, $A_{4,5}^{1,1}$, $A_{4,6}^{\al,0}$ and $A_{4,12}$, all the Lie algebras in Table \ref{8} and \ref{9} have no generalized K\"ahler structure of type $(1,1)$. 
					
					\end{enumerate}
					We are left with the Lie algebras $A_{3,6}\oplus A_1$, $A_2\oplus 2A_1$, $2A_2$, $A_{4,5}^{1,1}$, $A_{4,6}^{\al,0}$ and $A_{4,12}$ for which we will devote a special treatment. We give the details for $A_{3,6}\oplus A_1$.
					
					From Table \ref{9}, $A_{3,6}\oplus A_1$ has two classes of generalized complex structures with $\mathrm{Im\;}R=\mathrm{span}\{f_3,f_4\}=\h_1$ or $\mathrm{Im\;}R=\mathrm{span}\{f_1,f_2\}=\h_2$. Any automorphism $T$ of $A_{3,6}\oplus A_1$ satisfies $T(\h_2)=\h_2$ and $T(\h_1)\cap\h_1\not=\{0\}$. So
					If $(\mathcal{J}_1,\mathcal{J}_2)$ is a generalized K\"ahler structure on $A_{3,6}\oplus A_1$ then we can suppose that
					\[ \mathcal{J}_1=(E_{12}-E_{12},-f_{34}^\#,-f^{34}_\#),  R_2=r f^{12}_\#,\; J_2(f_1)=\la f_1\esp  J_2(f_2)=\la f_2. \]Write
					\[ J_2=\left[ \begin {array}{cccc} \lambda&0&\mu_{{1,3}}&\mu_{{1,4}}
					\\ \noalign{\medskip}0&\lambda&\mu_{{2,3}}&\mu_{{2,4}}
					\\ \noalign{\medskip}0&0&\mu_{{3,3}}&\mu_{{3,4}}\\ \noalign{\medskip}0
					&0&\mu_{{4,3}}&\mu_{{4,4}}\end {array} \right] 
					\esp \si_2=\left[ \begin {array}{cccc} 0&b_{{1,2}}&b_{{1,3}}&b_{{1,4}}
					\\ \noalign{\medskip}-b_{{1,2}}&0&b_{{2,3}}&b_{{2,4}}
					\\ \noalign{\medskip}-b_{{1,3}}&-b_{{2,3}}&0&b_{{3,4}}
					\\ \noalign{\medskip}-b_{{1,4}}&-b_{{2,4}}&-b_{{3,4}}&0\end {array}
					\right]. 
					 \]The relation $\mathcal{J}_1\mathcal{J}_2=\mathcal{J}_2\mathcal{J}_1$ is equivalent to $J_2=\la(E_{11}+E_{22})+\mu_{4,4}(E_{44}-E_{33})+\mu_{3,4}E_{34}+\mu_{4,3}E_{4,3}$ and $\si_2=-b_{1,2}f^{12}_\#$.
					 Moreover, $\mathcal{J}_2$ is integrable if and only if $b_{1,2}=\frac{\la^2+1}r$ and $\mu_{3,4}=-\frac{1+\mu_{4,4}^2}{\mu_{4,3}}$.
					  
					  On the other hand, the 2-cocycle $B=\frac\la{r}f^{12}_\#$ and the automorphism $T=\frac1{\sqrt{-r}}\left( E_{11}+E_{22}  \right)+E_{33}+E_{44}+\frac{\mu_{4,4}\mu_{4,3}}{\mu_{4,4}^2+1}E_{43}$ satisfy
					  \[\begin{cases} \phi(T)\exp(B)\mathcal{J}_1\exp(-B)\phi(T^{-1})=
					  \mathcal{J}_1,\\ \phi(T)\exp(B)\mathcal{J}_2\exp(-B)\phi(T^{-1})=
					  \left(\rho E_{4,3}-\frac1\rho E_{3,4}  ,f_{12}^\#,f^{12}_\#     \right)=\mathcal{J}_3,\end{cases}  \]where $\rho=\frac{\mu_{4,3}}{\mu_{4,4}^2+1}$.
					  So $\mathcal{J}_3$ is a generalized complex structure which commutes with $\mathcal{J}_1$ and one can check that the metric $\langle \mathcal{J}_1\mathcal{J}_3.,.\rangle$ is given by
					  \[ G=f_1\otimes f_1+f^1\otimes f^1+f_2\otimes f_2+f^2\otimes f^2+\frac{\mu_{4,3}}{\mu_{4,4}^2+1}(f_4\otimes f_4+f^3\otimes f^3)+\frac{\mu_{4,4}^2+1}{\mu_{4,3}}(f_3\otimes f_3+f^4\otimes f^4). \]This achieves this case. The cases $A_2\oplus 2A_1$,  $2A_2$ and $A_{4,6}^{\al,0}$ are treated in a similar way. For $A_{4,5}^{1,1}$ and $A_{4,12}$ a direct computation shows that these Lie algebras have no generalized K\"ahler structure of type $(1,1)$.
					  \end{proof}
					  
	As shown in \cite{GaK}, a generalized K\"ahler structure gives rise to a bihermitian	metric, i.e., a Riemannian metric $g$ and two integrable complex structures $I_+$ and $I_-$ Hermitian with respect to $g$. Let us give the associated bihermitian metrics to the generalized K\"ahler structures given in Theorem \ref{type11}.
	\begin{enumerate}\item For $A_{3,6}\oplus A$, $(g,I_+,I_-)$ are given by
		\[ g=2 f^1\otimes f^1+2 f^2\otimes f^2+2\rho f^3\otimes f^3+\frac2\rho f^4\otimes f^4,\; I_+=E_{12}-E_{21}+\frac1\rho E_{34}-\rho E_{43}\esp I_-=-I_+. \]
		Moreover, $(g,I_+)$ is a K\"ahler structure and $g$ is flat.
		
		\item For $A_{2}\oplus 2A_1$, $(g,I_+,I_-)$ are given by
		\[ g=2\rho f^1\otimes f^1+\frac2\rho f^2\otimes f^2+2 f^3\otimes f^3+2 f^4\otimes f^4,\; I_+=\frac1\rho E_{12}-\rho E_{21}+ E_{34}- E_{43}\esp I_-=-I_+. \]
		Moreover, $(g,I_+)$ is a K\"ahler structure and the Ricci operator of $g$ is given by $\mathrm{Ric}=-\frac1{2\rho}\left(E_{11}+E_{22} \right)$ and the Ricci curvature is nonpositive.
		
		\item For $2A_{2}$, $(g,I_+,I_-)$ are given by
		\[ g=2\rho f^1\otimes f^1+\frac2\rho f^2\otimes f^2+\frac2r f^3\otimes f^3+\frac2r f^4\otimes f^4,\; I_+=\frac1\rho E_{12}-\rho E_{21}+ E_{34}-E_{43}\esp I_-=-I_+. \]
		Moreover, $(g,I_+)$ is a K\"ahler structure and the Ricci operator of $g$ is given by $\mathrm{Ric}=-\frac1{2\rho}\left(E_{11}+E_{22} \right)-\frac{r}2\left(E_{33}+E_{44} \right)$. In particular, if $r=\frac1\rho$ then $g$ is Einstein with negative scalar curvature.
		\item For $A_{4,6}^{\al,0}$, $(g,I_+,I_-)$ are given by
		\[ g=2\rho f^1\otimes f^1+ f^2\otimes f^2+ f^3\otimes f^3+\frac2\rho f^4\otimes f^4,\; I_+=\frac1\rho E_{14}-\rho E_{41}+ E_{23}-E_{32}\esp I_-=-I_+. \]
		Moreover, $(g,I_+)$ is a K\"ahler structure and the Ricci operator of $g$ is given by $\mathrm{Ric}=-\frac{\rho\al^2}{2}\left(E_{11}+E_{44} \right)$ and and the Ricci curvature is nonpositive.
		
		\end{enumerate}
	
		It is known that a four dimensional generalized K\"ahler manifold of type $(0,0)$ has a non trivial holomorphic Poisson tensor \cite{GaK}. So in our case, the only four dimensional Lie algebras which can carry a generalized K\"ahler structure of type $(0,0)$ are those given in Theorem \ref{holom}.  By a case by case approach and a direct computation, one can show no one of them has a generalized complex structure of type $(0,0)$. This prove the following result.
		\begin{theo}\label{type0} There is no generalized K\"ahler structure $(\mathcal{J}_1,\mathcal{J}_2)$ on a four dimensional Lie algebra such that both $\mathcal{J}_1$ and $\mathcal{J}_2$ are of type 0.
			
		\end{theo}

\section{Cohomologies of generalized complex structures of type 1 on four dimensional Lie algebras}\label{section4}

In this section, we compute for each class of generalized complex structure of type 1 obtained in the last section its different cohomologies. 

Let us start by recalling the definitions of these cohomologies (see \cite{Ga} for the details). Recall that a generalized complex structure $K=(J,R,\si)$ on a Lie algebra $\G$ of dimension $2n$ is entirely determined by a pure spinor $\rho\in\wedge^\bullet\G^*\otimes\mathbb{C}$ such that
\[ \mathrm{L}:=\left\{X+\xi-iK(X+\xi),X+\xi\in\Phi(\G) \otimes\mathbb{C}  \right\}=\left\{X+\xi\in \Phi(\G) \otimes\mathbb{C},(X+\xi).\rho=i_X\rho+\xi\wedge\rho=0   \right\}. \]
For $k\in\{0,\ldots,2n  \}$, we define the vector subspaces $\mathrm{U}_{-n+k}\subset\wedge^\bullet\G^*\otimes\mathbb{C}$ by 
\[ \mathrm{U}_{-n}=\langle \rho\rangle \esp \mathrm{U}_{-n-k}=\wedge^k\overline{\mathrm{L}}.\mathrm{U}_{-n}, \]where $\overline{\mathrm{L}}$ is the conjugate of $\mathrm{L}$. This family of vector spaces defines a $\mathbb{Z}$-graduation of $\wedge^\bullet\G^*\otimes\mathbb{C}$, i.e.,
\[ \wedge^\bullet\G^*\otimes\mathbb{C}=\mathrm{U}_{-n}\oplus\ldots\oplus \mathrm{U}_{n}. \]Moreover, $\mathrm{U}_{l}=\overline{\mathrm{U}_{-l}}$ for any $l\in\{0,\ldots,n\}$ and the differential $\mathrm{d}$ satisfies $\mathrm{d}\mathrm{U}_{\bullet}\subset \mathrm{U}_{\bullet-1}+\mathrm{U}_{\bullet+1}$. So $\mathrm{d}=\partial+\bar{\partial}$ where
\[ \partial:\mathrm{U}_{\bullet}\too \mathrm{U}_{\bullet+1}\esp  \bar{\partial}:\mathrm{U}_{\bullet}\too \mathrm{U}_{\bullet-1},\]
$\partial\circ\partial=0$,  $\bar\partial\circ\bar\partial=0$ and $\partial\circ\bar\partial+\bar\partial\circ\partial=0$. Associated to these differential complexes, there are three cohomologies:
\begin{eqnarray*}
	\mathrm{GH^\bullet_{\partial}}&=&\frac{\ker(\partial:\mathrm{U}_\bullet\too \mathrm{U}_{\bullet+1}) }{\mathrm{Im}(\partial:\mathrm{U}_{\bullet-1}\too \mathrm{U}_\bullet )},\\
	\mathrm{GH^\bullet_{BC}}&=&\frac{\ker(\partial: \mathrm{U}_\bullet\too \mathrm{U}_{\bullet+1})
		\cap \ker(\bar\partial: \mathrm{U}_\bullet\too \mathrm{U}_{\bullet-1})}{\mathrm{Im}(\partial\bar\partial:\mathrm{U}_\bullet\too \mathrm{U}_\bullet)},\\
	\mathrm{GH}^\bullet_A&=&\frac{\ker(\partial\bar\partial: \mathrm{U}_\bullet\too \mathrm{U}_\bullet)  }{\mathrm{Im}(\partial:\mathrm{U}_{\bullet-1}\too \mathrm{U}_\bullet )+\mathrm{Im}(\bar\partial:\mathrm{U}_{\bullet+1}\too \mathrm{U}_\bullet )}.
	\end{eqnarray*}They are known as generalized Dolbeault cohomology, generalized Bott-Chern cohomology and generalized Aeppli cohomology. The cohomology $\mathrm{GH^\bullet_{\bar\partial}}$ can be obtained from $\mathrm{GH^\bullet_{\partial}}$ and the relations $\mathrm{U}_{l}=\overline{\mathrm{U}_{-l}}$.

Let us give these three cohomologies for the generalized complex structures obtained in Tables \ref{8} and \ref{9}. The computation is straightforward, we give the details for one example and for the others we give only $\rho,\overline{\mathrm{L}},\mathrm{U}_{-1},\mathrm{U}_0$ and the cohomologies. We have $\mathrm{U}_{-2}=\langle\rho\rangle$, $\mathrm{U}_{2}=\langle\overline{\rho}\rangle$, $\mathrm{U}_{1}=\overline{\mathrm{U}_{-1}}$, $\dim\mathrm{U}_{1}=4$ and $\dim\mathrm{U}_{0}=6$. Each one of these vector spaces is given by a family of generators, $\mathrm{U}_{-1}=\langle \mathrm{U}_{-1}^1,\ldots,\mathrm{U}_{-1}^4\rangle$, $\mathrm{U}_{1}=\langle \mathrm{U}_{1}^1,\ldots,\mathrm{U}_{1}^4\rangle$ and $\mathrm{U}_{0}=\langle \mathrm{U}_{0}^1,\ldots,\mathrm{U}_{0}^6\rangle$.

Note that 
\[ \mathrm{GH}^{-2}_\partial=\begin{cases}0\;\mbox{if}\;d\rho\not=0\\\mathrm{U}_{-2}\;\mbox{if}\;d\rho=0\end{cases},\mathrm{GH}^{2}_\partial=\begin{cases}0\;\mbox{if}\;\mathrm{Im}\overline{\partial}_{-1}\not=0\\\mathrm{U}_{2}\;\mbox{if}\;\mathrm{Im}\overline{\partial}_{-1}=0\end{cases},\mathrm{GH}^{-2}_{BC}=\begin{cases}0\;\mbox{if}\;d\rho\not=0\\\mathrm{U}_{-2}\;\mbox{if}\;d\rho=0\end{cases} \]
\[\mathrm{GH}^{2}_{BC}=\begin{cases}0\;\mbox{if}\;d\rho\not=0\\\mathrm{U}_{2}\;\mbox{if}\;d\rho=0\end{cases}, \mathrm{GH}^{-2}_A=\begin{cases}0\;\mbox{if}\;\mathrm{Im}\overline{\partial}_{-1}\not=0\\\mathrm{U}_{-2}\;\mbox{if}\;\mathrm{Im}\overline{\partial}_{-1}=0\end{cases} 
\esp \mathrm{GH}^{2}_A=\begin{cases}0\;\mbox{if}\;\mathrm{Im}\overline{\partial}_{-1}\not=0\\\mathrm{U}_{2}\;\mbox{if}\;\mathrm{Im}\overline{\partial}_{-1}=0.\end{cases}\]

\begin{exem}We begin by the first example in Table \ref{8}. 
We have
\begin{align*}
\rho&=	{f}^{2}+i{f}^{3}- \left( i\lambda+1 \right) {f}^{134}+ \left( i-\lambda
\right) {f}^{124},\;\mathrm{d}\rho=0,\\
\overline{\mathrm{L}}&=\langle i \left( -{\lambda}^{2}-1 \right){f}^{4} + \left( 1+i\lambda
\right)f_{{1}} ,if_{{2}}+f_{{3}}, \left( 1-i\lambda \right)f^1 -if_{{4}},{f}^{3}+i{f}^{2}\rangle,\\
\mathrm{U}_{-1}&=\langle \left( -2\,\lambda+2\,i \right) {f}^{24}+ \left( -2\,i\lambda-2
\right) {f}^{34},
2\,i+ \left( 2\,i\lambda+2 \right) {f}^{14},
2\,{f}^{12}+2\,i{f}^{13},
-2\,{f}^{23}+ \left( -2\,\lambda+2\,i \right) {f}^{1234}
\rangle\\
\mathrm{U}_0&=\langle  \left( 4\,i\lambda+4 \right) {f}^{4},
\left( 2\,i\lambda+2 \right) {f}^{2}+ \left( -2\,\lambda+2\,i \right) 
{f}^{3}+ \left( 2\,{\lambda}^{2}+2 \right) {f}^{134}+ \left( -2\,i-2\,i
{\lambda}^{2} \right) {f}^{124},\\
& \left( 4\,i-4\,\lambda \right) {f}^{234},
-4\,if^1,
2\,{f}^{2}-2\,i{f}^{3}+ \left( 2\,i\lambda+2 \right) {f}^{134}+
\left( -2\,\lambda+2\,i \right) {f}^{124},
-4\,{f}^{123}\rangle,\\
\mathrm{d}\mathrm{U}_{-1}^1&=\mathrm{d}\mathrm{U}_{-1}^3=\mathrm{d}\mathrm{U}_{-1}^4=0,\;\mathrm{d}\mathrm{U}_{-1}^2= \left( -2\,i\lambda-2 \right) {f}^{234},	\\
\mathrm{d}\mathrm{U}_0^1&=\mathrm{d}\mathrm{U}_0^2=\mathrm{d}\mathrm{U}_0^3=
\mathrm{d}\mathrm{U}_0^5=\mathrm{d}\mathrm{U}_0^6=0,\; \mathrm{d}\mathrm{U}_0^4=4\,i{f}^{23}.
\end{align*}

Thus
\[\begin{cases} \ker\partial_{-2}=\ker{\bar{\partial}}_{-2}=\mathrm{U}_{-2}\esp \mathrm{Im}\partial_{-2}=\mathrm{Im}{\bar{\partial}}_{-2}=0,\\
\ker\partial_{-1}=\langle \mathrm{U}_{-1}^1,\mathrm{U}_{-1}^3,\mathrm{U}_{-1}^4 \rangle, \mathrm{Im}\partial_{-1}=\langle \mathrm{U}_0^3\rangle\esp
\ker\bar\partial_{-1}=\mathrm{U}_{-1} , \mathrm{Im}\bar\partial_{-1}=0 ,\\
\ker\bar{\partial}_{0}=\langle \mathrm{U}_0^1,\mathrm{U}_0^2,\mathrm{U}_0^3,\mathrm{U}_0^5,\mathrm{U}_0^6 \rangle\esp \mathrm{Im}\bar{\partial}_{0}=\langle \mathrm{U}_{-1}^4\rangle,\\
\ker{\partial}_{0}=\langle \mathrm{U}_0^1,\mathrm{U}_0^2,\mathrm{U}_0^3,\mathrm{U}_0^5,\mathrm{U}_0^6 \rangle\esp \mathrm{Im}{\partial}_{0}=\langle \mathrm{U}_{1}^4\rangle,\\
\ker\bar{\partial}_{1}=\overline{\ker\partial_{-1}}\esp \mathrm{Im}\bar{\partial}_{1}=\overline{\mathrm{Im}\partial_{-1}}\;
\ker{\partial}_{1}=\overline{\ker\bar\partial_{-1}}\esp \mathrm{Im}{\partial}_{1}=\overline{\mathrm{Im}\bar\partial_{-1}}.\end{cases}\]

From this, we deduce:

\[\begin{cases}  \mathrm{GH}_\partial^{-1}=\langle \mathrm{U}_{-1}^1,\mathrm{U}_{-1}^3,\mathrm{U}_{-1}^4 \rangle,\;
	\mathrm{GH}_\partial^{0}=\langle \mathrm{U}_0^1,\mathrm{U}_0^2,\mathrm{U}_0^5,\mathrm{U}_0^6 \rangle,\;\mathrm{GH}_\partial^{1}=\langle \mathrm{U}_1^1,\mathrm{U}_1^2,\mathrm{U}_1^3 \rangle,\\
	\mathrm{GH_{BC}^{-1}}=\langle \mathrm{U}_{-1}^1,\mathrm{U}_{-1}^3,\mathrm{U}_{-1}^4 \rangle,\;\mathrm{GH_{BC}^{0}}=\langle \mathrm{U}_0^1,\mathrm{U}_0^2,\mathrm{U}_0^3,\mathrm{U}_0^5,\mathrm{U}_0^6 \rangle,\;\mathrm{GH_{BC}^{1}}=\langle \mathrm{U}_1^1,\mathrm{U}_1^3,\mathrm{U}_1^4\rangle,
	\\
 \mathrm{GH}_A^{-1}=\langle \mathrm{U}_{-1}^1,\mathrm{U}_{-1}^2,\mathrm{U}_{-1}^3\rangle,\;\mathrm{GH}_A^{0}=\langle \mathrm{U}_0^1,\mathrm{U}_0^2,\mathrm{U}_{0}^4,\mathrm{U}_0^5,\mathrm{U}_0^6\rangle,\; \mathrm{GH}_A^{1}=\langle \mathrm{U}_{1}^1,\mathrm{U}_1^2,\mathrm{U}_{1}^3\rangle.  \end{cases}\]
\end{exem}

Let us give now the cohomologies of the others generalized complex structures in Tables \ref{8} and \ref{9}.
\begin{eqnarray*}
	A_{3,1}\oplus A_1:\;\rho&=&	{f}^{3}+i{f}^{4}+{f}^{124}-i{f}^{123},\;\mathrm{d}\rho=0,\mathrm{Im}\bar\partial_{-1}=0,\\
	\overline{\mathrm{L}}&=&\langle -i{f}^{2}+f_{{1}},
	f_{{3}}-if_{{4}},
	f^1-if_{{2}},
	-i{f}^{4}+{f}^{3}
	\rangle,\\
	\mathrm{U}_{-1}&=&\langle-\,i{f}^{23}+\,{f}^{24},
	1-\,i{f}^{12},
	\,{f}^{13}+\,i{f}^{14},
	\,i{f}^{34}+{f}^{1234}
	\rangle,\\
	\mathrm{U}_0&=&\langle {f}^{2},
	\,{f}^{3}+\,i{f}^{4}-\,{f}^{124}+\,i{f}^{123},
	{f}^{234},
	f^1,
	-\,{f}^{3}+\,i{f}^{4}+2\,{f}^{124}+\,i{f}^{123},
	{f}^{134}
	\rangle,\\
	 \mathrm{GH}_\partial^{-1}&=&\langle \mathrm{U}_{-1}^1,\mathrm{U}_{-1}^2,\mathrm{U}_{-1}^4 \rangle,\;
	\mathrm{GH}_\partial^{0}=\langle \mathrm{U}_0^1,\mathrm{U}_0^2,\mathrm{U}_0^5,\mathrm{U}_0^6 \rangle,\;\mathrm{GH}_\partial^{1}=\langle \mathrm{U}_1^2,\mathrm{U}_1^3,\mathrm{U}_1^4 \rangle ,\\
	\;\mathrm{GH_{BC}^{-1}}&=&\langle \mathrm{U}_{-1}^1,\mathrm{U}_{-1}^2,\mathrm{U}_{-1}^4 \rangle,\;\mathrm{GH_{BC}^{0}}=\langle \mathrm{U}_0^1,\mathrm{U}_0^2,\mathrm{U}_0^3,\mathrm{U}_0^5,\mathrm{U}_0^6 \rangle,\;\mathrm{GH_{BC}^{1}}=\langle \mathrm{U}_1^1,\mathrm{U}_1^2,\mathrm{U}_1^4\rangle,\\
	 \mathrm{GH}_A^{-1}&=&\langle \mathrm{U}_{-1}^2,\mathrm{U}_{-1}^3,\mathrm{U}_{-1}^4\rangle,\;\mathrm{GH}_A^{0}=\langle \mathrm{U}_0^1,\mathrm{U}_0^2,\mathrm{U}_{0}^4,\mathrm{U}_0^5,\mathrm{U}_0^6\rangle,\; \mathrm{GH}_A^{1}=\langle \mathrm{U}_{1}^2,\mathrm{U}_1^3,\mathrm{U}_{1}^4\rangle.
\end{eqnarray*}
\begin{eqnarray*}
A_{3,4}\oplus A_1:	\rho&=&-{f}^{3}+i{f}^{4}+{f}^{124}+i{f}^{123}, \mathrm{d}\rho=0,\mathrm{Im}\bar\partial_{-1}=0,\\
	 \overline{\mathrm{L}}&=&\langle -i{f}^{2}+f_{{1}},
	f_{{3}}+if_{{4}},
	f^1-if_{{2}},
	{f}^{4}-i{f}^{3}
	\rangle,\\
	\mathrm{U}_{-1}&=&\langle \,i{f}^{23}+\,{f}^{24},
	-1+\,i{f}^{12},
	-\,{f}^{13}+\,i{f}^{14},
	\,{f}^{34}-\,i{f}^{1234}
	\rangle,\\	
	\mathrm{U}_0&=&\langle {f}^{2},
	-\,{f}^{3}+\,i{f}^{4}-\,{f}^{124}-\,i{f}^{123},
	{f}^{234},
	f^1,
	-\,i{f}^{3}+\,{f}^{4}-\,i{f}^{124}-\,{f}^{123},
	{f}^{134}
	\rangle,\\
		\mathrm{GH}_\partial^{-1}&=&\langle \mathrm{U}_{-1}^2, \mathrm{U}_{-1}^4 \rangle,\;\mathrm{GH}_\partial^{0}=\langle \mathrm{U}_0^2,\mathrm{U}_0^5 \rangle,\;
		\mathrm{GH}_\partial^{1}=\langle \mathrm{U}_1^2,\mathrm{U}_{1}^4 \rangle,\\
		\mathrm{GH_{BC}^{-1}}&=&\langle \mathrm{U}_{-1}^2,\mathrm{U}_{-1}^4\rangle,\;
		\mathrm{GH_{BC}^{0}}=\langle \mathrm{U}_0^2,\mathrm{U}_0^5 \rangle,\;\mathrm{GH_{BC}^{1}}=\langle \mathrm{U}_1^2,\mathrm{U}_1^4 \rangle,\\
		\mathrm{GH}_A^{-1}&=&\langle \mathrm{U}_{-1}^2,\mathrm{U}_{-1}^4\rangle,\;\mathrm{GH}_A^{0}=\langle \mathrm{U}_{0}^2,\mathrm{U}_0^5\rangle,\; \mathrm{GH}_A^{1}=\langle \mathrm{U}_{1}^2,\mathrm{U}_{1}^4\rangle.
\end{eqnarray*}	
\begin{eqnarray*}
	A_{3,6}\oplus A_1: \rho&=&-f^1+i{f}^{2}+{f}^{234}+i{f}^{134}, \mathrm{d}\rho=i{f}^{13}+{f}^{23}=f^4.\rho,\mathrm{Im}\bar\partial_{-1}=\langle\rho\rangle,\\
	 \overline{\mathrm{L}}&=&\langle f_{{1}}+if_{{2}},
	-i{f}^{4}+f_{{3}},
	i{f}^{2}+f^1,
	{f}^{3}-if_{{4}}
	\rangle,\\
	\mathrm{U}_{-1}&=&\langle -1+\,i{f}^{34},
	-\,i{f}^{14}-\,{f}^{24},
	\,i{f}^{12}+\,{f}^{1234},
	\,{f}^{13}-\,i{f}^{23}
	\rangle,\\
	\mathrm{U}_0&=&\langle {f}^{4},
	\,f^1+\,i{f}^{2}+\,{f}^{234}-\,i{f}^{134},
	{f}^{3},
	{f}^{124},
	-\,f^1+\,i{f}^{2}-\,{f}^{234}-\,i{f}^{134},
	{f}^{123}
	\rangle,\\
	\mathrm{GH}_\partial^{-1}&=&\langle \mathrm{U}_{-1}^1, \mathrm{U}_{-1}^3 \rangle,\;
	\mathrm{GH}_\partial^{0}=\langle \mathrm{U}_0^1,\mathrm{U}_0^3,\mathrm{U}_{0}^4,\mathrm{U}_0^6 \rangle,\;
	\mathrm{GH}_\partial^{1}=\langle \mathrm{U}_1^1,\mathrm{U}_{1}^3 \rangle,\\
	\mathrm{GH_{BC}^{-1}}&=&\langle \mathrm{U}_{-1}^1,\mathrm{U}_{-1}^3\rangle,\;
	\mathrm{GH_{BC}^{0}}=\langle \mathrm{U}_0^1,\mathrm{U}_0^3,\mathrm{U}_0^4,\mathrm{U}_0^6 \rangle,\;\mathrm{GH_{BC}^{1}}=\langle \mathrm{U}_1^1,\mathrm{U}_1^3 \rangle,\\
	\mathrm{GH}_A^{-1}&=&\langle \mathrm{U}_{-1}^1,\mathrm{U}_{-1}^3\rangle,\;\mathrm{GH}_A^{0}=\langle \mathrm{U}_{0}^1,\mathrm{U}_0^3,\mathrm{U}_0^4,\mathrm{U}_0^6\rangle,\; \mathrm{GH}_A^{1}=\langle \mathrm{U}_{1}^1,\mathrm{U}_{1}^3\rangle.
\end{eqnarray*}
\begin{eqnarray*}
	A_{3,6}\oplus A_1: \rho&=&-{f}^{3}+i{f}^{4}+{f}^{124}+i{f}^{123}, \mathrm{d}\rho=0,\mathrm{Im}\bar\partial_{-1}=0,\\
	 \overline{\mathrm{L}}&=&\langle -i{f}^{2}+f_{{1}},
	f_{{3}}+if_{{4}},
	f^1-if_{{2}},
	i{f}^{4}+{f}^{3}\rangle,\\
	\mathrm{U}_{-1}&=&\langle \,i{f}^{23}+\,{f}^{24},
	-1+\,i{f}^{12},
	-\,{f}^{13}+\,i{f}^{14},
	\,i{f}^{34}+\,{f}^{1234}
	\rangle,\\
	\mathrm{U}_0&=&\langle {f}^{2},
	-\,{f}^{3}+\,i{f}^{4}-\,{f}^{124}-\,i{f}^{123},
	{f}^{234},
	f^1,
	{f}^{3}+\,i{f}^{4}+\,{f}^{124}-\,i{f}^{123},
	{f}^{134}
	\rangle,\\
	\mathrm{GH}_\partial^{-1}&=&\langle \mathrm{U}_{-1}^2, \mathrm{U}_{-1}^4 \rangle,\;
	\mathrm{GH}_\partial^{0}=\langle \mathrm{U}_0^2,\mathrm{U}_0^5 \rangle,\;
	\mathrm{GH}_\partial^{1}=\langle \mathrm{U}_1^2,\mathrm{U}_{1}^4 \rangle,\\
	\mathrm{GH_{BC}^{-1}}&=&\langle \mathrm{U}_{-1}^2,\mathrm{U}_{-1}^4\rangle,\;
	\mathrm{GH_{BC}^{0}}=\langle \mathrm{U}_0^2,\mathrm{U}_0^5 \rangle,\;\mathrm{GH_{BC}^{1}}=\langle \mathrm{U}_1^2,\mathrm{U}_1^4 \rangle,\\
	\mathrm{GH}_A^{-1}&=&\langle \mathrm{U}_{-1}^2,\mathrm{U}_{-1}^4\rangle,\;\mathrm{GH}_A^{0}=\langle \mathrm{U}_{0}^2,\mathrm{U}_0^5\rangle,\; \mathrm{GH}_A^{1}=\langle \mathrm{U}_{1}^2,\mathrm{U}_{1}^4\rangle.
\end{eqnarray*}

\begin{eqnarray*}
	A_{3,8,9}\oplus A_1: \rho&=&f^1-i{f}^{2}- \left( i\lambda+1 \right) {f}^{234}+ \left( \lambda-i
	\right) {f}^{134},\mathrm{d}\rho=-if^{13}-f^{23},\mathrm{Im}\bar\partial_{-1}=\langle\rho\rangle,\\
	\overline{\mathrm{L}}&=&\langle f_{{1}}+if_{{2}},
	i \left( -{\lambda}^{2}-1 \right){f}^{4} + \left( 1+i\lambda
	\right)f_{{3}} ,
	i \left( {\lambda}^{2}+1 \right){f}^{3} + \left( 1+i\lambda
	\right)f_{{4}} ,
	{f}^{2}-if^1
	\rangle,\\
	\mathrm{U}_{-1}&=&\langle 1+ \left( \,\lambda-\,i \right) {f}^{34},
	\left( -\,\lambda+\,i \right) {f}^{14}+ \left( \,i\lambda+1
	\right) {f}^{24},
	\left( \,\lambda-\,i \right) {f}^{13}+ \left( -\,i\lambda-1
	\right) {f}^{23},
	-\,{f}^{12}+ \left( -\,\lambda+\,i \right) {f}^{1234}
	\rangle,\\
	\mathrm{U}_0&=&\langle  {f}^{4},
	{f}^{3},
	\,if^1-\,{f}^{2}+ \left( -\,\lambda+\,i \right) {f}^{234}+ \left( 
	\,i\lambda+ 1\right) {f}^{134},\\
	&&	\,   \left( i-\lambda \right) f^1+\, \left( 1+
	i\lambda \right) {f}^{2}+\,i \left( {\lambda}^{2}+1 \right) 
	{f}^{234}-\, \left( {\lambda}^{2}+1
	\right)   {f}^{134},
	{f}^{124},
	{f}^{123}
	\rangle,\\
	\mathrm{GH}_\partial^{-1}&=&\langle \mathrm{U}_{-1}^1,\mathrm{U}_{-1}^4 \rangle,\;\mathrm{GH}_\partial^{0}=\langle \mathrm{U}_0^1,\mathrm{U}_{0}^2,\mathrm{U}_0^5,\mathrm{U}_0^6 \rangle,\;
	\mathrm{GH}_\partial^{1}=\langle \mathrm{U}_1^1,\mathrm{U}_1^4 \rangle,\\
	\mathrm{GH_{BC}^{-1}}&=&\langle \mathrm{U}_{-1}^1,\mathrm{U}_{-1}^4\rangle,\;
	\mathrm{GH_{BC}^{0}}=\langle \mathrm{U}_0^1,\mathrm{U}_0^2,\mathrm{U}_0^5,\mathrm{U}_0^6 \rangle,\;\mathrm{GH_{BC}^{1}}=\langle \mathrm{U}_1^1,\mathrm{U}_1^4 \rangle,\\
	\mathrm{GH}_A^{-1}&=&\langle \mathrm{U}_{-1}^1,\mathrm{U}_{-1}^4\rangle,\;\mathrm{GH}_A^{0}=\langle \mathrm{U}_{0}^1,\mathrm{U}_{0}^2,\mathrm{U}_0^5,\mathrm{U}_0^6\rangle,\; \mathrm{GH}_A^{1}=\langle \mathrm{U}_{1}^1,\mathrm{U}_{1}^4\rangle.
\end{eqnarray*}

\begin{eqnarray*}
	A_{4,1}: \rho&=&-{f}^{3}+i{f}^{4}+ \left( 1+i\lambda \right) {f}^{124}+ \left( i-
	\lambda \right) {f}^{123}, \;\mathrm{d}\rho=0,\mathrm{Im}\bar\partial_{-1}=0,\\
	\overline{\mathrm{L}}&=&\langle -i \left( {\lambda}^{2}+1 \right){f}^{2} + \left( 1+i\lambda
	\right)f_{{1}} ,
	f_{{3}}+if_{{4}},
	\left( 1-i\lambda \right)f^1 -if_{{2}},
	i{f}^{4}+{f}^{3}\rangle,\\
	\mathrm{U}_{-1}&=&\langle  \left( -\,\lambda+\,i \right) {f}^{23}+ \left( \,i\lambda+1
	\right) {f}^{24},
	-1+ \left( -\,\lambda+\,i \right) {f}^{12},
	-\,{f}^{13}+2\,i{f}^{14},
	\,i{f}^{34}+ \left( \,i\lambda+1 \right) {f}^{1234}	
	\rangle,\\
	\mathrm{U}_0&=&\langle {f}^{2},
	-\left( \,i\lambda+1 \right) {f}^{3}+ \left( -\,\lambda+\,i
	\right) {f}^{4}- \left( \,{\lambda}^{2}+1 \right) {f}^{124}-i \left( 
	1+{\lambda}^{2} \right) {f}^{123},
	{f}^{234},
	\,f^1,\\&&
	\,{f}^{3}+\,i{f}^{4}+ \left( \,i\lambda+1 \right) {f}^{124}+
	\left( \,\lambda-\,i \right) {f}^{123},
	{f}^{134}
	\rangle,\\
	\mathrm{GH}_\partial^{-1}&=&\langle \mathrm{U}_{-1}^1, \mathrm{U}_{-1}^4 \rangle,\;\mathrm{GH}_\partial^{0}=\langle \mathrm{U}_0^2,\mathrm{U}_0^5 \rangle,\;
	\mathrm{GH}_\partial^{1}=\langle \mathrm{U}_1^2,\mathrm{U}_1^3 \rangle,\\
	\mathrm{GH_{BC}^{-1}}&=&\langle \mathrm{U}_{-1}^1,\mathrm{U}_{-1}^4\rangle,\;
	\mathrm{GH_{BC}^{0}}=\langle \mathrm{U}_0^2,\mathrm{U}_0^3,\mathrm{U}_0^5,\mathrm{U}_0^6 \rangle,
	\;\mathrm{GH_{BC}^{1}}=\langle \mathrm{U}_1^1,\mathrm{U}_1^4 \rangle,\\
	\mathrm{GH}_A^{-1}&=&\langle \mathrm{U}_{-1}^2,\mathrm{U}_{-1}^3\rangle,\;\mathrm{GH}_A^{0}=\langle \mathrm{U}_{0}^1,\mathrm{U}_{0}^2,\mathrm{U}_{0}^4,\mathrm{U}_0^5\rangle,\; \mathrm{GH}_A^{1}=\langle \mathrm{U}_{1}^2,\mathrm{U}_{1}^3\rangle.
\end{eqnarray*}

\begin{eqnarray*}
	A_{4,5}^{-\frac12,-\frac12}:\; \rho&=&{f}^{2}+i{f}^{3}-{f}^{134}+i{f}^{124},\;
	\mathrm{d}\rho=\frac12\,{f}^{24}+\frac{i}2{f}^{34},\mathrm{Im}\bar\partial_{-1}=
	\langle\rho\rangle,\\
	\overline{\mathrm{L}}&=&\langle-i{f}^{4}+f_{{1}},
	if_{{2}}+f_{{3}},
	f^1-if_{{4}},
	{f}^{3}+i{f}^{2}
	\rangle,\\
	\mathrm{U}_{-1}&=&\langle \,i{f}^{24}-\,{f}^{34},
	\,i+\,{f}^{14},
	\,{f}^{12}+\,i{f}^{13},
	-\,{f}^{23}+\,i{f}^{1234}	
	\rangle,\\
	\mathrm{U}_0&=&\langle {f}^{4},
	\,{f}^{2}+\,i{f}^{3}+\,{f}^{134}-\,i{f}^{124},
	{f}^{234},
	f^1,
	{f}^{2}-\,i{f}^{3}+\,{f}^{134}+\,i{f}^{124},
	{f}^{123}
	\rangle,\\
	\mathrm{GH}_\partial^{-1}&=&\langle \mathrm{U}_{-1}^2 \rangle,\;\mathrm{GH}_\partial^{0}=\langle \mathrm{U}_0^1,\mathrm{U}_0^6 \rangle,\;
	\mathrm{GH}_\partial^{1}=\langle \mathrm{U}_1^4 \rangle,\\
	\mathrm{GH_{BC}^{-1}}&=&\langle \mathrm{U}_{-1}^2\rangle,\;
	\mathrm{GH_{BC}^{0}}=\langle \mathrm{U}_0^1,\mathrm{U}_0^3,\mathrm{U}_0^6 \rangle,\;\mathrm{GH_{BC}^{1}}=\langle \mathrm{U}_1^2 \rangle,\\
	\mathrm{GH}_A^{-1}&=&\langle \mathrm{U}_{-1}^4\rangle,\;\mathrm{GH}_A^{0}=\langle \mathrm{U}_{0}^1,\mathrm{U}_{0}^4,\mathrm{U}_0^6\rangle,\; \mathrm{GH}_A^{1}=\langle \mathrm{U}_{1}^4\rangle.
\end{eqnarray*}

\begin{eqnarray*}
	A_{4,6}^{-2\be,\be}:\;\rho&=&{f}^{2}+i{f}^{3}-{f}^{134}+i{f}^{124},\;
	\mathrm{d}\rho=\left( i-\beta \right) {f}^{24}- \left( 1+i\beta \right) {f}^{34},
	\mathrm{Im}\bar\partial_{-1}=\langle\rho\rangle,\\
	\overline{\mathrm{L}}&=&\langle-i{f}^{4}+f_{{1}},
	if_{{2}}+f_{{3}},
	f^1-if_{{4}},
	{f}^{3}+i{f}^{2}
	\rangle,\\
	\mathrm{U}_{-1}&=&\langle \,i{f}^{24}-\,{f}^{34},
	2\,i+\,{f}^{14},
	\,{f}^{12}+\,i{f}^{13},
	-\,{f}^{23}+\,i{f}^{1234}
	\rangle,\\
	\mathrm{U}_0&=&\langle\,{f}^{4},
	{f}^{2}+\,i{f}^{3}+\,{f}^{134}-\,i{f}^{124},
	{f}^{234},
	f^1,
	\,{f}^{2}-\,i{f}^{3}+\,{f}^{134}+\,i{f}^{124},
	{f}^{123}
	\rangle,\\
\mathrm{GH}_\partial^{-1}&=&\langle \mathrm{U}_{-1}^2 \rangle,\;\mathrm{GH}_\partial^{0}=\langle \mathrm{U}_0^1,\mathrm{U}_0^6 \rangle,\;
\mathrm{GH}_\partial^{1}=\langle \mathrm{U}_1^4 \rangle,\\
\mathrm{GH_{BC}^{-1}}&=&\langle \mathrm{U}_{-1}^2\rangle,\;\mathrm{GH_{BC}^{0}}=\langle \mathrm{U}_0^1,\mathrm{U}_0^3,\mathrm{U}_0^6 \rangle,\;\mathrm{GH_{BC}^{1}}=\langle \mathrm{U}_1^2 \rangle,\\
\mathrm{GH}_A^{-1}&=&\langle \mathrm{U}_{-1}^4\rangle,\;\mathrm{GH}_A^{0}=\langle \mathrm{U}_{0}^1,\mathrm{U}_{0}^4,\mathrm{U}_0^6\rangle,\; \mathrm{GH}_A^{1}=\langle \mathrm{U}_{1}^4\rangle.
\end{eqnarray*}

\begin{eqnarray*}
	A_{4,10}:\; \rho&=&{f}^{2}+i{f}^{3}- \left( i\lambda+1 \right) {f}^{134}+ \left( i-
	\lambda \right) {f}^{124},\;\mathrm{d}\rho=i{f}^{24}-{f}^{34},\mathrm{Im}\bar\partial_{-1}=\langle\rho\rangle,\\
	\overline{\mathrm{L}}&=&\langle -i \left( {\lambda}^{2}+1 \right){f}^{4} + \left( 1+i\lambda
	\right)f_{{1}} ,
	if_{{2}}+f_{{3}},
	\left( 1-i\lambda \right)f^1 -if_{{4}},
	{f}^{3}+i{f}^{2}
	\rangle,\\
	\mathrm{U}_{-1}&=&\langle  \left( -\,\lambda+\,i \right) {f}^{24}- \left( \,i\lambda+1
	\right) {f}^{34},
	\,i+ \left( \,i\lambda+1 \right) {f}^{14},
	\,{f}^{12}+\,i{f}^{13},
	-\,{f}^{23}+ \left( -\,\lambda+\,i \right) {f}^{1234}	
	\rangle,\\
	\mathrm{U}_0&=&\langle {f}^{4},
	\left( \,i\lambda+ 1\right) {f}^{2}+ \left( -\,\lambda+\,i \right) 
	{f}^{3}+ \left( \,{\lambda}^{2}+1 \right) {f}^{134}- \left( \,i+\,i
	{\lambda}^{2} \right) {f}^{124},
	{f}^{234},
	f^1,\\&&
	\,{f}^{2}-\,i{f}^{3}+ \left( \,i\lambda+1 \right) {f}^{134}+
	\left( -\,\lambda+\,i \right) {f}^{124},
	{f}^{123}
	\rangle,\\
	\mathrm{GH}_\partial^{-1}&=&\langle \mathrm{U}_{-1}^4 \rangle,\;\mathrm{GH}_\partial^{0}=\langle \mathrm{U}_0^1,\mathrm{U}_0^6 \rangle,\;
	\mathrm{GH}_\partial^{1}=\langle \mathrm{U}_1^2 \rangle,\\
	\mathrm{GH_{BC}^{-1}}&=&\langle \mathrm{U}_{-1}^4\rangle,\;\mathrm{GH_{BC}^{0}}=\langle \mathrm{U}_0^1,\mathrm{U}_0^3,\mathrm{U}_0^6 \rangle,
	\;\mathrm{GH_{BC}^{1}}=\langle \mathrm{U}_1^4 \rangle,\\
	\mathrm{GH}_A^{-1}&=&\langle \mathrm{U}_{-1}^2\rangle,\;\mathrm{GH}_A^{0}=\langle \mathrm{U}_{0}^1,\mathrm{U}_{0}^4,\mathrm{U}_0^6\rangle,\; \mathrm{GH}_A^{1}=\langle \mathrm{U}_{1}^2\rangle.
\end{eqnarray*}

 \begin{eqnarray*}
 	A_2\oplus2A_1:\;\rho&=& f^1+i{f}^{2}+{f}^{234}-i{f}^{134},\;\mathrm{d}\rho=-i{f}^{12}- {f}^{1234},\mathrm{Im}\bar\partial_{-1}=0,\\
 	\overline{\mathrm{L}}&=&\langle f_{{1}}-if_{{2}}, -i{f}^{4}+f_{{3}}, -i{f}^{2}+f^1, {f}^{3}-if_{{4}}\rangle,\\
 	\mathrm{U}_{-1}&=&\langle  {f}^{34}, i{f}^{14}-{f}^{24}, i{f}^{12}+{f}^{1234}, -{f}^{13}-i{f}^{23}\rangle\\
 	\mathrm{U}_0&=&\langle  {f}^{4}, -f^1+i{f}^{2}+{f}^{234}+i{f}^{134}, {f}^{3}, {f}^{124}, f^1+i{f}^{2}-{f}^{234}+i{f}^{134}, {f}^{123} \rangle,\\
 	\mathrm{GH}_\partial^{-1}&=&\langle \mathrm{U}_{-1}^1 \rangle,\;
 	\mathrm{GH}_\partial^{0}=\langle \mathrm{U}_0^1,\mathrm{U}_0^2,\mathrm{U}^3_0 \rangle,\;\mathrm{GH}_\partial^{1}=\langle \mathrm{U}_1^1, \mathrm{U}_{1}^2, \mathrm{U}_{1}^4 \rangle ,\\
 	\;\mathrm{GH_{BC}^{-1}}&=&\langle  \mathrm{U}_{-1}^1, \mathrm{U}_{-1}^3 \rangle,\;\mathrm{GH_{BC}^{0}}=\langle \mathrm{U}_{0}^1,\mathrm{U}_{0}^3,\mathrm{U}_{0}^4,\mathrm{U}_{0}^6 \rangle,\;\mathrm{GH_{BC}^{1}}=\langle \mathrm{U}_{1}^1, \mathrm{U}_{1}^3 \rangle,\\
 	\mathrm{GH}_A^{-1}&=&\langle \mathrm{U}_{-1}^1, \mathrm{U}_{-1}^2, \mathrm{U}_{-1}^4  \rangle,\;\mathrm{GH}_A^{0}=\langle \mathrm{U}_{0}^1,\mathrm{U}_{0}^2, \mathrm{U}_{0}^3,\mathrm{U}_{0}^5  \rangle,\; \mathrm{GH}_A^{1}=\langle \mathrm{U}_{1}^1, \mathrm{U}_{1}^2, \mathrm{U}_{1}^4 \rangle.
 \end{eqnarray*}

 \begin{eqnarray*}
 	A_2\oplus2A_1 :\;\rho&=&-{f}^{3}-i{f}^{4}-{f}^{124}+i{f}^{123}  ,\;\mathrm{d}\rho=0,\mathrm{Im}\bar\partial_{-1}=\langle\rho\rangle,\\
 	\overline{\mathrm{L}}&=&\langle -i{f}^{2}+f_{{1}}, if^1+f_{{2}},  f_{{3}}-if_{{4}}, -i{f}^{4}+{f}^{3} \rangle,\\
 	\mathrm{U}_{-1}&=&\langle  i{f}^{23}-{f}^{24}, -i{f}^{13}+{f}^{14}, -1+i{f}^{12}, -i{f}^{34}-{f}^{1234}  \rangle\\
 	\mathrm{U}_0&=&\langle -i{f}^{3}+{f}^{4}+i{f}^{124}+{f}^{123}, {f}^{2}, {f}^{234}, f^1, {f}^{134}, {f}^{3}-i{f}^{4}-{f}^{124}-i{f}^{123}  \rangle,\\
 	\mathrm{GH}_\partial^{-1}&=&\langle \mathrm{U}_{-1}^2,\mathrm{U}_{-1}^3,\mathrm{U}_{-1}^4 \rangle,\;
 	\mathrm{GH}_\partial^{0}=\langle \mathrm{U}_{0}^4, \mathrm{U}_{0}^5, \mathrm{U}_{0}^6  \rangle,\;\mathrm{GH}_\partial^{1}=\langle \mathrm{U}_{1}^2\rangle ,\\
 	\;\mathrm{GH_{BC}^{-1}}&=&\langle \mathrm{U}_{-1}^2,\mathrm{U}_{-1}^3,\mathrm{U}_{-1}^4 \rangle,\;\mathrm{GH_{BC}^{0}}=\langle \mathrm{U}_{0}^1, \mathrm{U}_{0}^4, \mathrm{U}_{0}^5, \mathrm{U}_{0}^6  \rangle,\;\mathrm{GH_{BC}^{1}}=\langle \mathrm{U}_{1}^2, \mathrm{U}_{1}^3, \mathrm{U}_{1}^4\rangle,\\
 	\mathrm{GH}_A^{-1}&=&\langle \mathrm{U}_{-1}^1,\mathrm{U}_{-1}^2 \rangle,\;\mathrm{GH}_A^{0}=\langle \mathrm{U}_{0}^2, \mathrm{U}_{0}^3, \mathrm{U}_{0}^4, \mathrm{U}_{0}^5 \rangle,\; \mathrm{GH}_A^{1}=\langle \mathrm{U}_{1}^1, \mathrm{U}_{1}^2 \rangle.
 \end{eqnarray*}

 \begin{eqnarray*}
 	2A_2 :\;\rho&=&i{f}^{3}-{f}^{4}-i{f}^{134}+i{f}^{124}+{f}^{123} ,\;\mathrm{d}\rho={f}^{34}-i{f}^{1234},\mathrm{Im}\bar\partial_{-1}=\langle\rho\rangle,\\
 	\overline{\mathrm{L}}&=&\langle i{f}^{3}-i{f}^{2}+f_{{1}}, if^1+f_{{2}}, -if^1+f_{{3}}-if_{{4}},-i{f}^{4}+{f}^{3}    \rangle,\\
 	\mathrm{U}_{-1}&=&\langle {f}^{23}+i{f}^{24}-i{f}^{34}, {f}^{13}+i{f}^{14}, i+{f}^{12}+i{f}^{14}, {f}^{34}-i{f}^{1234}\rangle\\
 	\mathrm{U}_0&=&\langle -{f}^{3}-i{f}^{4}+{f}^{124}-i{f}^{123}, {f}^{2}+{f}^{124}, {f}^{234}, f^1, {f}^{134}, -{f}^{3}+i{f}^{4}+{f}^{124}+i{f}^{123} \rangle,\\
 	\mathrm{GH}_\partial^{-1}&=&\langle \mathrm{U}_{-1}^2-\mathrm{U}_{-1}^3 \rangle,\;\mathrm{GH}_\partial^{0}=\langle\mathrm{U}_{0}^4, \mathrm{U}_{0}^6  \rangle,\;\mathrm{GH}_\partial^{1}=\langle \mathrm{U}_{1}^2\rangle ,\\
 	\;\mathrm{GH_{BC}^{-1}}&=&\langle \mathrm{U}_{-1}^2-\mathrm{U}_{-1}^3 \rangle,\;\mathrm{GH_{BC}^{0}}=\langle \mathrm{U}_{0}^4, 2\mathrm{U}_{0}^3-i\mathrm{U}_{0}^1-i\mathrm{U}_{0}^6 \rangle,\;\mathrm{GH_{BC}^{1}}=\langle \mathrm{U}_{1}^2-\mathrm{U}_{1}^3 \rangle,\\
 	\mathrm{GH}_A^{-1}&=&\langle \mathrm{U}_{-1}^2 \rangle,\;\mathrm{GH}_A^{0}=\langle  \mathrm{U}_{0}^3, \mathrm{U}_{0}^4 \rangle,\; \mathrm{GH}_A^{1}=\langle \mathrm{U}_{1}^2\rangle.
 \end{eqnarray*}

 \begin{eqnarray*}
 	2A_2    :\;\rho&=&i{f}^{3}-{f}^{4}+i{f}^{124}+{f}^{123} ,\;\mathrm{d}\rho={f}^{34}-i{f}^{1234},\mathrm{Im}\bar\partial_{-1}=\langle\rho\rangle,\\
 	\overline{\mathrm{L}}&=&\langle -i{f}^{2}+f_{{1}}, if^1+f_{{2}}, f_{{3}}-if_{{4}}, -i{f}^{4}+{f}^{3} \rangle,\\
 	\mathrm{U}_{-1}&=&\langle  {f}^{23}+i{f}^{24}, -{f}^{13}-i{f}^{14}, i+{f}^{12}, i{f}^{1234}-{f}^{34}  \rangle\\
 	\mathrm{U}_0&=&\langle  -{f}^{3}-i{f}^{4}+{f}^{124}-i{f}^{123}, {f}^{2}, {f}^{234}, f^1, {f}^{134}, -i{f}^{3}-{f}^{4}+i{f}^{124}-{f}^{123} \rangle,\\
 	\mathrm{GH}_\partial^{-1}&=&\langle \mathrm{U}_{-1}^3 \rangle,\;
 	\mathrm{GH}_\partial^{0}=\langle \mathrm{U}_{0}^4, \mathrm{U}_{0}^6 \rangle,\;\mathrm{GH}_\partial^{1}=\langle \mathrm{U}_1^2\rangle ,\\
 	\;\mathrm{GH_{BC}^{-1}}&=&\langle \mathrm{U}_{-1}^3 \rangle,\;\mathrm{GH_{BC}^{0}}=\langle \mathrm{U}_{0}^4, \mathrm{U}_{0}^5, 2\mathrm{U}_{0}^3-\mathrm{U}_{0}^6-i\mathrm{U}_{0}^1 \rangle,\;\mathrm{GH_{BC}^{1}}=\langle \mathrm{U}_{1}^3 \rangle,\\
 	\mathrm{GH}_A^{-1}&=&\langle \mathrm{U}_{-1}^2 \rangle,\;\mathrm{GH}_A^{0}=\langle \mathrm{U}_{0}^2, \mathrm{U}_{0}^3,\mathrm{U}_{0}^4 \rangle ,\; \mathrm{GH}_A^{1}=\langle \mathrm{U}_{1}^2 \rangle.
 \end{eqnarray*}

 \begin{eqnarray*}
 	A_{33}\oplus A_1 :\;\rho&=&if^1-{f}^{2}+i{f}^{234}+{f}^{134}    ,\;\mathrm{d}\rho=-i{f}^{13}+{f}^{23},\mathrm{Im}\bar\partial_{-1}=\langle\rho\rangle,\\
 	\overline{\mathrm{L}}&=&\langle  f_{{1}}-if_{{2}}, -i{f}^{4}+f_{{3}}, i{f}^{3}+f_{{4}}, -i{f}^{2}+f^1 \rangle,\\
 	\mathrm{U}_{-1}&=&\langle i+{f}^{34}, -{f}^{14}-i{f}^{24}, {f}^{13}+i{f}^{23}, i{f}^{1234}-{f}^{12} \rangle\\
 	\mathrm{U}_0&=&\langle  {f}^{4}, {f}^{3}, -if^1-{f}^{2}+i{f}^{234}-{f}^{134}, -f^1-i{f}^{2}+{f}^{234}-i{f}^{134}, {f}^{124}, {f}^{123}  \rangle,\\
 	\mathrm{GH}_\partial^{-1}&=&\langle \mathrm{U}_{-1}^1 \rangle,\;
 	\mathrm{GH}_\partial^{0}=\langle \mathrm{U}_{0}^1, \mathrm{U}_{0}^2  \rangle,\;\mathrm{GH}_\partial^{1}=\langle \mathrm{U}_{1}^1\rangle ,\\
 	\;\mathrm{GH_{BC}^{-1}}&=&\langle \mathrm{U}_{-1}^1 \rangle,\;\mathrm{GH_{BC}^{0}}=\langle\mathrm{U}_{0}^1, \mathrm{U}_{0}^2 \rangle,\;\mathrm{GH_{BC}^{1}}=\langle \mathrm{U}_{1}^1 \rangle,\\
 	\mathrm{GH}_A^{-1}&=&\langle \mathrm{U}_{-1}^1\rangle,\;\mathrm{GH}_A^{0}=\langle \mathrm{U}_{0}^1, \mathrm{U}_{0}^2\rangle,\; \mathrm{GH}_A^{1}=\langle \mathrm{U}_{1}^1 \rangle.
 \end{eqnarray*}
 
 \begin{eqnarray*}
 	A_{37}^\al\oplus A_1 :\;\rho&=&if^1-{f}^{2}+i{f}^{234}+{f}^{134}  ,\;\mathrm{d}\rho=-(1+i\al)({f}^{13}+i{f}^{23}),\mathrm{Im}\bar\partial_{-1}=\langle\rho\rangle,\\
 	\overline{\mathrm{L}}&=&\langle f_{{1}}-if_{{2}}, -i{f}^{4}+f_{{3}}, i{f}^{3}+f_{{4}}, -i{f}^{2}+f^1  \rangle,\\
 	\mathrm{U}_{-1}&=&\langle i+{f}^{34}, -{f}^{14}-i{f}^{24}, {f}^{13}+i{f}^{23}, i{f}^{1234}-{f}^{12} \rangle\\
 	\mathrm{U}_0&=&\langle  {f}^{4}, {f}^{3}, -if^1-{f}^{2}+i{f}^{234}-{f}^{134}, -f^1-i{f}^{2}+{f}^{234}-i{f}^{134}, i{f}^{124}, -i{f}^{123}  \rangle,\\
 	\mathrm{GH}_\partial^{-1}&=&\langle \mathrm{U}_{-1}^1\rangle,\;
 	\mathrm{GH}_\partial^{0}=\langle \mathrm{U}_{0}^1 , \mathrm{U}_{0}^2   \rangle,\;\mathrm{GH}_\partial^{1}=\langle \mathrm{U}_{1}^1 \rangle ,\\
 	\;\mathrm{GH_{BC}^{-1}}&=&\langle \mathrm{U}_{-1}^1 \rangle,\;\mathrm{GH_{BC}^{0}}=\langle\mathrm{U}_{0}^1 , \mathrm{U}_{0}^2  \rangle,\;\mathrm{GH_{BC}^{1}}=\langle \mathrm{U}_{1}^1 \rangle,\\
 	\mathrm{GH}_A^{-1}&=&\langle \mathrm{U}_{-1}^1 \rangle,\;\mathrm{GH}_A^{0}=\langle \mathrm{U}_{0}^1 , \mathrm{U}_{0}^2  \rangle,\; \mathrm{GH}_A^{1}=\langle \mathrm{U}_{1}^1\rangle.
 \end{eqnarray*}

 \begin{eqnarray*}
 	A_{4,2}^{-1}   :\;\rho&=&-i{f}^{3}+{f}^{4}+(\la-i) {f}^{124}+i( i-\la){f}^{123}  ,\;\mathrm{d}\rho=i{f}^{34} +(1+i\la){f}^{1234},\mathrm{Im}\bar\partial_{-1}=0,\\
 	\overline{\mathrm{L}}&=&\langle f_{{3}}-if_{{4}}, ( 1-i\la)f^1 -if_{{2}},  ( 1-i\la){f}^{2} +if_{{1}}, -i{f}^{4}+{f}^{3}  \rangle,\\
 	\mathrm{U}_{-1}&=&\langle -i+(-1-i\la){f}^{12}, -i{f}^{13}+{f}^{14}, -i{f}^{23}+{f}^{24}, {f}^{34}+(-i+\la){f}^{1234} \rangle\\
 	\mathrm{U}_0&=&\langle f^1, {f}^{2}, i{f}^{3}+{f}^{4}+(-i+\la) {f}^{124}+(1+i\la){f}^{123}, -{f}^{3}-i{f}^{4}+(-i\la+1) {f}^{124}+(-\la-i){f}^{123}, {f}^{134}, {f}^{234} \rangle,\\
 	\mathrm{GH}_\partial^{-1}&=&\langle \mathrm{U}_{-1}^2\rangle,\;
 	\mathrm{GH}_\partial^{0}=\langle \mathrm{U}_{0}^3 \rangle,\;\mathrm{GH}_\partial^{1}=\langle \mathrm{U}_{1}^1 \rangle ,\\
 	\;\mathrm{GH_{BC}^{-1}}&=&\langle \mathrm{U}_{-1}^2, \mathrm{U}_{-1}^4  \rangle,\;\mathrm{GH_{BC}^{0}}=\langle\mathrm{U}_{0}^5\rangle,\;\mathrm{GH_{BC}^{1}}=\langle \mathrm{U}_{1}^2, \mathrm{U}_{1}^4 \rangle,\\
 	\mathrm{GH}_A^{-1}&=&\langle \mathrm{U}_{-1}^1\rangle,\;\mathrm{GH}_A^{0}=\langle \mathrm{U}_{0}^1 , \mathrm{U}_{0}^3 , \mathrm{U}_{0}^4  \rangle,\; \mathrm{GH}_A^{1}=\langle \mathrm{U}_{1}^1  \rangle.
 \end{eqnarray*}
 
 \begin{eqnarray*}
 	A_{4,2}^{1} :\;\rho&=&if^1+{f}^{3}+i{f}^{234}+{f}^{124} ,\;\mathrm{d}\rho=-i{f}^{14}-{f}^{34},\mathrm{Im}\bar\partial_{-1}=\langle\rho\rangle,\\
 	\overline{\mathrm{L}}&=&\langle f_{{1}}+if_{{3}}, -i{f}^{4}+f_{{2}}, i{f}^{2}+f_{{4}}, i{f}^{3}+f^1 \rangle,\\
 	\mathrm{U}_{-1}&=&\langle i+{f}^{24}, -{f}^{14}+i{f}^{34}, {f}^{12}+i{f}^{23}, i{f}^{1234}+{f}^{13}\rangle\\
 	\mathrm{U}_0&=&\langle {f}^{4}, {f}^{2}, -if^1+{f}^{3}+i{f}^{234}-{f}^{124}, -f^1+i{f}^{3}+{f}^{234}-i{f}^{124}, {f}^{134}, {f}^{123}   \rangle,\\
 	\mathrm{GH}_\partial^{-1}&=&\langle  \mathrm{U}_{-1}^1 \rangle,\;
 	\mathrm{GH}_\partial^{0}=\langle \mathrm{U}_{0}^1   \rangle,\;\mathrm{GH}_\partial^{1}=\langle0\rangle ,\\
 	\;\mathrm{GH_{BC}^{-1}}&=&\langle \mathrm{U}_{-1}^1 \rangle,\;\mathrm{GH_{BC}^{0}}=\langle\mathrm{U}_{0}^1 \rangle,\;\mathrm{GH_{BC}^{1}}=\langle \mathrm{U}_{1}^1 \rangle,\\
 	\mathrm{GH}_A^{-1}&=&\langle0 \rangle,\;\mathrm{GH}_A^{0}=\langle\mathrm{U}_{0}^1 , \mathrm{U}_{0}^2  \rangle,\; \mathrm{GH}_A^{1}=\langle 0 \rangle.
 \end{eqnarray*}
 
 \begin{eqnarray*}
 	A_{4,3} :\;\rho&=&  f^1+i{f}^{4}+{f}^{234}-i{f}^{123},\;\mathrm{d}\rho=-{f}^{14}+i{f}^{1234},\mathrm{Im}\bar\partial_{-1}=0,\\
 	\overline{\mathrm{L}}&=&\langle f_{{1}}-if_{{4}}, -i{f}^{3}+f_{{2}}, i{f}^{2}+f_{{3}}, {f}^{4}+if^1  \rangle,\\
 	\mathrm{U}_{-1}&=&\langle 1-i{f}^{23}, i{f}^{13}+{f}^{34}, -i{f}^{12}-{f}^{24}, -{f}^{14}+i{f}^{1234} \rangle\\
 	\mathrm{U}_0&=&\langle  {f}^{3}, {f}^{2}, -if^1-{f}^{4}+i{f}^{234}-{f}^{123}, if^1-{f}^{4}-i{f}^{234}-{f}^{123}, {f}^{134}, {f}^{124}  \rangle,\\
 	\mathrm{GH}_\partial^{-1}&=&\langle \mathrm{U}_{-1}^1 \rangle,\;
 	\mathrm{GH}_\partial^{0}=\langle \mathrm{U}_{0}^1, \mathrm{U}_{0}^3  \rangle,\;\mathrm{GH}_\partial^{1}=\langle \mathrm{U}_{1}^1, \mathrm{U}_{1}^3\rangle ,\\
 	\;\mathrm{GH_{BC}^{-1}}&=&\langle \mathrm{U}_{-1}^1, \mathrm{U}_{-1}^4 \rangle,\;\mathrm{GH_{BC}^{0}}=\langle\mathrm{U}_{0}^1, \mathrm{U}_{0}^6\rangle,\;\mathrm{GH_{BC}^{1}}=\langle \mathrm{U}_{1}^1, \mathrm{U}_{1}^4 \rangle,\\
 	\mathrm{GH}_A^{-1}&=&\langle \mathrm{U}_{-1}^1, \mathrm{U}_{-1}^3 \rangle,\;\mathrm{GH}_A^{0}=\langle \mathrm{U}_{0}^1, \mathrm{U}_{0}^3, \mathrm{U}_{0}^4 \rangle,\; \mathrm{GH}_A^{1}=\langle \mathrm{U}_{1}^1, \mathrm{U}_{1}^3 \rangle.
 \end{eqnarray*}

 \begin{eqnarray*}
 	A_{4,5}^{-\al,\al},\;\al\not=1:\;\rho&=&f^1+i{f}^{4}+{f}^{234}-i{f}^{123} ,\;\mathrm{d}\rho=-{f}^{14}+i{f}^{1234},\mathrm{Im}\bar\partial_{-1}=0,\\
 	\overline{\mathrm{L}}&=&\langle f_{{1}}-if_{{4}}, -i{f}^{3}+f_{{2}}, -i{f}^{4}+f^1, {f}^{2}-if_{{3}}    \rangle,\\
 	\mathrm{U}_{-1}&=&\langle 1-i{f}^{23}, i{f}^{13}+{f}^{34}, i{f}^{14}+{f}^{1234}, -{f}^{12}+i{f}^{24} \rangle\\
 	\mathrm{U}_0&=&\langle  {f}^{3}, -f^1+i{f}^{4}+{f}^{234}+i{f}^{123}, {f}^{2}, {f}^{134}, f^1+i{f}^{4}-{f}^{234}+i{f}^{123}, {f}^{124} \rangle,\\
 	\mathrm{GH}_\partial^{-1}&=&\langle \mathrm{U}_{-1}^1 \rangle,\;
 	\mathrm{GH}_\partial^{0}=\langle \mathrm{U}_{0}^2  \rangle,\;\mathrm{GH}_\partial^{1}=\langle \mathrm{U}_{1}^1  \rangle ,\\
 	\;\mathrm{GH_{BC}^{-1}}&=&\langle \mathrm{U}_{-1}^1, \mathrm{U}_{-1}^3  \rangle,\;\mathrm{GH_{BC}^{0}}=\langle  0 \rangle,\;\mathrm{GH_{BC}^{1}}=\langle  \mathrm{U}_{1}^1,  \mathrm{U}_{1}^3 \rangle,\\
 	\mathrm{GH}_A^{-1}&=&\langle \mathrm{U}_{-1}^1\rangle,\;\mathrm{GH}_A^{0}=\langle  \mathrm{U}_{0}^2,   \mathrm{U}_{0}^5  \rangle,\; \mathrm{GH}_A^{1}=\langle  \mathrm{U}_{1}^1  \rangle.
 \end{eqnarray*}
 
 \begin{eqnarray*}
 	A_{4,5}^{-1,1} :\;\rho&=&f^1+i{f}^{4}+{f}^{234}-i{f}^{123}    ,\;\mathrm{d}\rho=-{f}^{14}+i{f}^{1234},\mathrm{Im}\bar\partial_{-1}=0,\\
 	\overline{\mathrm{L}}&=&\langle f_{{1}}-if_{{4}}, -i{f}^{3}+f_{{2}}, -i{f}^{4}+f^1, {f}^{2}-if_{{3}}    \rangle,\\
 	\mathrm{U}_{-1}&=&\langle 1-i{f}^{23}, i{f}^{13}+{f}^{34}, i{f}^{14}+{f}^{1234}, -{f}^{12}+i{f}^{24} \rangle\\
 	\mathrm{U}_0&=&\langle  {f}^{3}, -f^1+i{f}^{4}+{f}^{234}+i{f}^{123}, {f}^{2}, {f}^{134}, f^1+i{f}^{4}-{f}^{234}+i{f}^{123}, {f}^{124} \rangle,\\
 	\mathrm{GH}_\partial^{-1}&=&\langle \mathrm{U}_{-1}^1, \mathrm{U}_{-1}^4 \rangle,\;
 	\mathrm{GH}_\partial^{0}=\langle \mathrm{U}_{0}^2,  \mathrm{U}_{0}^6  \rangle,\;\mathrm{GH}_\partial^{1}=\langle \mathrm{U}_{1}^1  \rangle ,\\
 	\;\mathrm{GH_{BC}^{-1}}&=&\langle \mathrm{U}_{-1}^1, \mathrm{U}_{-1}^3, \mathrm{U}_{-1}^4  \rangle,\;\mathrm{GH_{BC}^{0}}=\langle  \mathrm{U}_{0}^6 \rangle,\;\mathrm{GH_{BC}^{1}}=\langle  \mathrm{U}_{1}^1,  \mathrm{U}_{1}^3, \mathrm{U}_{1}^4 \rangle,\\
 	\mathrm{GH}_A^{-1}&=&\langle \mathrm{U}_{-1}^1\rangle,\;\mathrm{GH}_A^{0}=\langle  \mathrm{U}_{0}^2,  \mathrm{U}_{0}^3,  \mathrm{U}_{0}^5,  \mathrm{U}_{0}^6  \rangle,\; \mathrm{GH}_A^{1}=\langle  \mathrm{U}_{1}^1  \rangle.
 \end{eqnarray*}

 \begin{eqnarray*}
 	A_{4,5}^{-1,\be} \be\not=-1 :\;\rho&=&{f}^{3}+i{f}^{4}+{f}^{124}-i{f}^{123}    ,\;\mathrm{d}\rho=-\be{f}^{34}+i\be{f}^{1234},\mathrm{Im}\bar\partial_{-1}=0,\\
 	\overline{\mathrm{L}}&=&\langle -i{f}^{2}+f_{{1}}, if^1+f_{{2}}, f_{{3}}-if_{{4}}, {f}^{4}+i{f}^{3}  \rangle,\\
 	\mathrm{U}_{-1}&=&\langle -i{f}^{23}+{f}^{24}, i{f}^{13}-{f}^{14}, -i{f}^{12}, -{f}^{34}+i{f}^{1234}  \rangle\\
 	\mathrm{U}_0&=&\langle i{f}^{3}-{f}^{4}-i{f}^{124}-{f}^{123}, -i{f}^{2}, {f}^{234}, f^1, {f}^{134}, -i{f}^{3}-{f}^{4}+i{f}^{124}-{f}^{123} \rangle,\\
 	\mathrm{GH}_\partial^{-1}&=&\langle \mathrm{U}_{-1}^3\rangle,\;
 	\mathrm{GH}_\partial^{0}=\langle \mathrm{U}_{0}^6   \rangle,\;\mathrm{GH}_\partial^{1}=\langle \mathrm{U}_{1}^3 \rangle ,\\
 	\;\mathrm{GH_{BC}^{-1}}&=&\langle \mathrm{U}_{-1}^3, \mathrm{U}_{-1}^4 \rangle,\;\mathrm{GH_{BC}^{0}}=\langle 0 \rangle,\;\mathrm{GH_{BC}^{1}}=\langle \mathrm{U}_{1}^3, \mathrm{U}_{1}^4 \rangle,\\
 	\mathrm{GH}_A^{-1}&=&\langle  \mathrm{U}_{-1}^3 \rangle,\;\mathrm{GH}_A^{0}=\langle \mathrm{U}_{0}^1,  \mathrm{U}_{0}^6  \rangle,\; \mathrm{GH}_A^{1}=\langle \mathrm{U}_{1}^3  \rangle.
 \end{eqnarray*}

 
 \begin{eqnarray*}
 	A_{4,5}^{\al,\al},\; \al\not=-1 :\;\rho&=& {f}^{2}+i{f}^{3}-{f}^{134}+i{f}^{124} ,\;\mathrm{d}\rho=-\al{f}^{24}-i\al{f}^{34},\mathrm{Im}\bar\partial_{-1}=\langle\rho\rangle,\\
 	\overline{\mathrm{L}}&=&\langle -i{f}^{4}+f_{{1}}, f_{{2}}-if_{{3}}, if^1+f_{{4}}, {f}^{3}+i{f}^{2} \rangle,\\
 	\mathrm{U}_{-1}&=&\langle i{f}^{24}-{f}^{34}, 1-i{f}^{14}, i{f}^{12}-{f}^{13}, -{f}^{23}+i{f}^{1234}  \rangle\\
 	\mathrm{U}_0&=&\langle  -i{f}^{4}, i{f}^{2}-{f}^{3}+i{f}^{134}+{f}^{124}, i{f}^{234}, -if^1, -i{f}^{2}-{f}^{3}-i{f}^{134}+{f}^{124}, -i{f}^{123}  \rangle,\\
 	\mathrm{GH}_\partial^{-1}&=&\langle \mathrm{U}_{-1}^2\rangle,\;
 	\mathrm{GH}_\partial^{0}=\langle \mathrm{U}_{0}^1  \rangle,\;\mathrm{GH}_\partial^{1}=\langle 0 \rangle ,\\
 	\;\mathrm{GH_{BC}^{-1}}&=&\langle \mathrm{U}_{-1}^2 \rangle,\;\mathrm{GH_{BC}^{0}}=\langle\mathrm{U}_{0}^1 \rangle,\;\mathrm{GH_{BC}^{1}}=\langle \mathrm{U}_{1}^2 \rangle,\\
 	\mathrm{GH}_A^{-1}&=&\langle 0 \rangle,\;\mathrm{GH}_A^{0}=\langle \mathrm{U}_{0}^1, \mathrm{U}_{0}^4   \rangle,\; \mathrm{GH}_A^{1}=\langle 0  \rangle.
 \end{eqnarray*}

 \begin{eqnarray*}
 	A_{4,5}^{\al,1},\;\al\not=-1 :\;\rho&=&f^1-i{f}^{3}+{f}^{234}-i{f}^{124}  ,\;\mathrm{d}\rho=-{f}^{14}+i{f}^{34},\mathrm{Im}\bar\partial_{-1}=\langle\rho\rangle,\\
 	\overline{\mathrm{L}}&=&\langle f_{{1}}+if_{{3}}, -i{f}^{4}+f_{{2}}, i{f}^{2}+f_{{4}}, {f}^{3}-if^1  \rangle,\\
 	\mathrm{U}_{-1}&=&\langle 1-i{f}^{24}, i{f}^{14}+{f}^{34}, -i{f}^{12}+{f}^{23}, -{f}^{13}-i{f}^{1234} \rangle\\
 	\mathrm{U}_0&=&\langle  {f}^{4}, {f}^{2}, if^1-{f}^{3}-i{f}^{234}+{f}^{124}, if^1+{f}^{3}-i{f}^{234}-{f}^{124}, {f}^{134}, {f}^{123}
 	\rangle,\\
 	\mathrm{GH}_\partial^{-1}&=&\langle \mathrm{U}_{-1}^1\rangle,\;
 	\mathrm{GH}_\partial^{0}=\langle \mathrm{U}_{0}^1   \rangle,\;\mathrm{GH}_\partial^{1}=\langle  0  \rangle ,\\
 	\;\mathrm{GH_{BC}^{-1}}&=&\langle \mathrm{U}_{-1}^1 \rangle,\;\mathrm{GH_{BC}^{0}}=\langle\mathrm{U}_{0}^1 \rangle,\;\mathrm{GH_{BC}^{1}}=\langle \mathrm{U}_{1}^1 \rangle,\\
 	\mathrm{GH}_A^{-1}&=&\langle0 \rangle,\;\mathrm{GH}_A^{0}=\langle \mathrm{U}_{0}^1, \mathrm{U}_{0}^2  \rangle,\; \mathrm{GH}_A^{1}=\langle 0 \rangle.
 \end{eqnarray*}

 \begin{eqnarray*}
 	A_{4,5}^{-1,1} :\;\rho&=&f^1-i{f}^{3}+{f}^{234}-i{f}^{124}  ,\;\mathrm{d}\rho=-{f}^{14}+i{f}^{34},\mathrm{Im}\bar\partial_{-1}=0,\\
 	\overline{\mathrm{L}}&=&\langle f_{{1}}+if_{{3}}, -i{f}^{4}+f_{{2}}, i{f}^{2}+f_{{4}}, {f}^{3}-if^1  \rangle,\\
 	\mathrm{U}_{-1}&=&\langle 1-i{f}^{24}, i{f}^{14}+{f}^{34}, -i{f}^{12}+{f}^{23}, -{f}^{13}-i{f}^{1234} \rangle\\
 	\mathrm{U}_0&=&\langle  {f}^{4}, {f}^{2}, if^1-{f}^{3}-i{f}^{234}+{f}^{124}, if^1+{f}^{3}-i{f}^{234}-{f}^{124}, {f}^{134}, {f}^{123}
 	\rangle,\\
 	\mathrm{GH}_\partial^{-1}&=&\langle \mathrm{U}_{-1}^1, \mathrm{U}_{-1}^3\rangle,\;
 	\mathrm{GH}_\partial^{0}=\langle \mathrm{U}_{0}^1 , \mathrm{U}_{0}^4   \rangle,\;\mathrm{GH}_\partial^{1}=\langle \mathrm{U}_{1}^3  \rangle ,\\
 	\;\mathrm{GH_{BC}^{-1}}&=&\langle \mathrm{U}_{-1}^1, \mathrm{U}_{-1}^2,  \mathrm{U}_{-1}^3 \rangle,\;\mathrm{GH_{BC}^{0}}=\langle\mathrm{U}_{0}^1\rangle,\;\mathrm{GH_{BC}^{1}}=\langle \mathrm{U}_{1}^1, \mathrm{U}_{1}^2, \mathrm{U}_{1}^3 \rangle,\\
 	\mathrm{GH}_A^{-1}&=&\langle \mathrm{U}_{-1}^3 \rangle,\;\mathrm{GH}_A^{0}=\langle \mathrm{U}_{0}^1 , \mathrm{U}_{0}^2, \mathrm{U}_{0}^3, \mathrm{U}_{0}^4, \rangle,\; \mathrm{GH}_A^{1}=\langle \mathrm{U}_{1}^3 \rangle.
 \end{eqnarray*}
 
 \begin{eqnarray*}
 	A_{4,6}^{\al,\be}  :\;\rho&=& {f}^{2}-i{f}^{3}+{f}^{134}+i{f}^{124} ,\;\mathrm{d}\rho=(-1+i\be)(i{f}^{24}+ {f}^{34}),\mathrm{Im}\bar\partial_{-1}=\langle\rho\rangle,\\
 	\overline{\mathrm{L}}&=&\langle -i{f}^{4}+f_{{1}}, f_{{2}}+if_{{3}}, if^1+f_{{4}}, {f}^{3}-i{f}^{2}   \rangle,\\
 	\mathrm{U}_{-1}&=&\langle i{f}^{24}+{f}^{34}, 1-i{f}^{14}, i{f}^{12}+{f}^{13}, -{f}^{23}+i{f}^{1234} \rangle\\
 	\mathrm{U}_0&=&\langle {f}^{4}, i{f}^{2}+{f}^{3}-i{f}^{134}+{f}^{124}, {f}^{234}
 	, f^1, i{f}^{2}-{f}^{3}-i{f}^{134}-{f}^{124}, {f}^{123}  \rangle,\\
 	\mathrm{GH}_\partial^{-1}&=&\langle \mathrm{U}_{-1}^2\rangle,\;
 	\mathrm{GH}_\partial^{0}=\langle \mathrm{U}_{0}^1  \rangle,\;\mathrm{GH}_\partial^{1}=\langle0  \rangle ,\\
 	\;\mathrm{GH_{BC}^{-1}}&=&\langle \mathrm{U}_{-1}^2 \rangle,\;\mathrm{GH_{BC}^{0}}=\langle\mathrm{U}_{0}^1\rangle,\;\mathrm{GH_{BC}^{1}}=\langle \mathrm{U}_{1}^2\rangle,\\
 	\mathrm{GH}_A^{-1}&=&\langle0 \rangle,\;\mathrm{GH}_A^{0}=\langle \mathrm{U}_{0}^1 , \mathrm{U}_{0}^4  \rangle,\; \mathrm{GH}_A^{1}=\langle 0  \rangle.
 \end{eqnarray*}
 
 \begin{eqnarray*}
 	A_{4,6}^{\al,0}:\;\rho&=&f^1+i{f}^{4}+{f}^{234}-i{f}^{123} ,\;\mathrm{d}\rho=\al(-{f}^{14}+i{f}^{1234}),\mathrm{Im}\bar\partial_{-1}=0,\\
 	\overline{\mathrm{L}}&=&\langle f_{{1}}-if_{{4}}, {f}^{2}-if_{{3}}, {f}^{3}+if_{{2}}, {f}^{4}+if^1   \rangle,\\
 	\mathrm{U}_{-1}&=&\langle  1-i{f}^{23}, -{f}^{12}+i{f}^{24}, -{f}^{13}+i{f}^{34}, -{f}^{14}+i{f}^{1234}  \rangle\\
 	\mathrm{U}_0&=&\langle  {f}^{2}, {f}^{3}, -if^1-{f}^{4}+i{f}^{234}-{f}^{123}, -if^1+2{f}^{4}+i{f}^{234}+{f}^{123}, {f}^{124}, {f}^{134}  \rangle,\\
 	\mathrm{GH}_\partial^{-1}&=&\langle \mathrm{U}_{-1}^1\rangle,\;
 	\mathrm{GH}_\partial^{0}=\langle {\mathrm{U}_{0}^3}
 	 \rangle,\;\mathrm{GH}_\partial^{1}=\langle \mathrm{U}_{1}^1  \rangle ,\\
 	\;\mathrm{GH_{BC}^{-1}}&=&\langle \mathrm{U}_{-1}^1,  \mathrm{U}_{-1}^4 \rangle,\;\mathrm{GH_{BC}^{0}}=\langle 0 \rangle,\;\mathrm{GH_{BC}^{1}}=\langle \mathrm{U}_{1}^1, \mathrm{U}_{1}^4 \rangle,\\
 	\mathrm{GH}_A^{-1}&=&\langle \mathrm{U}_{-1}^1 \rangle,\;\mathrm{GH}_A^{0}=\langle \mathrm{U}_{0}^3 , \mathrm{U}_{0}^4  \rangle,\; \mathrm{GH}_A^{1}=\langle \mathrm{U}_{1}^1  \rangle.
 \end{eqnarray*}
 
 
 \begin{eqnarray*}
 	A_{4,9}^{-\frac{1}{2}} :\;\rho&=&  {f}^{2}-i{f}^{4}-{f}^{134}+i{f}^{123}  ,\;\mathrm{d}\rho=-{f}^{24}-i{f}^{1234},\mathrm{Im}\bar\partial_{-1}=0,\\
 	\overline{\mathrm{L}}&=&\langle f_{{2}}+if_{{4}},\;if^1+f_{{3}},\;i{f}^{4}+{f}^{2},\;{f}^{3}+if_{{1}}\rangle,\\
 	\mathrm{U}_{-1}&=&\langle 1-i{f}^{13}, 2i{f}^{12}+2{f}^{14}, -i{f}^{24}+{f}^{1234}, -2{f}^{23}-2i{f}^{34}  \rangle\\
 	\mathrm{U}_0&=&\langle  f^1, {f}^{2}+i{f}^{4}+{f}^{134}+i{f}^{123}, {f}^{3}, {f}^{124}, {f}^{2}-i{f}^{4}+{f}^{134}-i{f}^{123}, {f}^{234}  \rangle,\\
 	\mathrm{GH}_\partial^{-1}&=&\langle \mathrm{U}_{-1}^1\rangle,\;
 	\mathrm{GH}_\partial^{0}=\langle \mathrm{U}_{0}^2 \rangle,\;\mathrm{GH}_\partial^{1}=\langle \mathrm{U}_{1}^1  \rangle ,\\
 	\;\mathrm{GH_{BC}^{-1}}&=&\langle \mathrm{U}_{-1}^1, \mathrm{U}_{-1}^3  \rangle,\;\mathrm{GH_{BC}^{0}}=\langle 0 \rangle,\;\mathrm{GH_{BC}^{1}}=\langle \mathrm{U}_{1}^1, \mathrm{U}_{1}^3 \rangle,\\
 	\mathrm{GH}_A^{-1}&=&\langle \mathrm{U}_{-1}^1 \rangle,\;\mathrm{GH}_A^{0}=\langle \mathrm{U}_{0}^2 , \mathrm{U}_{0}^5 \rangle,\; \mathrm{GH}_A^{1}=\langle \mathrm{U}_{1}^1 \rangle.
 \end{eqnarray*}
 
 \begin{eqnarray*}
 	A_{4,9}^{1}   :\;\rho&=&-{f}^{2}-i{f}^{3}+{f}^{134}-i{f}^{124}   ,\;\mathrm{d}\rho={f}^{24}+i{f}^{34},\mathrm{Im}\bar\partial_{-1}=\langle\rho\rangle,\\
 	\overline{\mathrm{L}}&=&\langle -i{f}^{4}+f_{{1}}, f_{{2}}-if_{{3}}, if^1+f_{{4}}, -i{f}^{3}+{f}^{2}   \rangle,\\
 	\mathrm{U}_{-1}&=&\langle -i{f}^{24}+{f}^{34}, -1+i{f}^{14}, -i{f}^{12}+{f}^{13}, -i{f}^{23}-{f}^{1234}  \rangle\\
 	\mathrm{U}_0&=&\langle  {f}^{4}, {f}^{2}+{f}^{3}-i{f}^{134}-{f}^{124}, {f}^{234}, f^1, {f}^{2}-i{f}^{3}+{f}^{134}+i{f}^{124}, {f}^{123} \rangle,\\
 	\mathrm{GH}_\partial^{-1}&=&\langle \mathrm{U}_{-1}^4-2\mathrm{U}_{-1}^2\rangle,\;
 	\mathrm{GH}_\partial^{0}=\langle \mathrm{U}_{0}^1  \rangle,\;\mathrm{GH}_\partial^{1}=\langle 0 \rangle ,\\
 	\;\mathrm{GH_{BC}^{-1}}&=&\langle \mathrm{U}_{-1}^4-2\mathrm{U}_{-1}^2 \rangle,\;\mathrm{GH_{BC}^{0}}=\langle\mathrm{U}_{0}^1\rangle,\;\mathrm{GH_{BC}^{1}}=\langle \mathrm{U}_{1}^4-2\mathrm{U}_{1}^2 \rangle,\\
 	\mathrm{GH}_A^{-1}&=&\langle 0 \rangle,\;\mathrm{GH}_A^{0}=\langle \mathrm{U}_{0}^1 , \mathrm{U}_{0}^6  \rangle,\; \mathrm{GH}_A^{1}=\langle 0  \rangle.
 \end{eqnarray*}
 
 
 \begin{eqnarray*}
 	A_{4,11}^{\al}    :\;\rho&=&{f}^{2}-i{f}^{3}+{f}^{134}+i{f}^{124}   ,\;\mathrm{d}\rho=(i+\al) (-{f}^{24}+i {f}^{34}),\mathrm{Im}\bar\partial_{-1}=\langle\rho\rangle,\\
 	\overline{\mathrm{L}}&=&\langle -i{f}^{4}+f_{{1}}, f_{{2}}+if_{{3}}, if^1+f_{{4}}, {f}^{3}-i{f}^{2}  \rangle,\\
 	\mathrm{U}_{-1}&=&\langle i{f}^{24}+{f}^{34}, 1-i{f}^{14}, i{f}^{12}+{f}^{13}, -{f}^{23}+i{f}^{1234}  \rangle\\
 	\mathrm{U}_0&=&\langle {f}^{4}, i{f}^{2}+{f}^{3}-i{f}^{134}+{f}^{124}, i{f}^{234}, f^1, i{f}^{2}-{f}^{3}-i{f}^{134}-{f}^{124}, {f}^{123} \rangle,\\
 	\mathrm{GH}_\partial^{-1}&=&\langle 4\al i\mathrm{U}_{-1}^2-\mathrm{U}_{-1}^4 \rangle,\;
 	\mathrm{GH}_\partial^{0}=\langle \mathrm{U}_{0}^1 \rangle,\;\mathrm{GH}_\partial^{1}=\langle 0  \rangle ,\\
 	\;\mathrm{GH_{BC}^{-1}}&=&\langle 4\al i\mathrm{U}_{-1}^2-\mathrm{U}_{-1}^4 \rangle,\;\mathrm{GH_{BC}^{0}}=\langle \mathrm{U}_{0}^1  \rangle,\;\mathrm{GH_{BC}^{1}}=\langle 4\al i\mathrm{U}_{1}^2-\mathrm{U}_{1}^4\rangle,\\
 	\mathrm{GH}_A^{-1}&=&\langle 0 \rangle,\;\mathrm{GH}_A^{0}=\langle \mathrm{U}_{0}^1,  8 \mathrm{U}_{0}^4+7i \mathrm{U}_{0}^6 \rangle,\; \mathrm{GH}_A^{1}=\langle 0 \rangle.
 \end{eqnarray*}
 
 \begin{eqnarray*}
 	A_{4,12}    :\;\rho&=& -f^1+i{f}^{2}+{f}^{234}+i{f}^{134}  ,\;\mathrm{d}\rho={f}^{13}+i{f}^{14}-i{f}^{23}+{f}^{24},\mathrm{Im}\bar\partial_{-1}=\langle\rho\rangle,\\
 	\overline{\mathrm{L}}&=&\langle f_{{1}}+if_{{2}}, -i{f}^{4}+f_{{3}}, i{f}^{3}+f_{{4}}, i{f}^{2}+f^1   \rangle,\\
 	\mathrm{U}_{-1}&=&\langle  1-i{f}^{34}, i{f}^{14}+{f}^{24}, i{f}^{13}+{f}^{23}, i{f}^{12}+{f}^{1234} \rangle\\
 	\mathrm{U}_0&=&\langle {f}^{4}, {f}^{3}, f^1+i{f}^{2}+{f}^{234}-i{f}^{134}, if^1+{f}^{2}+i{f}^{234}-{f}^{134}, {f}^{124}, {f}^{123} \rangle,\\
 	\mathrm{GH}_\partial^{-1}&=&\langle \mathrm{U}_{-1}^1\rangle,\;
 	\mathrm{GH}_\partial^{0}=\langle \mathrm{U}_{0}^1 , \mathrm{U}_{0}^2   \rangle,\;\mathrm{GH}_\partial^{1}=\langle \mathrm{U}_{1}^1  \rangle ,\\
 	\;\mathrm{GH_{BC}^{-1}}&=&\langle \mathrm{U}_{-1}^1 \rangle,\;\mathrm{GH_{BC}^{0}}=\langle\mathrm{U}_{0}^1 , \mathrm{U}_{0}^2 \rangle,\;\mathrm{GH_{BC}^{1}}=\langle \mathrm{U}_{1}^1 \rangle,\\
 	\mathrm{GH}_A^{-1}&=&\langle \mathrm{U}_{-1}^1 \rangle,\;\mathrm{GH}_A^{0}=\langle \mathrm{U}_{0}^1 , \mathrm{U}_{0}^2 \rangle,\; \mathrm{GH}_A^{1}=\langle \mathrm{U}_{1}^1  \rangle.
 \end{eqnarray*}

\section{Tables}\label{section5}

{\renewcommand*{\arraystretch}{1.4}
	\begin{center}
		
		\begin{tabular}{|c|l|c|c|c|c|}
			\hline
			Lie algebra   & Nonzero brackets   & Type 2 &Type 1  & Type 0 &Lie algebra in \cite{O1} \\
			\hline
			$A_2\oplus2A_1$& $[f_1,f_2]=f_2$ & Yes& Yes & Yes & $\mathfrak{r}\mathfrak{r}_{3,0}$  \\
			\hline
			$2A_2$& $[f_1,f_2]=f_2$, $[f_3,f_4]=f_4$ & Yes& Yes & Yes &   $\mathfrak{r}_2\mathfrak{r}_2$ \\
			\hline
			$A_{3,2}\oplus A_1$& $[f_1,f_3]=f_1$, $[f_2,f_3]=f_1+f_2$ &No &No  &No & $\mathfrak{r}\mathfrak{r}_{3}$ \\
			\hline
			$A_{3,3}\oplus A_1$& $[f_1,f_3]=f_1$, $[f_2,f_3]=f_2$ & Yes& Yes & No  & $\mathfrak{r}\mathfrak{r}_{3,1}$ \\
			\hline
			$A_{3,5}^\al\oplus A_1$&  \multirow{2}{*}{$[f_1,f_3]=f_1$, $[f_2,f_3]=\al f_2$} &\multirow{2}{*}{No} &\multirow{2}{*}{No}  &  \multirow{2}{*}{No} & $\mathfrak{r}\mathfrak{r}_{3,\al}$  \\
			$0<|\al|<1$&&&&& $0<|\al|<1$\\
			\hline
			$A_{3,7}^\al\oplus A_1$&$[f_1,f_3]=\al f_1-f_2$ &\multirow{2}{*}{No} &\multirow{2}{*}{ Yes}&  \multirow{2}{*}{No}   &$\mathfrak{r}\mathfrak{r}^\prime_{3,\al}$   \\
			$\al>0$& $[f_2,f_3]=f_1+\al f_2$&&&&$\al>0$\\
			\hline
			$A^\al_{4,2}$ $\al\notin\{0,-2\}$& $[f_1,f_4]=\al f_1$, $[f_2,f_4]=f_2$ & \multirow{2}{*}{No}&\multirow{2}{*}{No}  &  \multirow{2}{*}{No} & \multirow{2}{*}{$\mathfrak{r}_{4,\frac{1}{\al}}$} \\
			$\al\not\in\{-1,1\}$&$[f_3,f_4]=f_2+f_3$&&&&\\
			\hline
			\multirow{2}{*}{$A^1_{4,2}$} & $[f_1,f_4]=f_1$, $[f_2,f_4]=f_2$ & \multirow{2}{*}{Yes}&\multirow{2}{*}{Yes}  & \multirow{2}{*}{No} & \multirow{2}{*}{$\mathfrak{r}_{4,1}$} \\
			&$[f_3,f_4]=f_2+f_3$&&&&\\
			\hline
			\multirow{2}{*}{$A^{-1}_{4,2}$}& $[f_1,f_4]=- f_1$, $[f_2,f_4]=f_2$ & \multirow{2}{*}{No}&\multirow{2}{*}{Yes}  & \multirow{2}{*}{Yes} & \multirow{2}{*}{$\mathfrak{r}_{4,-1}$} \\
			&$[f_3,f_4]=f_2+f_3$&&&&\\
			\hline
			$A_{4,3}$&  $[f_1,f_4]=f_1$, $[f_3,f_4]=f_2$ &No & Yes &Yes   &  $\mathfrak{r}_{4,0}$ \\
			\hline
			\multirow{2}{*}{$A_{4,4}$}& $[f_1,f_4]=f_1$, $[f_2,f_4]=f_1+f_2$ &\multirow{2}{*}{No} &\multirow{2}{*}{No} & \multirow{2}{*}{No}  &\multirow{2}{*}{$\mathfrak{r}_{4}$}  \\
			&$[f_3,f_4]=f_2+f_3$&&&&\\
			
			\hline
			$A^{\al,\be}_{4,6}$ $\al\not=0$& $[f_1,f_4]=\al f_1$, $[f_2,f_4]=\be f_2-f_3$ &\multirow{2}{*}{Yes} & \multirow{2}{*}{ Yes} &\multirow{2}{*}{ No}  & \multirow{2}{*}{ $\mathfrak{r}^\prime_{4,\frac{\be}{\al},\frac{1}{\al}}$ }\\
			$\al\not=-2\be$, $\be>0$&$[f_3,f_4]=f_2+\be f_3$&&&&\\
			\hline
			$A^{\al,0}_{4,6}$ & $[f_1,f_4]=\al f_1$, $[f_2,f_4]=-f_3$ &\multirow{2}{*}{Yes} & \multirow{2}{*}{ Yes} & \multirow{2}{*}{Yes}  & \multirow{2}{*}{$\mathfrak{r}^\prime_{4,0,\frac{1}{\al}}$} \\
			$\al\not=0$	&$[f_3,f_4]=f_2$&&&&\\
			\hline
			\multirow{2}{*}{  $A_{4,7}$}& $[f_2,f_3]=f_1$, $[f_1,f_4]=2f_1$ &\multirow{2}{*}{Yes} & \multirow{2}{*}{No} &  \multirow{2}{*}{Yes} & \multirow{2}{*}{$\mathfrak{h}_4$ } \\
			&$[f_2,f_4]=f_2$, $[f_3,f_4]=f_2+f_3$&&&&\\
			\hline
			$A_{4,9}^\be$& $[f_2,f_3]=f_1$, $[f_3,f_4]=\be f_3$ &\multirow{2}{*}{ Yes}& \multirow{1}{*}{ Yes} & \multirow{2}{*}{Yes} & \multirow{2}{*}{$\mathfrak{d}_{4,\frac{1}{1+\be}}$ } \\
			$-1<\be\leq1$ & $[f_1,f_4]=(1+\be)f_1$, $[f_2,f_4]=f_2$&&$\be\in\{1,-\frac12\}$&&\\
			\hline
			$A_{4,11}^\al$&  $[f_1,f_4]=2\al f_1$, $[f_3,f_4]=f_2+\al f_3$ &\multirow{2}{*}{Yes} & \multirow{2}{*}{ Yes} & \multirow{2}{*}{Yes} &$\mathfrak{d}^\prime_{4,2\al}$  \\
			$\al>0$ & $[f_2,f_3]=f_1$, $[f_2,f_4]=\al f_2-f_3$ &&&&$\al>0$\\
			\hline
			\multirow{2}{*}{$A_{4,12}$}& $[f_1,f_3]=f_1$, $[f_2,f_3]=f_2$ & \multirow{2}{*}{Yes}& \multirow{2}{*}{ Yes} & \multirow{2}{*}{Yes}  &\multirow{2}{*}{$\mathfrak{r}^\prime_2$}  \\
			&$[f_1,f_4]=-f_2$, $[f_2,f_4]=f_1$&&&&\\
			\hline
		\end{tabular}\captionof{table}{Non unimodular four dimensional Lie algebras.\label{1} }
	\end{center}}

	\begin{center}
		
		\begin{tabular}{|c|l|c|c|c|c|}
			\hline
			Lie algebra   & Nonzero brackets   & Type 2 &Type 1  & Type 0 &Lie algebra in \cite{O1} \\
			\hline
			$A_{4,5}^{\al,\be}$, $\al+\be\not=-1,0$&  $[f_1,f_4]=f_1$, $[f_2,f_4]=\al f_2$   &\multirow{2}{*}{ No}&\multirow{2}{*}{ No} & \multirow{2}{*}{ No}  &\multirow{2}{*}{$\mathfrak{r}_{4,\al,\be}$ }  \\
			$-1<\al<\be<1$& $[f_3,f_4]=\be f_3$&    &  &   &  \\		
			\hline
			$A_{4,5}^{-\al,\al}$, $\al\not=0$& $[f_1,f_4]=f_1$, $[f_2,f_4]=-\al f_2$  &\multirow{2}{*}{ No}&\multirow{2}{*}{ Yes}& \multirow{2}{*}{ Yes}& \multirow{2}{*}{$\mathfrak{r}_{4,-\al,\al}$}\\
			$\al<1$                      &$[f_3,f_4]=\al f_3$&&&&\\
			\hline
			$A_{4,5}^{-1,\be}$, $\be\not=0$ & $[f_1,f_4]=f_1$, $[f_2,f_4]=-f_2$  &\multirow{2}{*}{ No}&\multirow{2}{*}{ Yes}&\multirow{2}{*}{ Yes}&\multirow{2}{*}{ $\mathfrak{r}_{4,-1,-|\be|}$
				}\\
			$-1<\be<1$                      &$[f_3,f_4]=\be f_3$&&&&\\
			\hline
			$A_{4,5}^{\al,\al}$, $\al\not=0$& $[f_1,f_4]=f_1$, $[f_2,f_4]=\al f_2$  &\multirow{2}{*}{Yes}&\multirow{2}{*}{ Yes}&\multirow{2}{*}{No}&\multirow{2}{*}{ $\mathfrak{r}_{4,\al,\al}$}\\
			$-1<\al<1$                           &$[f_3,f_4]=\al f_3$&   &&&\\
			\hline
			$A_{4,5}^{\al,1}$, $\al\not=0$& $[f_1,f_4]=f_1$, $[f_2,f_4]=\al f_2$  &\multirow{2}{*}{ Yes}&\multirow{2}{*}{ Yes}&\multirow{2}{*}{No}&\multirow{2}{*}{ $\mathfrak{r}_{4,\al,1}$}\\
			$-1<\al<1$                                 &$[f_3,f_4]=f_3$&&&&\\
			\hline
			\multirow{2}{*}{$A_{4,5}^{-1,1}$}& $[f_1,f_4]=f_1$, $[f_2,f_4]=- f_2$  &\multirow{2}{*}{ Yes}&\multirow{2}{*}{ Yes}& \multirow{2}{*}{ Yes}& \multirow{2}{*}{$\mathfrak{r}_{4,-1,-1}$}\\
			&$[f_3,f_4]= f_3$&&&&\\
			\hline
			\multirow{2}{*}{ $A_{4,5}^{1,1}$}& $[f_1,f_4]=f_1$, $[f_2,f_4]= f_2$  &\multirow{2}{*}{No}&\multirow{2}{*}{ Yes}&\multirow{2}{*}{No}&\multirow{2}{*}{ $\mathfrak{r}_{4,1,1}$}\\
			&$[f_3,f_4]= f_3$&&&&\\
			\hline
			\end{tabular}
	\end{center}

	{\renewcommand*{\arraystretch}{1.4}
	\begin{center}
		\begin{tabular}{|c|l|c|c|c|c|}
			\hline
			Lie algebra&Non zero brackets&Type 2&Type 1&Type 0&Lie algebra in \cite{O1}\\
			\hline
			$A_{3,1}\oplus A_1$&$[f_2,f_3]=f_1$&Yes&Yes&Yes&$\mathfrak{rh}_3$\\
			\hline
			$A_{3,4}\oplus A_1$&$[f_1,f_3]=f_1,[f_2,f_3]=-f_2$&No&Yes&Yes&$\mathfrak{rr}_{3,-1}$\\
			\hline
			$A_{3,6}\oplus A_1$&$[f_1,f_3]=-f_2,[f_2,f_3]=f_1$&Yes&Yes&Yes&$\mathfrak{rr'}_{3,0}$\\
			\hline
			$A_{3,8}\oplus A_1$&$[f_1,f_2]=-f_3,[f_3,f_1]=f_2,[f_2,f_3]=f_1$&Yes&Yes&No&
			$\mathrm{sl}(2,\R)\oplus A_1$\\
			\hline
			$A_{3,9}\oplus A_1$&$[f_1,f_2]=f_3,[f_3,f_1]=f_2,[f_2,f_3]=f_1$&Yes&Yes&No&$\mathrm{su}(2)\oplus A_1$\\
			\hline
			$A_{4,1}$&$[f_2,f_4]=f_1,[f_3,f_4]=f_2$&No&Yes&Yes&$\mathfrak{n}_4$\\
			\hline
			\multirow{2}{*}{$A_{4,2}^{-2}$}&$[f_1,f_4]=-2f_1,[f_2,f_4]=f_2,$&No&No&No&$\mathfrak{r}_{4,-\frac12}$\\
			&$[f_3,f_4]=f_2+f_3$&&&&\\
			\hline
			\multirow{2}{*}{$A_{4,5}^{-\frac12,-\frac12}$}&$
			[f_1,f_4]=f_1,[f_2,f_4]=-\frac12 f_2,$&Yes&Yes &No&$\mathfrak{r}_{4,-\frac12,-\frac12}$\\
			&$[f_3,f_4]=-\frac12f_3$&&&&\\
			\hline
			\multirow{2}{*}{$A_{4,5}^{\al,-1-\al},\al\in(-1,-\frac12)$}&$
			[f_1,f_4]=f_1,[f_2,f_4]=\al f_2,$&No&No &No&$\mathfrak{r}_{4,\al,-1-\al}$\\
			&$[f_3,f_4]=-(1+\al)f_3$&&&&\\
			\hline
			\multirow{2}{*}{$A_{4,6}^{-2\be,\be},\be>0$}&$
			[f_1,f_4]=-2\be f_1,[f_2,f_4]=\be f_2-f_3,$&Yes&Yes &No&$\mathfrak{r'}_{4,-\frac12,-\frac1{2\be}}$\\
			&$[f_3,f_4]=f_2+\be f_3$&&&&\\
			\hline
			$A_{4,8}$&$
			[f_2,f_3]= f_1,[f_2,f_4]= f_2,[f_3,f_4]=- f_3$&Yes&No &No&$\mathfrak{d}_4$\\
			\hline
			$A_{4,10}$&$
			[f_2,f_3]= f_1,[f_2,f_4]= -f_3,[f_3,f_4]= f_2$&Yes&Yes &No&$\mathfrak{d'}_{4,0}$\\
			\hline
		\end{tabular}\captionof{table}{Unimodular four dimensional Lie algebras.\label{2}}
	\end{center}}

			{\renewcommand*{\arraystretch}{1.2}
				\begin{center}
					\begin{tabular}{|c|ll|}
						\hline
						The Lie algebra & \multicolumn{2}{c|}{Generalized complex structures  of type 1 and their pure spinors}\\
						
						\hline
						\multirow{2}{*}{$A_2\oplus2A_1$ }  &$J=E_{12}-E_{21}$, $ R=-f^{\#}_{34}$, $\si=-f_{\#}^{34}$  & $\rho=f^1+if^2+f^{234}-if^{134}$,\quad$\mathrm{d}\rho=f^1.\rho$\\
						\cline{2-3}
						&$J=E_{34}-E_{43}$, $ R=-f^{\#}_{12}$, $\si=-f_{\#}^{12}$ & $\rho=-{f}^{3}-i{f}^{4}-{f}^{124}+i{f}^{123}$,\quad$\mathrm{d}\rho=0$\\
						\hline
						\hline
						\multirow{3}{*}{$2A_2$ }  &$J=E_{34}-E_{43}+E_{24}$&$\rho=ik{f}^{3}-k{f}^{4}+i{f}^{124}+{f}^{123}-i{f}^{134}$\\
						& $R=-kf^{\#}_{12}$, $\si=-\frac{1}{k}f_{\#}^{12}+\frac{1}{k} f_\#^{13}$&$\mathrm{d}\rho=-f^3.\rho$\\
						\cline{2-3}
						&$J=E_{34}-E_{43}$, $R=-f^{\#}_{12}$, $\si=-f_{\#}^{12}$&$\rho=i{f}^{3}-{f}^{4}+i{f}^{124}+{f}^{123}$, $\mathrm{d}\rho=-f^3.\rho$\\
						\hline
						\hline
						$A_{33}\oplus A_1$  &$J= E_{12}-E_{21}$, $ R=-f^{\#}_{34}$, $\si=-f_{\#}^{34}$&$\rho= if^1-{f}^{2}+i{f}^{234}+{f}^{134}$,\quad  $\mathrm{d}\rho=f^3.\rho$\\
						\hline\hline
						$A_{3,7}^{\al}\oplus A_1$ &$J=\pm E_{12}\mp E_{21}$, $ R=-f^{\#}_{34}$, $\si=-f_{\#}^{34}$ &$\rho=if^1\mp f^2\pm if^{234}+f^{134}$,\;$\mathrm{d}\rho=(\mp f_4+\al f^3).\rho$\\
						\hline
						\hline
						\multirow{2}{*}{$A_{4,2}^{-1}$}      &$J=\la E_{11}+\la E_{22}+E_{34}-E_{43}$&
						$\rho=-i{f}^{3}+{f}^{4}+(\la-i) {f}^{124}+(1+i\la){f}^{123}$\\
						& $ R=-f^{\#}_{12}$, $\si=-(\la^2+1) f_\#^{12}$&$\mathrm{d}\rho=f^4.\rho$\\
						\hline
						$A_{4,2}^{1}$ & $J=-E_{13}+E_{31}$, $ R=-kf^{\#}_{24}$, $\si=-\frac{1}{k}f_{\#}^{24}$& $\rho=ikf^1+k{f}^{3}+i{f}^{234}+{f}^{124}$,\quad$\mathrm{d}\rho=f^4.\rho$\\
						\hline
						\hline
						\multirow{1}{*}{$A_{4,3}$}   &$J=E_{14}- E_{41}$, $R=-f^{\#}_{23}$, $\si=-f_{\#}^{23}$& $\rho=f^1+i{f}^{4}+{f}^{234}-i{f}^{123}$, $\;\mathrm{d}\rho=f^4.\rho$\\
						\hline
						\hline
						$A_{4,5}^{-\al,\al}$, $|\al|<1$ &$J=E_{14}-E_{41}$, $ R=-f^{\#}_{23}$, $\si=-f_{\#}^{23}$& $\rho=f^1+i{f}^{4}+{f}^{234}-i{f}^{123}$,\quad$\mathrm{d}\rho=f^4.\rho $ \\
						\hline
						$A_{4,5}^{-1,\be}$, $|\be|<1$& $J=E_{34}-E_{43}$, $ R=-f^{\#}_{12}$, $\si=-f_{\#}^{12}$&$\rho={f}^{3}+i{f}^{4}+{f}^{124}-i{f}^{123}$,\quad$\mathrm{d}\rho=f^4.\rho$\\
						\hline
						$A_{4,5}^{\al,\al}  $, $-1<\al\leq1$&$J=E_{23}-E_{32}$, $ R=-f^{\#}_{14}$, $\si=-f_{\#}^{14}$&$\rho={f}^{2}+i{f}^{3}-{f}^{134}+i{f}^{124}$,\quad$\mathrm{d}\rho=\al f^4.\rho$\\
						\hline
						$A_{4,5}^{\al,1}$,    $|\al|<1$&$J=-E_{13}+ E_{31}$, $ R=-f^{\#}_{24}$, $\si=-f_{\#}^{24}$&$\rho=f^1-i{f}^{3}+{f}^{234}-i{f}^{124}$,\quad$\mathrm{d}\rho=f^4.\rho$\\
						\hline
						\hline
						\multirow{2}{*}{$A_{4,5}^{-1,1}$}&$J=E_{14}-E_{41}$, $ R=-f^{\#}_{23}$, $\si=-f_{\#}^{23}$& $\rho=f^1+i{f}^{4}+{f}^{234}-i{f}^{123}$,\quad$\mathrm{d}\rho=f^4.\rho $\\
						\cline{2-3}
						& $J=-E_{13}+ E_{31}$, $ R=-f^{\#}_{24}$, $\si=-f_{\#}^{24}$&$\rho=f^1-i{f}^{3}+{f}^{234}-i{f}^{124}$,\quad$\mathrm{d}\rho=f^4.\rho$\\
						\hline
						\hline
						$A_{4,6}^{\al,\be}$, $\be>0$ &$J=\mp E_{23}\pm E_{32}$, $R=-f^{\#}_{14}$, $\si=-f_{\#}^{14}$&$\rho={f}^{2}\mp i{f}^{3}\pm {f}^{134}+i{f}^{124}$,\quad$\mathrm{d}\rho=(\mp f_1+\be f^4).\rho$\\
						\hline
						\multirow{2}{*}{$A_{4,6}^{\al,0}$}&$J=E_{14}-E_{41}$,  $ R=-f^{\#}_{23}$, $\si=-f_{\#}^{23} $& $\rho=f^1+i{f}^{4}+{f}^{234}-i{f}^{123}$,\quad $\mathrm{d}\rho=\al f^4.\rho$ \\
						\cline{2-3}
						&$J=\mp E_{23}\pm E_{32}$, $R=-f^{\#}_{14}$, $\si=-f_{\#}^{14}$&$\rho={f}^{2}\mp i{f}^{3}\pm {f}^{134}+i{f}^{124}$,\quad$\mathrm{d}\rho=\mp f_1.\rho$\\
						\hline
						\hline
						$A_{4,9}^{-\frac{1}{2}}$&$J=\mp E_{24}\pm E_{42}$, $ R=-f^{\#}_{13}$, $\si=-f_{\#}^{13}$&$\rho={f}^{2}\mp i{f}^{4}\mp {f}^{134}+i{f}^{123}$,\quad$\mathrm{d}\rho=f^4.\rho$\\
						
						\hline
						
						\multirow{1}{*}{ $A_{4,9}^{1}$}&$J=E_{23}-E_{32}$, $ R=-f^{\#}_{14}$, $\si=-f_{\#}^{14}$&$\rho=-{f}^{2}-i{f}^{3}+{f}^{134}-i{f}^{124}$,\quad$\mathrm{d}\rho=f^4.\rho$\\
						\hline
						\hline
						$A_{4,11}^{\al}$&$J=\mp E_{23}\pm E_{32}$, $ R=-kf^{\#}_{14}$, $\si=-\frac{1}{k}f_{\#}^{14}$&$\rho=\pm k{f}^{2}-ik{f}^{3}+{f}^{134}\pm i{f}^{124}$,\;$\mathrm{d}\rho=(\mp kf_1+\al f^4).\rho$\\
						\hline
						\hline
						$A_{4,12}$&$J=- E_{12}+ E_{21}$, $ R=-kf^{\#}_{34}$, $\si=-\frac{1}{k}f_{\#}^{34}$&$\rho=-kf^1+i\frac{1}{k}{f}^{2}+{f}^{234}+i{f}^{134}$,\;$\mathrm{d}\rho=(-kf_3+f^3).\rho$\\
						\hline
					\end{tabular}\captionof{table}{Generalized complex structures of type 1 on non-unimodular Lie algebras \label{8} $(\la\in\R,\,k\in\mathbb{R}^*,\,\e\in\{0,1\}$.)}\end{center}}
			

				{\renewcommand*{\arraystretch}{1.4}
					\begin{center}	\begin{tabular}{|c|ll|}
							\hline
							Lie algebra & \multicolumn{2}{c|}{Generalized complex structures of type 1 and their pure spinors} \\
							\hline
							\multirow{3}{*}{$A_{3,1}\oplus A_1$}&  $J=E_{34}-E_{43}$, $R=-f_{12}^\#$, $\si=-f^{12}_\#$& $\rho=f^3+if^4+f^{124}-if^{123}$,\quad$\mathrm{d}\rho=0$\\
							\cline{2-3}
							&$J=\la(E_{11}+E_{44})+E_{23}-E_{32}$ & $\rho=f^2+if^3-(1+i\la)f^{134}+(i-\la)f^{124}$\\
							&$R=-f_{14}^\#$, $\si=-(1+\la^2)f^{14}_\#$& $\mathrm{d}\rho=0$\\
							\hline	\hline
							$A_{3,4}\oplus A_1$& $J=E_{43}- E_{34}$, $R=-f_{12}^\#$, $\si=-f^{12}_\#$&$\rho=-f^3+if^4+f^{124}+if^{123}$,\quad$\mathrm{d}\rho=0$\\
							\hline\hline
							\multirow{2}{*}{$A_{3,6}\oplus A_1,$}&$J=E_{21}-E_{12}$, $R=-f_{34}^\#$, $\si=-f^{34}_\#$& $\rho=-f^1+if^2+f^{234}+if^{134}$,\quad$\mathrm{d}\rho=f_4.\rho$,\\
							\cline{2-3}
							&$J=E_{43}-E_{34}$, $R=-f_{12}^\#$, $\si=-f^{12}_\#$& $\rho=-f^3+if^4+f^{124}+if^{123}$,\quad$\mathrm{d}\rho=0$\\
							\hline
							\hline
							$A_{3,8}\oplus A_1$& $J=E_{21}- E_{12}+\la(E_{33}+E_{44})$ &$\rho=f^1-if^2-(1+i\la)f^{234}+(\la-i)f^{134}$\\
							$A_{3,9}\oplus A_1$&$R=-f_{34}^\#$, $\si=-(1+\la^2)f^{34}_\#$& $\mathrm{d}\rho=(f_4+\la f^3).\rho$\\
							\hline\hline
							\multirow{2}{*}{$A_{4,1}$}& $J=E_{43}- E_{34}+\la(E_{11}+E_{22})$, &$\rho=-f^3+if^4+(1+i\la)f^{124}+(i-\la)f^{123}$\\
							&$R=-f_{12}^\#$, $\si=-(1+\la^2)f^{12}_\#$&  $\mathrm{d}\rho=0$\\
							\hline\hline
							$A_{4,5}^{-\frac12,-\frac12}$&	$J=E_{23}-E_{32}$, $R=-f_{14}^\#$,  $\si=-f^{14}_\#$ &$\rho=f^2+if^3-f^{134}+if^{124}$,\quad$\mathrm{d}\rho=-\frac{1}2f^4.\rho $ \\
							\hline\hline
							$A_{4,6}^{-2\be,\be}$	&$J=\e(E_{23}-E_{32})$, $R=-\e f_{14}^\#$, $\si=-\e f^{14}_\#$&$\rho=f^2+\e if^3-f^{134}+\e if^{124}$,\quad$\mathrm{d}\rho=(f_1+\be f^4).\rho $, \\
							\hline\hline
							\multirow{2}{*}{$A_{4,10}$}& $J=E_{23}- E_{32}+\la(E_{11}+E_{44})$ & $\rho=f^2+if^3-\e(1+i\la)f^{134}+\e(i-\la)f^{124}$\\
							&$R=-\e f_{14}^\#$, $\si=-\e (1+\la^2)f^{14}_\#$ &$\mathrm{d}\rho=(\e f_1-\la f^4).\rho$ \\
							\hline
						\end{tabular}\captionof{table}{Generalized complex structures of type 1 on unimodular Lie algebras $(\la\in\R,\e\in\{1,-1\})$. \label{9}}	
					\end{center}

		{\renewcommand*{\arraystretch}{1.4}
			\begin{center}
				\begin{tabular}{|l|l|c|}
					\hline
					Source & Isomorphism&Target \\
					\hline
					$\U_1,\;\;y=0$&$f_1=e_1,f_2=e_3-\frac23q_1e_1,f_3=e_4-\frac23q_2e_1,f_4=e_2$&$A_{4,5}^{-\frac12,-\frac12}$\\
					\hline
					\multirow{2}{*}{$\U_1,\;\;y\not=0$}&$f_1=e_1,f_2=(4y^2+9)\left( e_4+\frac{4yq_1-6q_2}{4y^2+9}e_1 \right),$&$A_{4,6}^{-\frac1{|y|},\frac1{2|y|}}$\\
					&$f_3=s(4y^2+9)\left( e_3-\frac{4yq_2+6q_1}{4y^2+9}e_1 \right),f_4=\frac{s}ye_2$,$s=-\frac{|y|}y.$&\\
					\hline
					$\U_2,\;\;y\not=0,(b_1,b_2)=(0,0)$&
					$f_1=q_2e_1-ye_3,f_2=q_1e_1+ye_4,f_3=\frac1y e_2,f_4=e_1$&$A_{3,6}\oplus A_1$\\
					\hline
					\multirow{2}{*}{	$\U_2,\;\; b_2y>0$}&$f_1=\frac{1}{\sqrt{yb_2}}(\frac{q_2}{y}e_1-e_3), f_2=\frac{1}{\sqrt{yb_2}}(\frac{q_1}{y}e_1+e_4)$&$A_{3,8}\oplus A_1$\\
					&$f_3=\frac1y(\frac{b_1}{b_2}e_1+e_2),f_4=e_1$&\\
					\hline
					\multirow{2}{*}{$\U_2,\;\; b_2y<0$}&$f_1=\frac{1}{\sqrt{-yb_2}}(\frac{q_2}{y}e_1-e_3), f_2=\frac{1}{\sqrt{-yb_2}}(\frac{q_1}{y}e_1+e_4)$&$A_{3,9}\oplus A_1$\\
					&$f_3=\frac1y(\frac{b_1}{b_2}e_1+e_2),f_4=e_1$&\\
					\hline
					$\U_2,\;\;y\not=0, b_2=0,b_1\not=0$&$f_1=-\frac{1}{y^2b_1}e_1,f_2=\frac{1}{yb_1}(\frac{q_2}ye_1-e_3),
					f_3=\frac{1}{yb_1}(\frac{q_1}ye_1+e_4),f_4=\frac1ye_2$
					&$A_{4,10}$\\
					\hline
					\multirow{2}{*}{$\U_2,\;\;y=0,b_2\not=0,q_1\not=0$}&$f_1=-\frac{1}{q_1^2b_2}e_1,f_2=
					-\frac{b_1}{q_1b_2}e_1-\frac1{q_1}e_2,$&$A_{4,1}$
					\\
					&$f_3={q_2}e_3-q_1e_4,f_4=-\frac{q_1q_2b_2+1}{b_2q_1^2}e_3+e_4$&\\
					\hline
					\multirow{2}{*}{$\U_2,\;\;y=0,b_2\not=0,q_2\not=0$}&
					$f_1=\frac{1}{q_2^2b_2}e_1,f_2=
					-\frac{b_1}{q_2b_2}e_1-\frac1{q_2}e_2,$&$A_{4,1}$
					\\
					&$f_3={q_1}e_4-q_2e_3,f_4=\frac{-q_1q_2b_2+1}{b_2q_2^2}e_4+e_3$&\\
					\hline
					$\U_2,\;\;y=0,b_2\not=0,q_1=q_2=0$&$f_1=b_1e_1+b_2e_2,f_2=e_3,f_3=e_4,f_4=e_1$&$A_{3,1}\oplus A_1$\\
					\hline
					$\U_2,\;\;y=0,b_2=0,b_1=0,q_1\not=0$&$f_1=-q_1e_1,f_2=-q_1e_2,f_3=-\frac{1}{q_1}e_3,f_4=-\frac{q_2}{q_1}e_3+e_4$&$A_{3,1}\oplus A_1$\\
					\hline
					$\U_2,\;\;y=0,b_2=0,b_1=0,q_2\not=0$&$f_1=-q_2e_1,f_2=-q_2e_2,f_3=-\frac{1}{q_2}e_4,f_4=-\frac{q_1}{q_2}e_4+e_3$&$A_{3,1}\oplus A_1$\\
					\hline
					$\U_2,\;\;y=0,b_2=0,b_1\not=0$&$f_1=b_1e_1,f_2=e_3,f_3=e_4,f_4=e_2+\frac{q_2}{b_1}e_3-\frac{q_1}{b_1}e_4$&$A_{3,1}\oplus A_1$\\
					\hline
					$\U_3,\;\;p^2+qr=1,r=0$&$f_1=-\frac{q}{2p}e_1+e_2,f_2=e_1,f_3=\frac{1}pe_4,f_4=\frac{pb_1+qb_2}{p^2}e_1-\frac{b_2}pe_2+e_3$&$A_{3,4}\oplus A_1$\\
					\hline
					\multirow{2}{*}{$\U_3,\;\;p^2+qr=1,r\not=0$}&$f_1=
					\frac12\left(\frac{p-1}re_1+e_2\right),f_2=\frac12\left(\frac{p+1}re_1+e_2\right),f_3=e_4$&$A_{3,4}\oplus A_1$\\
					&$f_4=\left( b_1p-\frac{(p^2-1)b_2}r  \right)e_1-(pb_2-rb_1)e_2+e_3$&\\
					\hline
					\multirow{2}{*}{$\U_3,\;\;p^2+qr=-1$}&$f_1=-pe_1-re_2,f_2=e_1,f_3=e_4$&$A_{3,6}\oplus A_1$\\
					&$f_4=-\left( b_1p-\frac{(p^2+1)b_2}r  \right)e_1+(pb_2-rb_1)e_2+e_3$&\\
					\hline
					\multirow{1}{*}{$\U_3,\;\;p^2+qr=0,r\not=0,b_1=\mu p, b_2=\mu r$}&$f_1=-pe_1-re_2,f_2=e_1,f_3=e_4, f_4=\mu e_1+e_3
					$&$A_{3,1}\oplus A_1$\\
					\hline
					\multirow{2}{*}{$\U_3,\;\;p^2+qr=0,r\not=0,rb_1-p b_2\not=0$}&
					$f_1=-(rb_1-pb_2)(pe_1+re_2),f_2=-(rb_1-pb_2)e_1,$&$A_{4,1}$\\
					&$f_3={b_2}e_1+{r}
					e_3,\;f_4=\frac{b_1}pe_1-\frac{b_2}pe_2-\frac{p}re_3-e_4$&\\
					\hline
					$\U_3,\;\;p=r=0,q\not=0, b_2=0$&$f_1=-qe_1,f_2=e_2,f_3=e_4,f_4=e_3+\frac{b_1}qe_2$&$A_{3,1}\oplus A_1$\\
					\hline
					$\U_3,\;\;p=r=0,q\not=0, b_2\not=0$&$f_1=-b_2qe_1,f_2=b_1e_1+b_2e_2,f_3=e_3,f_4=e_4$&$A_{4,1}$\\
					\hline
					$\U_3,\;\;p=r=q=0, b_2\not=0$&$f_1=b_1e_1+b_2e_2,f_2=e_3,f_3=e_4,
					f_4=e_1$&$A_{3,1}\oplus A_1$\\
					\hline
					$\U_3,\;\;p=r=q=0, b_1\not=0$&$f_1=b_1e_1+b_2e_2,f_2=e_3,f_3=e_4,
					f_4=e_2$&$A_{3,1}\oplus A_1$\\
					\hline
				\end{tabular}\captionof{table}{Isomorphisms from    the Lie algebras obtained in Theorem \ref{main1} onto the unimodular Lie algebras in Table \ref{2}.\label{10}}\end{center}}
		{\renewcommand*{\arraystretch}{1.4}
			\begin{center}
				\begin{tabular}{|cl|l|c|}
					\hline
					&Source& Isomorphism&Target\\
					
					\hline
					$\B_{1}$& $q_1\in\mathbb{R}$ & $f_1=\frac{q_2}{2}e_1+e_2+e_4$, $f_2=e_1$, $f_3=\frac{q_1}{2}e_1+e_3$, $f_4=-e_2$& $A_{4,5}^{-1,1}$ \\
					\hline
					\hline
					$\B_{2}$&  $q_1,q_2\in\mathbb{R}$ &$f_1=-e_2$, $f_2=e_1$, $f_3=-q_2e_1-e_4$, $f_4=q_1e_1+e_3$& $2A_2$\\
					\hline
					\hline
					$\B_{3}$ &$y=0$, $x=0$&$f_1=-e_2$, $f_2=e_1$, $f_3=-q_1e_1+e_3$, $f_4=-q_2e_1+e_4$&$A_2\oplus2A_1$\\
					\hline
					$\B_{3}$& $y=0$, $x=1$& $f_1=-\frac{q_2}{2}e_1+e_4$, $f_2=e_1$, $f_3=-\frac{q_1}{2}e_1+e_3$, $f_4=-e_2$ &$A_{4,5}^{-1,1}$ \\
					\hline
					$\B_{3}$& $y=0$, $x\not\in\{-1,0\}$, $|x|<1$&$f_1=e_1$, $f_2=-\frac{q_2}{x+1}e_1+e_4$, $f_3=-\frac{q_1}{x+1}e_1+e_3$, $f_4=e_2$ &$A_{4,5}^{-x,-x}$ \\
					\hline
					$\B_{3}$ &$y=0$, $x\not\in\{-1,0\}$, $|x|>1$&$f_1=-\frac{q_2x}{x+1}e_1+e_4$, $f_2=e_1$, $f_3=-\frac{q_1x}{x+1}e_1+e_3$, $f_4=e_2$ &$A_{4,5}^{-\frac{1}{x},1}$ \\
					\hline
					$\B_{3}$ &$y=0$, $x=-1$, $q_1=q_2=0$ &$f_1=e_1$, $f_2=e_3$, $f_3=e_4$, $f_4=e_2$ &$A_{4,5}^{1,1}$ \\
					\hline
					$\B_{3}$  &$y=0$, $x=-1$, $q_1\not=0$ &$f_1=\frac{q_2}{q_1}e_3+e_4$, $f_2=q_1e_1$, $f_3=e_3$, $f_4=e_2$& $A_{4,2}^1$ \\
					\hline
					$\B_{3}$  &$y=0$, $x=-1$, $q_2\not=0$ & $f_1=e_3+\frac{q_1}{q_2}e_4$, $f_2=-q_2e_1$, $f_3=e_4$, $f_4=e_2$& $A_{4,2}^1$ \\
					\hline
					$\B_{3}$ &$y\not=0$ $(q_1,q_2)=(0,0)$, $xy\leq0$& $f_1=e_1$, $f_2=-e_3$, $f_3=e_4$, $f_4=\frac{1}{y}e_2$ &$A_{4,6}^{\frac{1}{y},-\frac{x}{y}}$ \\
					\hline
					$\B_{3}$ &$y\not=0$ $(q_1,q_2)=(0,0)$, $xy\geq0$& $f_1=e_1$, $f_2=e_3$, $f_3=e_4$, $f_4=-\frac{1}{y}e_2$ &$A_{4,6}^{-\frac{1}{y},\frac{x}{y}}$ \\
					
					\hline
					\multirow{2}{*}{$\B_{3}$} & \multirow{2}{*}{ $y\not=0$ $(q_1,q_2)\not=(0,0)$,  $xy\leq0$}&
					$f_1=e_1$, $f_2=e_1-\frac{xq_1+yq_2+q_1}{q_1^2+q_2^2}e_3-\frac {xq_2-yq_1+q_2}{q_1^2+q_2^2}e_4$& \multirow{2}{*}{$A_{4,6}^{\frac{1}{y},-\frac{x}{y}}$}\\
					&&$f_3=-\frac {xq_2-yq_1+q_2}{q_1^2+q_2^2}e_3+\frac{xq_1+yq_2+q_1}{q_1^2+q_2^2}e_4$, $f_4=\frac1y e_2$ &\\
					\hline
					\multirow{2}{*}{$\B_{3}$} & \multirow{2}{*}{ $y\not=0$ $(q_1,q_2)\not=(0,0)$,  $xy\geq0$}&
					$f_1=e_1$, $f_2=-e_1+\frac {xq_{{1}}+yq_{{2}}+q_{{1}}}{{q_{{1}}}^{2}+{q_{{2}}}^{2}}e_3+\frac {q_{{2}}x-q_{{1}}y+q_{{2}}}{{q_{{1}}}^{2}+{q_{{2}}}^{2}}e_4$& \multirow{2}{*}{$A_{4,6}^{-\frac{1}{y},\frac{x}{y}}$}\\
					&&$f_3=-\frac {q_{{2}}x-q_{{1}}y+q_{{2}}}{{q_{{1}}}^{2}+{q_{{2}}}^{2}}e_3+\frac {xq_{{1}}+yq_{{2}}+q_{{1}}}{{q_{{1}}}^{2}+{q_{{2}}}^{2}}e_4$, $f_4=-\frac{1}{y} e_2$ &\\
					\hline
					\hline
					\multirow{2}{*}{ $\B_{4}$}&\multirow{2}{*}{$y=0$}&$f_1=e_1$, $f_2=-2q_1e_1+e_3$ &\multirow{2}{*}{$A_{4,9}^{1}$}\\
					&&$f_3=(-2q_1-2q_2)e_1+e_3+e_4$, $f_4=2e_2$&\\
					\hline
					\multirow{2}{*}{$\B_{4}$} &\multirow{2}{*}{$y\not=0$}&$f_1=-se_1$, $f_2=\frac{4yq_2+2q_1}{4y^2 + 1}e_1-e_3$ &\multirow{2}{*}{$A_{4,11}^{\frac{1}{2|y\,|}}$}\\
					&&$f_3=s(\frac{4yq_1-2q_2}{4y^2+1}e_1+e_4)$, $f_4=\frac{s}{y}e_2$, $s=\frac{|y|}{y}$&\\
					\hline
				\end{tabular}\captionof{table}{Isomorphisms from    the Lie algebras obtained in Theorem \ref{main2} onto the nonunimodular Lie algebras in Table \ref{1}.\label{12}}\end{center}}

		{\renewcommand*{\arraystretch}{1.1}
			\begin{center}
				\begin{tabular}{|cl|l|c|}
					\hline
					&Source& Isomorphism&Target\\
					\hline
					$\A_1$& $y_1=0$, $y_2=0$& $f_1=e_3$, $f_2=e_4$, $f_3=-e_1$, $f_4=e_2$ & $A_{3,3}\oplus A_1$\\
					\hline
					$\A_1$& $y_1\not=0$, $y_2=0$&$f_1=e_3$, $f_2=se_4$, $f_3=-\frac{s}{y_1}e_1$, $f_4=e_2$, $s=\frac{|y_1|}{y_1}$&$A_{3,7}^{\frac{1}{|y_1|}}\oplus A_1$\\
					\hline
					$\A_1$& $y_2\not=0$&$f_1=-e_3$, $f_2=e_4$, $f_3=-e_1+\frac{y_1}{y_2}e_2$, $f_4=\frac{1}{y_2}e_2$&  $A_{4,12}$\\
					\hline
					\hline
					\multirow{2}{*}{$\A_2$}&\multirow{2}{*}{$c\not=0$}&$f_1=s(-\frac{bc+a}{x(c^2+1)}e_2+e_3)$, $f_2=\frac{ac-b}{x(c^2+1)}e_2+e_4$ & \multirow{2}{*}{$A_{3,7}^{\frac{1}{|c|}}\oplus A_1$} \\
					&& $f_3=-\frac{s}{cx}e_1$, $f_4=e_2$, $s=\frac{|c|}{c}$ &\\
					\hline
					\multirow{2}{*}{$\A_2$}& \multirow{2}{*}{$c=0$}&$f_1=-\frac{b}{x}e_2+e_4$, $f_2=-\frac{a}{x}e_2+e_3$ & \multirow{2}{*}{$A_{3,3}\oplus A_1$}\\
					&&$f_3=-\frac{1}{x}e_1$, $f_4=e_2$&\\
					\hline
					\hline
					$\A_3$&$r=0$, $p=0$, $q=0$&$f_1=-\frac{1}{b_3}e_4$,  $f_2=\frac{b_1}{b_3}e_1+\frac{b_2}{b_3}e_2+e_3$, $f_3=e_1$, $f_4=e_2$& $A_2\oplus 2A_1$\\
					\hline
					\multirow{2}{*}{$\A_3$}&\multirow{2}{*}{$r=0$, $p=0$, $q\not=0$}&$f_1=\frac{qb_2+b_1b_3}{b_3}e_1+b_2e_2+b_3e_3,$  $f_2=\frac{q}{b_3}e_1$ & \multirow{2}{*}{$A_{4,3}$}\\
					&&$f_3=e_2$, $f_4=\frac{1}{b_3}e_4$&\\
					\hline
					\multirow{2}{*}{ $\A_3$}&$r=0$, $p\not=0$&$f_1=-\frac{pb_1+qb_2+b_1b_3}{(b_3+p)(p-b_3)}e_1+\frac{b_2}{b_3+p}e_2+e_3$ &\multirow{2}{*}{$A_{4,5}^{-\frac{p}{b_3},\frac{p}{b_3}}$}\\
					&  $0<\frac{p}{b_3}<1$&$f_2=-\frac{q}{2p}e_1+e_2$, $f_3=e_1$, $f_4=\frac{1}{b_3}e_4$&\\
					\hline
					\multirow{2}{*}{ $\A_3$}&$r=0$, $p\not=0$ &$f_1=-\frac{pb_1+qb_2+b_1b_3}{(b_3+p)(p-b_3)}e_1+\frac{b_2}{b_3+p}e_2+e_3$ &\multirow{2}{*}{$A_{4,5}^{\frac{p}{b_3},-\frac{p}{b_3}}$}\\
					&$-1<\frac{p}{b_3}<0$ & $f_2=e_1$, $f_3=-\frac{q}{2p}e_1+e_2$, $f_4=\frac{1}{b_3}e_4$&\\
					
					\hline
					\multirow{2}{*}{ $\A_3$}&$r=0$, $p\not=0$  & $f_1=e_1$, $f_2=-\frac {q}{2p}e_1+e_2$       &\multirow{2}{*}{$A_{4,5}^{-1,\frac{b_3}{p}}$}\\
					&$|\frac{p}{b_3}|>1$&$f_3=-\frac {pb_1+qb_2+b_1b_3}{p^2-b_3^2}e_1+\frac {b_2}{p+b_3}e_2+e_3$, $f_4=\frac{1}{p}e_4$&\\
					\hline
					\multirow{2}{*}{$\A_3$}&$r=0$,  $p=\pm b_3$& $f_1=-\frac{q}{2p}e_1+e_2$,  $f_2=\frac{2pb_1+qb_2}{2p^2}e_1$  &  \multirow{2}{*}{$A_{4,2}^{-1}$}\\
					&$2b_1p+qb_2\not=0$& $f_3=\frac{b_2}{2p}e_2+e_3$, $f_4=\frac{1}{p}e_4$&\\
					\hline
					\multirow{2}{*}{$\A_3$}&$r=0$, $p=\pm b_3$ &$f_1=\frac{b_2}{2p}e_2+e_3$, $f_2=-\frac{q}{2p}e_1+e_2$ &  \multirow{2}{*}{$A_{4,5}^{-1,1}$}\\
					&$2b_1p+qb_2=0$&$f_3=e_1$, $f_4=\frac{1}{p}e_4$&\\
					\hline
					\multirow{2}{*}{$\A_3$}&\multirow{2}{*}{$r\not=0$, $p^2+qr=-1$}&$f_1=-\frac{p^2b_2-prb_1-rb_1b_3+b_2}{r(b_3^2+1)}e_1-\frac{pb_2-rb_1-b_2b_3}{b_3^2+1}e_2+e_3$& \multirow{2}{*}{$A_{4,6}^{-b_3,0}$}\\
					&&$f_2=e_1$, $f_3=pe_1+re_2$, $f_4=-e_4$&\\
					\hline
					\multirow{2}{*}{$\A_3$}&$r\not=0$,  $|b_3|<1$&$f_1=\frac{p+1}{r}e_1+e_2$, $f_2=\frac{p-1}{r}e_1+e_2$, $f_4=e_4$& \multirow{2}{*}{$A_{4,5}^{-1,b_3}$}\\
					&$p^2+qr=1$&$f_3=-\frac{p^2b_2-prb_1-rb_1b_3-b_2}{r(b_3^2-1)}e_1-\frac{pb_2-rb_1-b_2b_3}{b_3^2-1}e_2+e_3$&\\
					\hline
					\multirow{2}{*}{$\A_3$}&$r\not=0$, $b_3>1$&$f_1=-\frac{p^2b_2-prb_1-rb_1b_3-b_2}{r(b_3^2-1)}e_1-\frac{pb_2-rb_1-b_2b_3}{b_3^2-1}e_2+e_3$&\multirow{2}{*}{$A_{4,5}^{-\frac{1}{b_3},\frac{1}{b_3}}$}\\
					&$p^2+qr=1$& $f_2=\frac{p-1}{r}e_1+e_2$, $f_3=\frac{p+1}{r}e_1+e_2$, $f_4=\frac{1}{b_3}e_4$&\\
					\hline
					\multirow{2}{*}{$\A_3$}&$r\not=0$, $b_3<-1$&$f_1=-\frac{p^2b_2-prb_1-rb_1b_3-b_2}{r(b_3^2-1)}e_1-\frac{pb_2-rb_1-b_2b_3}{b_3^2-1}e_2+e_3$&\multirow{2}{*}{$A_{4,5}^{\frac{1}{b_3},-\frac{1}{b_3}}$}\\
					&$p^2+qr=1$& $f_2=\frac{p+1}{r}e_1+e_2$, $f_3=\frac{p-1}{r}e_1+e_2$, $f_4=\frac{1}{b_3}e_4$&\\
					\hline
					\multirow{2}{*}{$\A_3$}&$r\not=0$, $p^2+qr=1$, $b_3=1$&$f_1=\frac{p+1}{r}e_1+e_2$, $f_2=\frac{p-1}{r}e_1+e_2$&\multirow{2}{*}{$A_{4,5}^{-1,1}$}\\
					&$rb_1=pb_2-b_2$&$f_3=-\frac{b_2}{r}e_1+e_3$, $f_4=e_4$&\\
					\hline
					\multirow{2}{*}{$\A_3$}&$r\not=0$, $p^2+qr=1$, $b_3=1$ & $f_1=\frac{p-1}{r}e_1+e_2$, $f_2=\frac{p+1}{r}e_1+e_2$, $f_4=e_4$&\multirow{2}{*}{$A_{4,2}^{-1}$}\\
					&$rb_1\not=pb_2-b_2$ &$f_3=\frac{pb_2-rb_1+b_2}{r(pb_2-rb_1-b_2)}e_1-\frac{2}{pb_2-rb_1-b_2}e_3$&\\
					\hline
					\multirow{2}{*}{$\A_3$}&$r\not=0$, $p^2+qr=1$, $b_3=-1$&$f_1=\frac{p-1}{r}e_1+e_2$, $f_2=\frac{p+1}{r}e_1+e_2$ &\multirow{2}{*}{$A_{4,5}^{-1,1}$}\\
					&$rb_1=pb_2-b_2$&$f_3=\frac{b_2}{r}e_1-e_3$, $f_4=-e_4$&\\
					\hline
					\multirow{2}{*}{$\A_3$}&$r\not=0$,$p^2+qr=1$, $b_3=-1$ &$f_1=\frac{p+1}{r}e_1+e_2$, $f_2=\frac{p-1}{r}e_1+e_2$, $f_4=-e_4$
					&\multirow{2}{*}{$A_{4,2}^{-1}$}\\
					&$rb_1\not=pb_2-b_2$ &$f_3=-\frac{pb_2-rb_1-b_2}{r(pb_2-rb_1+b_2)}e_1-\frac{2}{pb_2-rb_1+b_2}e_3$
					&\\
					\hline
					
					\hline
					\multirow{2}{*}{$\A_4$}&\multirow{2}{*}{$q_1\not=0$}&$f_1=2q_1e_1$, $f_2=-(\frac{1}{3}b_2q_2+b_1)e_1-\frac{1}{3}b_2e_2+e_3$&\multirow{2}{*}{ $A_{4,9}^{-\frac12}$}\\
					&& $f_3=-q_2e_1+2e_2$, $f_4=-\frac{1}{2}e_4$&\\
					\hline
					\hline
					\multirow{2}{*}{$\A_5$}&&$f_1=pe_2+e_3+e_4$, $f_2=-pe_1+e_3-e_4$&\multirow{2}{*}{$A_{4,12}$}\\
					&&$f_3=-\frac12e_1-\frac12e_2$, $f_4=\frac12e_1-\frac12e_2$&\\
					\hline	
				\end{tabular}\captionof{table}{Isomorphisms from    the Lie algebras obtained in Theorem \ref{main3} onto the nonunimodular Lie algebras in Table \ref{1}.\label{11}}	
			\end{center}}

		{\renewcommand*{\arraystretch}{1.3}\begin{center}
				\begin{tabular}{|c|l|}
					\hline
					The Lie algebra  &\multicolumn{1}{c|}{ 2-cocycle}\\
					\hline
					$A_2\oplus2A_1$ &$a_{{1,2}}f^\#_{1,2}+a_{1,3}f^\#_{1,3}+a_{1,4}f^\#_{1,4}+a_{3,4}f^\#_{3,4}$\\
					\hline
					$2A_2$ &$a_{{1,2}}f^\#_{{1,2}}+a_{{1,3}}f^\#_{{1,3}}+a_{{3,4}}f^\#_{{3,4}}$\\
					\hline
					$A_{3,1}\oplus A_1$	&$a_{{1,2}}f^\#_{{1,2}}+a_{{1,3}}f^\#_{{1,3}}+a_{{2,3}}f^\#_{{2,3}}+a_{{2,4}}f^\#_{{2,4}}+a_{{3,4}}f^\#_{{3,4}}$\\
					\hline
					$A_{3,2}\oplus A_1$&$a_{{1,3}}f^\#_{{1,3}}+a_{{2,3}}f^\#_{{2,3}}+a_{{3,4}}f^\#_{{3,4}}$\\
					\hline
					$A_{3,3}\oplus A_1$&$a_{{1,3}}f^\#_{{1,3}}+a_{{2,3}}f^\#_{{2,3}}+a_{{3,4}}f^\#_{{3,4}}$\\
					\hline
					$A_{3,4}\oplus A_1$&$a_{{1,2}}f^\#_{{1,2}}+a_{{1,3}}f^\#_{{1,3}}+a_{{2,3}}f^\#_{{2,3}}+a_{{3,4}}f^\#_{{3,4}}$\\
					\hline
					$A_{3,5}^\al\oplus A_1$&$a_{{1,3}}f^\#_{{1,3}}+a_{{2,3}}f^\#_{{2,3}}+a_{{3,4}}f^\#_{{3,4}}$\\
					\hline
					$A_{3,6}\oplus A_1$&$a_{{1,2}}f^\#_{{1,2}}+a_{{1,3}}f^\#_{{1,3}}+a_{{2,3}}f^\#_{{2,3}}+a_{{3,4}}f^\#_{{3,4}}$\\
					\hline
					$A_{3,7}^\al\oplus A_1$&$a_{{1,3}}f^\#_{{1,3}}+a_{{2,3}}f^\#_{{2,3}}+a_{{3,4}}f^\#_{{3,4}}$\\
					\hline
					$A_{3,8}\oplus A_1$&$a_{{1,2}}f^\#_{{1,2}}+a_{{1,3}}f^\#_{{1,3}}+a_{{2,3}}f^\#_{{2,3}}$\\
					\hline
					$A_{3,9}\oplus A_1$&$a_{{1,2}}f^\#_{{1,2}}+a_{{1,3}}f^\#_{{1,3}}+a_{{2,3}}f^\#_{{2,3}}$\\
					\hline
					$A_{4,1}$&$a_{{1,4}}f^\#_{{1,4}}+a_{{2,3}}f^\#_{{2,3}}+a_{{2,4}}f^\#_{{2,4}}+a_{{3,4}}f^\#_{{3,4}}$
					\\
					\hline
					$A^\al_{4,2}$&$a_{{1,4}}f^\#_{{1,4}}+a_{{2,4}}f^\#_{{2,4}}+a_{{3,4}}f^\#_{{3,4}}$\\
					\hline
					$A_{4,3}$&$a_{{1,4}}f^\#_{{1,4}}+a_{{2,3}}f^\#_{{2,3}}+a_{{2,4}}f^\#_{{2,4}}+a_{{3,4}}f^\#_{{3,4}}$\\
					\hline
					$A_{4,4}$&$a_{{1,4}}f^\#_{{1,4}}+a_{{2,4}}f^\#_{{2,4}}+a_{{3,4}}f^\#_{{3,4}}$\\
					\hline
					$A_{4,5}^{\al,\be}$, $\al\not=-1$, $\al\be\not=0$ &\multirow{2}{*}{$a_{{1,4}}f^\#_{{1,4}}+a_{{2,4}}f^\#_{{2,4}}+a_{{3,4}}f^\#_{{3,4}}$}\\
					$\be\not=-1$,  $\al\not=-\be$&\\
					\hline
					$A_{4,5}^{0,-1}$ &$a_{{1,3}}f^\#_{{1,3}}+a_{{1,4}}f^\#_{{1,4}}+a_{{2,4}}f^\#_{{2,4}}+a_{{3,4}}f^\#_{{3,4}}$\\
					\hline
					$A_{4,5}^{-\frac{1}{2},-\frac{1}{2}}$ &$a_{{1,4}}f^\#_{{1,4}}+a_{{2,3}}f^\#_{{2,3}}+a_{{2,4}}f^\#_{{2,4}}+a_{{3,4}}f^\#_{{3,4}}$\\
					\hline
					$A_{4,5}^{-1,\be}$, $\be\not\in\{-1,1\}$&$a_{{1,2}}f^\#_{{1,2}}+a_{{1,4}}f^\#_{{1,4}}+a_{{2,4}}f^\#_{{2,4}}+a_{{3,4}}f^\#_{{3,4}}$\\
					\hline
					$A_{4,5}^{-1,1}$&$a_{{1,2}}f^\#_{{1,2}}+a_{{1,4}}f^\#_{{1,4}}+a_{{2,3}}f^\#_{{2,3}}+a_{{2,4}}f^\#_{{
							2,4}}+a_{{3,4}}f^\#_{{3,4}}
					$\\
					\hline
					$A_{4,5}^{-\al,\al}$, $\al\not=1$ &$a_{{1,4}}f^\#_{{1,4}}+a_{{2,3}}f^\#_{{2,3}}+a_{{2,4}}f^\#_{{2,4}}+a_{{3,4}}f^\#_{{3,4}}$\\
					\hline
					$A^{\al,\be}_{4,6}$, $\be\not=0$&$a_{{1,4}}f^\#_{{1,4}}+a_{{2,4}}f^\#_{{2,4}}+a_{{3,4}}f^\#_{{3,4}}$\\
					\hline
					$A^{\al,0}_{4,6}$&$a_{{1,4}}f^\#_{{1,4}}+a_{{2,3}}f^\#_{{2,3}}+a_{{2,4}}f^\#_{{2,4}}+a_{{3,4}}f^\#_{{3,4}}$\\
					\hline
					$A_{4,7}$&$2\,a_{{2,3}}f^\#_{{1,4}}+a_{{2,3}}f^\#_{{2,3}}+a_{{2,4}}f^\#_{{2,4}}+a_{{3,4}}f^\#_{{3,4}}$\\
					\hline
					$A_{4,8}$&$a_{{2,3}}f^\#_{{2,3}}+a_{{2,4}}f^\#_{{2,4}}+a_{{3,4}}f^\#_{{3,4}}$\\
					\hline
					$A_{4,9}^\be$&$(\be a_{{2,3}}+a_{{2,3}}) f^\#_{{1,4}}+a_{{2,3}}f^\#_{{2,3}}+a_{{2,4}}f^\#_{{2,4}}+a_{{3,4}}f^\#_{{3,4}}$\\
					\hline
					$A_{4,10}$ &$a_{{2,3}}f^\#_{{2,3}}+a_{{2,4}}f^\#_{{2,4}}+a_{{3,4}}f^\#_{{3,4}}$\\
					\hline
					$A_{4,11}^\al$ &$2\al a_{{2,3}}f^\#_{{1,4}}+a_{{2,3}}f^\#_{{2,3}}+a_{{2,4}}f^\#_{{2,4}}+a_{{3,4}}f^\#_{{3,4}}$
					\\
					\hline
					$A_{4,12}$&$-a_{{2,3}}f^\#_{{1,4}}+a_{{2,3}}f^\#_{{2,3}}+a_{{2,4}}f^\#_{{1,3}}+a_{{2,4}}f^\#_{{2,4}}+a_{{3,4}}f^\#_{{3,4}}$
					\\
					\hline
					
				\end{tabular}\captionof{table}{2-cocycles on four dimensional Lie algebras.\label{13}}
			\end{center}}
		
		\section{Appendix}\label{section7}
		
	This appendix completes the proof of Theorem \ref{type1}. For each Lie algebra $A$ in Tables \ref{8} and \ref{9}, we give all the  isomorphisms from the $\A_i$, $\B_i$ and $\U_i$ to $A$, the corresponding generalized complex structures on $A$ and the automorphisms and the 2-cocycles of $A$ which permits to reduce these generalized complex structures.
	\subsection{The Unimodular Case}

	$\bullet$	$A_{3,4}\oplus A_1$
	
	\begin{enumerate}
		
		\item The isomorphism from $\U_3$: $f_1=-\frac{q}{2p}e_1+e_2,f_2=e_1,f_3=\frac{1}pe_4,f_4=\frac{pb_1+qb_2}{p^2}e_1-\frac{b_2}pe_2+e_3.$
		
		The generalized complex structure:
		
		$\begin{cases} J_1=\la E_{11}+\frac{b_2}{p^2}E_{13}-\frac{\la b_2}{p}E_{14}+\la E_{22}-\frac{2pb_1+qb_2}{2p^3}E_{23}+\frac{\la(2pb_1+qb_2)}{2p^2}E_{24}-pE_{34}+\frac{1}{p}E_{43}\\
		R_1=f_{12}^\#\\
		\si_1=(1+\la^2)( f^{12}_\#+\frac{2pb_1+qb_2}{2p^2}f^{14}_\#+ \frac{b_2}pf^{24}_\#).\end{cases}$
		
		We have
		\[	\phi(T)\exp(B)(J_1,R_1,\si_1)\exp(-B)\phi(T^{-1})=(E_{43}-E_{34},-
		f_{12}^\#,-f^{12}_\#),\]
		
		where
		
		$T= -E_{11}+\frac{\la b_2}{p^2}E_{13}+E_{22}+\frac{\la (2pb_1+qb_2)}{p^3}E_{23}+E_{33}+pE_{44}$,
		
		$B=-\la f^{12}_\#+\frac{2pb_1+qb_2}{2p^3}f^{23}_\#+\frac{b_2}{p^2}f^{23}_\#.$
		
		\item The isomorphism from $\U_3$: $f_1=\frac12 (\frac{p-1}re_1+e_2), f_2=\frac12 (\frac{p+1}re_1+e_2), f_3=e_4,
		f_4=( b_1p-\frac{(p^2-1)b_2}r)e_1-(pb_2-rb_1)e_2+e_3.$
		
		The generalized complex structure
		
		$\begin{cases}
		J_2=\la E_{11}+aE_{12}-\la aE_{14}+\la E_{22}+bE_{23}-\la bE_{24}-E_{34}+E_{43}\\
		R_2=f_{12}^\#\\
		\si_2=\frac{\la^2+1}{2r}(f^{12}_\#-bf^{14}_\#+af^{24}_\#),\\
		\end{cases}$
		
		where $a=(p+1)b_2-rb_1$ and $b=(p-1)b_2-rb_1.$
		
		We have
		\[ \phi(T)\exp(B)(J_2,R_2,\si_2)\exp(-B)\phi(T^{-1})=(E_{43}-E_{34},-f_{12}^\#,-f^{12}_\#),\]
		where
		
		$T=-\frac{1}{2}E_{11}+\frac{\la a}{2}E_{13} +\frac{1}{r}E_{22}-\frac{\la b}{r}E_{23}+E_{33}+E_{44}$,
		
		$B=-\frac1{2r}(f^{12}_\# +b f^{13}_\# -a f^{23}_\#).$
	\end{enumerate}

	$\bullet$	$A_{3,6}\oplus A_1$
	
	\begin{enumerate}\item The isomorphism from $\U_2$: $f_1=q_2e_1-ye_3,f_2=q_1e_1+ye_4,f_3=\frac1y e_2,f_4=e_1$.
		
		The generalized complex structure
		
		$\begin{cases}
		J_1=-E_{12}+E_{21}+\la E_{33}+(\la q_2-q_1)E_{41}+(\la q_1+q_2)E_{42}+\la E_{44}\\
		R_1=yf_{34}^\#\\
		\si_1=-\frac1y(1+\la^2)(q_2f^{13}_\#+q_1f^{23}_\#).
		\end{cases}$
		
		We have
		\[ \phi(T)\exp(B)(J_1,R_1,\si_1) \exp(-B)\phi(T^{-1})=(E_{21}-E_{12},-f_{34}^\#,-f_{\#}^{34}),\]		
		where
		
		$T=E_{11}+E_{22}+E_{33}-\frac1yE_{44}$,
		
		$B=\frac1y((\la q_2-q_1)f^{13}_\#+(\la q_1+q_2)f^{23}_\#-\lambda f_{\#}^{34})$.

		\item The isomorphism from $\U_3$: $f_1=-pe_1-re_2,f_2=e_1,f_3=e_4,
		f_4=-( b_1p-\frac{(p^2+1)b_2}r)e_1+(pb_2-rb_1)e_2+e_3$.
		
		The generalized complex structure
		
		$\begin{cases}
		J_2=\la E_{11}+\frac {pb_2-rb_1}{r}E_{13}-\frac{\la (pb_2-rb_1) }{r}E_{14}+\la E_{22}-\frac{b_2}{r}E_{23}-E_{34}+E_{43}\\
		R_2= -\frac1rf_{12}^\#\\
		\si_2=-(\la^2+1)( rf^{12}_\#+b_2f^{14}_\#+(pb_2-rb_1)f^{24}_\#).
		\end{cases}$
		
		We have
		\[ \exp(B)\phi(T^{-1})(J_2,R_2,\si_2)\phi(T)\exp(B)=(E_{43}-E_{34},-cf_{12}^\#,-\frac1cf_\#^{12}),\]
		
		where $c={\frac {r}{ \left( {\lambda}^{2}+1 \right) ^{2} \left( {p}^{2}{b_{{2}}
				}^{2}-2\,prb_{{1}}b_{{2}}+{r}^{2}{b_{{1}}}^{2}+{b_{{2}}}^{2} \right) }}$,

		$T=-\frac{b_2(\la^2+1)}{r}E_{11}+\frac{(\la^2+1)(pb_2-rb_1)}{r}E_{12}+\frac{\la(pb_2-rb_1)}{r}E_{13}-
		\frac{(\la^2+1)(pb_2-rb_1)}{r}E_{21}-\frac{b_2(\la^2+1)}{r}E_{22}-\frac{\la b_2}{r}E_{23}+E_{33}+E_{44}$ and
		$B=af_\#^{12}+bf_\#^{13}$, with
		\[ \begin{cases}a={\frac {\lambda\, \left( {\lambda}^{4}{p}^{2}{b_{{2}}}^{2}-2\,{\lambda
				}^{4}prb_{{1}}b_{{2}}+{\lambda}^{4}{r}^{2}{b_{{1}}}^{2}+{\lambda}^{4}{
				b_{{2}}}^{2}+2\,{\lambda}^{2}{p}^{2}{b_{{2}}}^{2}-4\,{\lambda}^{2}prb_
			{{1}}b_{{2}}+2\,{\lambda}^{2}{r}^{2}{b_{{1}}}^{2}+2\,{\lambda}^{2}{b_{
					{2}}}^{2}+{p}^{2}{b_{{2}}}^{2}-2\,prb_{{1}}b_{{2}}+{r}^{2}{b_{{1}}}^{2
			}+{b_{{2}}}^{2} \right) }{r}},\\
	b={\frac {{\lambda}^{4}{p}^{2}{b_{{2}}}^{2}-2\,{\lambda}^{4}prb_{{1}}b_{
				{2}}+{\lambda}^{4}{r}^{2}{b_{{1}}}^{2}+{\lambda}^{4}{b_{{2}}}^{2}+2\,{
				\lambda}^{2}{p}^{2}{b_{{2}}}^{2}-4\,{\lambda}^{2}prb_{{1}}b_{{2}}+2\,{
				\lambda}^{2}{r}^{2}{b_{{1}}}^{2}+2\,{\lambda}^{2}{b_{{2}}}^{2}+{p}^{2}
			{b_{{2}}}^{2}-2\,prb_{{1}}b_{{2}}+{r}^{2}{b_{{1}}}^{2}+{b_{{2}}}^{2}}{
			r}}.
	\end{cases} \]
	On the other hand, the automorphisms $T_2=\frac {1}{\sqrt {k}}(E_{12}-E_{21})+E_{33}+E_{44}$ and $T_3=\frac {1}{\sqrt {-k}}(E_{12}+E_{21})-E_{33}-E_{44}$
	satisfy
	\[ \phi(T_1)(E_{43}-E_{34},-kf_{12}^\#,-\frac1k f_{\#}^{12})\phi(T_1^{-1})=
	\phi(T_2)(E_{43}-E_{34},-kf_{12}^\#,-\frac1k f_{\#}^{12})\phi(T_2^{-1})=
	(E_{43}-E_{34},-f_{12}^\#,- f_{\#}^{12}). \]
\end{enumerate}

$\bullet$	($A_{3,8}\oplus A_1$, $\e=1$) and ($A_{3,9}\oplus A_1$, $\e=-1$)

\begin{enumerate}\item The isomorphism from $\U_2$: $f_1=\frac{1}{\sqrt{\e yb_2}}(\frac{q_2}{y}e_1-e_3), f_2=\frac{1}{\sqrt{\e yb_2}}(\frac{q_1}{y}e_1+e_4),\;f_3=\frac1y(\frac{b_1}{b_2}e_1+e_2),f_4=e_1$.
	
	The generalized complex structure
	
	$\begin{cases}
	J_1=-E_{12}+E_{21}+\la E_{33}+\frac{\la q_2-q_1}{\sqrt{\e yb_2y}}E_{41}+
	\frac {\la q_1+q_2}{\sqrt{\e yb_2}y}E_{42}+\la E_{44}\\
	R_1=yf_{34}^\#\\
	\si_1=-\frac{1}{y}(1+\la^2)(\frac{q_2}{y\sqrt{\e yb_2}}f^{13}_\#+\frac{q_1}{y\sqrt{\e yb_2}}f^{23}_\# -f^{34}_\#).
	\end{cases}$
	
	We have
	\[\exp(B)\phi(T^{-1})(J_1,R_1,\si_1)\phi(T)\exp(-B)=
	(E_{12}-E_{12}+\la(E_{33}+E_{44}),-cf_{34}^\#,-\frac1cf_{\#}^{34}),\]			
	where $c=-{\frac {\sqrt {\e yb_{{2}}}{y}^{2}\sqrt { \left( {\lambda}^{2}+1
				\right)  \left( {q_{{1}}}^{2}+{q_{{2}}}^{2} \right) }}{ \left( {
				\lambda}^{2}+1 \right)  \left( {q_{{1}}}^{2}+{q_{{2}}}^{2} \right) }}$,
	
	$T=\frac{\la q_2-q_1}{\sqrt{(\la^2+1)(q_1^2+q_2^2)}}E_{11}-
	{\frac {\la q_1+q_2}{\sqrt{(\la^2+1)(q_1^+q_2^2) }}}E_{12}+{\frac {\la q_1+q_2}{\sqrt{(\la^2+1)(q_1^+q_2^2) }}}E_{21}+{\frac {\la q_1-q_2}{\sqrt{(\la^2+1)(q_1^+q_2^2) }}}E_{22}+E_{33}+
	\frac{(\la^2+1)(q_1^2+q_2^2)}{\sqrt {\e yb_2}y\sqrt{(\la^2+1)({q_1^2+q_2^2)}}}E_{44}$ and
	
	$B=af^{13}_\#$, with  $a={\frac {{\lambda}^{2}{q_{{1}}}^{2}+{\lambda}^{2}{
				q_{{2}}}^{2}+{q_{{1}}}^{2}+{q_{{2}}}^{2}}{\sqrt {\e yb_{{2}}}{y}^{2}
			\sqrt { \left( {\lambda}^{2}+1 \right)  \left( {q_{{1}}}^{2}+{q_{{2}}}
				^{2} \right) }}}.$

	On the other hand, the automorphism $T_k:(f_1,f_2,f_3,f_4)\mapsto (f_1,f_2,f_3,kf_4)$ satisfies
	\[ \phi(T_k)(E_{12}-E_{12}+\la(E_{33}+E_{44}),-f_{34}^\#,-f_{\#}^{34})\phi(T_k)=(E_{12}-E_{12}+\la(E_{33}+E_{44}),-kf_{34}^\#,-\frac1kf_{\#}^{34}). \]\end{enumerate}

$\bullet$	$A_{4,1}$

\begin{enumerate}\item The isomorphism from $\U_2$: $f_1=-\frac{1}{q_1^2b_2}e_1,f_2=
	-\frac{b_1}{q_1b_2}e_1-\frac1{q_1}e_2,,f_3={q_2}e_3-q_1e_4,f_4=-\frac{q_1q_2b_2+1}{b_2q_1^2}e_3+e_4$.

	The complex structure
	
	$\begin{cases}
	J_1=\la E_{11}+\la E_{22}+\frac{b_2q_1^3+b_2q_1q_2^2+q_2}{q_1}E_{23}+{\frac {-{b_{{2}}}^{2}{q_{{1
					}}}^{4}-{b_{{2}}}^{2}{q_{{1}}}^{2}{q_{{2}}}^{2}-2\,b_{{2}}q_{{1}}q_{{2
				}}-1}{b_{{2}}{q_{{1}}}^{3}}}E_{34} +q_1b_2(q_1^2+q_2^2)E_{43}-\frac{b_2{q_{{1}}}
			^{3}+b_2q_1q_2^2+q_2}{q_1}E_{44}\\
		R_1=-b_{{2}}{q_{{1}}}^{3}f_{12}^\#\\
		\si_1=-\frac1{b_{{2}}{q_{{1}}}^{3}}(1+\la^2) f^{12}_\#.
		\end{cases}$
		
		We have
		
		\[ \phi(T^{-1})(J_1,R_1,\sigma_1)\phi(T)=(E_{43}-E_{34}+\la(E_{11}+E_{22}),-cf_{12}^\#,-\frac1c(1+\la^2) f_\#^{12}),\]
		
		where $c={q_{{1}}}^{5}{b_{{2}}}^{3}(q_1^2+q_2^2)^{2}$ and	$T=\frac{1}{q_1b_2(q_1^2+q_2^2)}(E_{11}+E_{22}+E_{33})+{\frac {b_{{2}}{q_{{1}}}^{3}+b_{{2}}q_{{1}}
				{q_{{2}}}^{2}+q_{{2}}}{{q_{{1}}}^{2}( {q_{{1}}}^{2}+{q_{{2}}}^{2
				}) b_{{2}}}}E_{34}+E_{44}.$
		
		On the other hand, the automorphism $T_1={k}^{\frac35}E_{11}+{k}
		^{\frac25}E_{22}+{k}^{\frac15}(E_{33}+E_{44})$
		satisfies
		\[ \phi(T_1)(E_{43}-E_{34},-f_{12}^\#,- f_\#^{12})\phi(T_1^{-1})=(E_{43}-E_{34}+\la(E_{11}+E_{22}),-kf_{12}^\#,-\frac1k(1+\la^2) f_\#^{12}). \]

		\item The isomorphism from $\U_2$: $f_1=\frac{1}{q_2^2b_2}e_1,f_2=
		-\frac{b_1}{q_2b_2}e_1-\frac1{q_2}e_2,f_3={q_1}e_4-q_2e_3,f_4=\frac{-q_1q_2b_2+1}{b_2q_2^2}e_4+e_3$.
		
		The generalized complex structure
		
		$\begin{cases}
		J_2=\la E_{11}+\la E_{22}+{\frac {{q_{{1}}}^{2}q_{{2}}b_{{2}
				}+b_{{2}}{q_{{2}}}^{3}-q_{{1}}}{q_{{2}}}}E_{33}+{\frac{-{b_{{2}}}^{2}{q_{{1
					}}}^{2}{q_{{2}}}^{2}-{b_{{2}}}^{2}{q_{{2}}}^{4}+2q_{{1}}q_{{2}}b_{{2
				}}-1}{b_{{2}}{q_{{2}}}^{3}}}E_{34}+q_{{2}}b_{{2}}
		( {q_{{1}}}^{2}+{q_{{2}}}^{2})E_{43}-{\frac {{q_{{1}}}^{2}q_{
					{2}}b_{{2}}+b_{{2}}{q_{{2}}}^{3}-q_{{1}}}{q_{{2}}}}E_{44}\\
		R_2=b_{{2}}{q_{{2}}}^{3}f_{12}^\#\\
		\si_2=\frac1{b_{{2}}{q_{{2}}}^{3}}(1+\la^2)f^{12}_\#.
		\end{cases}$
		
		We have
		\[ \phi(T^{-1})(J_2,R_2,\sigma_2)\phi(T)=
		(E_{43}-E_{34}+\la(E_{11}+E_{22}),-cf_{12}^\#,-\frac1c(1+\la^2) f_\#^{12}),\]where $c=-{q_{{2}}}^{5}{b_{{2}}}^{3} ( {q_{{1}}}^{2}+{q_{{2}}}^{2}
		) ^{2} $  and $T=\frac{1}{q_1b_2(q_1^2+q_2^2)}(E_{11}+E_{22}+E_{33})+{\frac {b_{{2}}{q_{{1}}}^{2}q_{{2}}+b_{{2}}
				{q_{{2}}}^{3}-q_{{1}}}{{q_{{2}}}^{2}( {q_{{1}}}^{2}+{q_{{2}}}^{2
				}) b_{{2}}}}E_{34}+E_{44}.$

		\item The isomorphism from $\U_3$: $f_1=-(rb_1-pb_2)(pe_1+re_2),f_2=-(rb_1-pb_2)e_1,f_3={b_2}e_1+{r}
		e_3,
		f_4=\frac{b_1}pe_1-\frac{b_2}pe_2-\frac{p}re_3-e_4.$

		The complex structure
		
		$\begin{cases}
		J_3=\la E_{11}+{\frac {b_{{2}}}{p(pb_{{2}}-rb_{{1}})}}E_{13}-{\frac {b_{{2}}( \la r+p)
			}{{r}^{2}(( pb_{{2}}-rb_{{1}}) p}}E_{14}+\la E_{22}-{\frac {b_{{2}}{p}^{2}-rb_{{2}}(\la-1)p2b_
			{{1}}{r}^{2}}{pr( pb_{{2}}-rb_{{1}})}}E_{23}\\\qquad+{\frac {{p}^{3}b_
			{{2}}+{p}^{2}rb_{{2}}+{r}^{2} \left(  \left( \lambda+1 \right) b_{{2}}
			+b_{{1}} \right) p+\lambda\,b_{{1}}{r}^{3}}{p \left( pb_{{2}}-rb_{{1}}
			\right) {r}^{3}}}E_{24}+{\frac {p}{r}}E_{33}-{\frac{p^{2}+{r}^{2}}{{r}^{3}}}E_{34} +rE_{43}-{\frac {p}{r}}E_{44}\\
	R_3=\frac {1}{r( pb_{{2}}-rb_{{1}}) ^{2}}f_{12}^\#\\
	\si_3=(\la^2+1)(r( pb_2-rb_1)^2f^{12}_\#+b_2r(pb_2-rb_1)f^{13}_\#+\frac{p^2b_2^2-r^2b_1^2}{p}f^{14}_\#+\frac{b_2( pb_2-rb_1)}{p}f^{24}_\#+\frac{b_2^2}{p}f^{34}_\#).
	\end{cases}$
	
	We have
	
	\[ \phi(T)\exp(B)(J_3,R_3,\si_3)\exp(-B)\phi(T)=
	(E_{43}-E_{34}+\la(E_{11}+E_{22}),-cf_{12}^\#,-\frac1c(1+\la^2) f_\#^{12}), \]where $c=-{\frac {{r}^{6}}{( pb_{{2}}-rb_{{1}}) ^{2}}},$

	$T={r}^{4}E_{11}+rE_{12}+{r}^{3}E_{22}+{\frac {-\lambda\,p{r}^{4}b_{{2}}+{p}^{2}{r}^{3}b_{{2}}+p{r}^{4}b
			_{{2}}+{r}^{5}b_{{1}}}{{r}^{2} \left( pb_{{2}}-rb_{{1}} \right) p}}E_{24}
	+{r}^{2}E_{33}-pE_{34}+rE_{44}$ and
	
	$B=af^{14}_\#+bf^{23}_\#+cf^{24}_\#,$ with
	
	\[ \begin{cases}a=-( {\lambda}^{2}+1) b_{{2}}( pb_{{2}}-rb_{{1}})\\
	b={\frac {( pb_{{2}}-rb_{{1}})( {p}^{3}b_{{2}}-r( \lambda\,b_{{2}}+b_{{1}}) {p}^{2}+{r}^{2}\lambda\,b_{{1
				}}p+{r}^{4}b_{{2}}) }{{r}^{3}p}}\\
		c=-{\frac { ( pb_{{2}}-rb_{{1}} )( {p}^{4}b_{{2}}-rb_{{1}}{p}^{3}+{r}^{2}b_{{2}}{p}^{
					2}+{r}^{3}( rb_{{2}}-b_{{1}}) p+{r}^{5}\lambda\,b_{{2}}
				) }{{r}^{5}p}}.
		\end{cases} \]

		\item The automorphism from $\U_3$: $f_1=-b_2qe_1,f_2=b_1e_1+b_2e_2,f_3=e_3,f_4=e_4$.
		The generalized complex structure

		$\begin{cases} J_4=E_{34}-E_{43}+\la(E_{11}+E_{22})\\
		R_4={\frac {1}{{b_{{2}}}^{2}q}}f_{12}^\#\\
		\si_4={{{b_{{2}}}^{2}q}}(1+\la^2)f^{12}_\#.\end{cases}$

		Is clearly equivalent to $(E_{43}-E_{34}+\la(E_{11}+E_{22}),-f_{12}^\#,-(1+\la^2) f_\#^{12})$.
	\end{enumerate}
	
	$\bullet$	$A_{4,5}^{-\frac12,-\frac12}\oplus A_1$
	
	The isomorphism from $\U_1$: $f_1=e_1,f_2=e_3-\frac23q_1e_1,f_3=e_4-\frac23q_2e_1,f_4=e_2$.
	
	The generalized complex structure
	
	$\begin{cases} J=\la E_{11}-\frac{2\la q_1+2q_2}{3}E_{12}-\frac{2\la q_2-2q_1}{3}E_{13}+E_{23}-E_{32}+\la E_{44}\\
	R=-f_{14}^\#\\
	\si=-(1+\la^2)(f^{14}_\#-\frac23q_1f^{24}_\#-\frac23q_2f^{34}_\#).
	\end{cases}$
	
	we have
	\[ \exp(B)(J,R,\si)\exp(-B)=(E_{23}-E_{32},-f_{34}^\#,-f_\#^{34}),\]

	where
	$B=\la f^{14}_\#-\frac23(\la q_1+q_2)f^{24}_\#-\frac23(\la q_2-q_1)f^{34}_\#.$

	$\bullet$	$A_{4,6}^{-2\be,\be}$
	
	\begin{enumerate}\item The isomorphism from $\U_1$: $f_1=e_1,f_2=(4y^2+9)\left( e_4+\frac{4yq_1-6q_2}{4y^2+9}e_1 \right),f_3=\e(4y^2+9)\left( e_3-\frac{4yq_2+6q_1}{4y^2+9}e_1 \right),f_4=\frac{\e}ye_2$,$\e=-\frac{|y|}y.$
		
		The generalized complex structure
		
		$\begin{cases}  J_\e=\la E_{11}+ (( 4\al y+6)q_1+4yq_2-6\la))E_{12}+2\e(q_1(2y-3\la)-q_2(2y\lambda+3))E_{13}-\e E_{23}+\frac{1}{\e}E_{32}+\la E_{44}\\
		R_\e=-\frac {y}{\epsilon}f_{14}^\#\\
		\si_\e=-(\la^2+1)(\frac {\e}{y}f^{14}_\#+\frac{\e( 4yq_{{1}}-6q_2)}{y}f^{24}_\#-\frac{4yq_2+6q_1}{y}f^{34}_\#).\end{cases}$

		We have
		\[\phi(T)\exp(B_\e)(J_\e,R_\e,\si_\e)\exp(-B_\e)\phi(T^{-1})=  \e(E_{23}-E_{32},-f_{14}^\#,-f_\#^{14}),\]		
		where
		
		$T=\frac1y E_{11}+E_{22}+E_{33}+E_{44}$,
		
		$B_\e={\frac{\e\la}{y}}f^{14}_\#+{2\e\frac{2\la yq_1-3\la q_2+2yq_2+3q_1}{y}}f^{24}_\#-
		2{\frac {2\,\lambda\,yq_{{2}}+3\la q_{{1}}-2yq_1+3q_2}{y}}f^{34}_\#. $
		
	\end{enumerate}

	$\bullet$	$A_{4,10}$
	
	The isomorphism from $\U_2$: $f_1=-\frac{1}{y^2b_1}e_1,f_2=\frac{1}{yb_1}(\frac{q_2}ye_1-e_3),
	f_3=\frac{1}{yb_1}(\frac{q_1}ye_1+e_4),f_4=\frac1ye_2$.

	The generalized complex structure

	$\begin{cases} J=\la E_{11}-(\la q_2-q_1)E_{12}-(\la q_1+q_2)E_{13}-E_{23}+E_{32}+\la E_{44}\\
	R=b_{{1}}{y}^{3}f_{14}^\#\\
	\si=\frac{\la^2+1}{b_{{1}}{y}^{3}}( f^{14}_\# -q_2f^{24}_\#-q_1f^{34}_\#).\end{cases}$
	
	We have
	\[ \exp(B)(J,R,\si)\exp(-B)=(E_{23}-E_{32}+\la(E_{11}+E_{44}),-cf_{14}^\#,-\frac1c(1+\la^2) f_{\#}^{14}), \]
	where $c=-y^3b_1$,	 $B=\frac1{b_1y^3}( (\la q_2-q_1)f^{24}_\#+(\la q_1+q_2)f^{34}_\#)$.
	
	On the other hand, the automorphisms $T_1=\frac1kE_{11}+\frac {1}{\sqrt {k}}(E_{22}+E_{33})+E_{14}$ and $T_2=-\frac1kE_{11}+\frac {1}{\sqrt {-k}}(E_{22}+E_{33})+E_{14}$
	satisfy
	\[ \phi(T_1)(E_{23}-E_{32}+\la(E_{11}+E_{44}),-kf_{14}^\#,-\frac1k(1+\la^2) f_{\#}^{14})\phi(T_1^{-1})=(E_{23}-E_{32}+\la(E_{11}+E_{44}),-f_{14}^\#,-(1+\la^2) f_{\#}^{14}) \]
	\[ \phi(T_2)(E_{23}-E_{32}+\la(E_{11}+E_{44}),-kf_{14}^\#,-\frac1k(1+\la^2) f_{\#}^{14})\phi(T_2^{-1})=(E_{23}-E_{32}+\la(E_{11}+E_{44}),f_{14}^\#,(1+\la^2) f_{\#}^{14}). \]

	\subsection{The Non-Unimodular Case}

	$\bullet$  $A_{2}\oplus 2A_1$
	
	\begin{enumerate}\item The isomorphism from $\A_3$: $f_2=\frac{b_1}{b_3}e_1+\frac{b_2}{b_3}e_2+e_3$, $f_3=e_1$, $f_4=e_2$.
		
		The generalized complex structure:
		
		$\begin{cases} J_1=b_3E_{12}-\frac{1}{b_3}E_{21}+\frac{b_1}{b_3^2}E_{31}+\frac{\la b_1}{b_3}E_{32}+\la E_{33}+\frac{b_2}{b_3^2}E_{41}+\frac{\la b_2}{b_3}E_{42}+\la E_{44}\\
		R_1=-f_{34}^\#\\
		\si_1=(\la^2+1)(\frac{b_2}{b_3}f^{23}_\#-\frac{b_1}{b_3}f^{24}_\#-f^{34}_\#).\end{cases}$
		
		We have
		\[ \exp(B)\phi(T)(J_1,R_1,\si_1)\phi(T^{-1})\exp(-B)=(E_{12}-E_{21},-f_{34}^\#,-f^{34}_\#), \]
		where
		
		$T=E_{11}+b_3E_{22}+\frac{\la b_2}{b_3^2}E_{13}-E_{34}-\frac{\la b_1}{b_3}E_{41}+E_{43}$,
		
		$B= -\frac{\la^2b_1+b_1}{b_3^2} f^{13}_\# -\frac{\la^2b_2+b_2}{b_3^2}f^{14}_\#+\la f^{34}_\#.$

		\item The isomorphism from $\B_3$: $f_1=-e_2$, $f_2=e_1$, $f_3=-q_1e_1+e_3$, $f_4=-q_2e_1+e_4$.
		
		The generalized complex structure:
		
		$\begin{cases} J_2=\la E_{11}+\la E_{22}-(\la q_1+q_2)E_{23}-(\la q_2-q_1)E_{24}+E_{34}-E_{43}\\
		R_2=-f_{12}^\#\\
		\si_2=(\la^2+1)(-f^{12}_\#+q_1f^{13}_\#+q_2f^{14}_\#).
		\end{cases} $

		We have
		\[ \exp(B)(J_2,R_2,\si_2)\exp(-B)=(E_{34}-E_{43},-f_{12}^\#,-f^{12}_\#),\]
		
		where $B=\la f^{12}_\# -(\la q_1+q_2)f^{13}_\#-(\la q_2-q_1)f^{14}_\#.$
	\end{enumerate}

	$\bullet$  $2A_{2}$		
	
	The isomorphism from $\B_2$:  $f_1=-e_2$, $f_2=e_1$, $f_3=-q_2e_1-e_4$, $f_4=q_1e_1+e_3$.
	
	The generalized complex structure:

	$\begin{cases} J=\la E_{11}+\la E_{22}-(\la q_2-q_1)E_{23}+(\la q_1+q_2)E_{24}+E_{34}-E_{43}\\
	R=-f_{12}^\#\\
	\si=(\la^2+1)(-f^{12}_\#+q_2f^{13}_\#-q_1f^{14}_\#).
	\end{cases} $
	\begin{itemize}
		\item[-] For $\la q_1+q_2\not=0$, we have
		\[ \exp(B).\phi(T).(J,R,\si)\phi(T^{-1}).\exp(-B)=(E_{24}+E_{34}-E_{43},-\frac{1}{\la q_1+q_2}f_{12}^\#,-(\la q_1+q_2)f^{12}_\#+(\la q_1+q_2)f^{13}_\#),\]		
		where
		
		$T=E_{11}+(\la q_1+q_2) ^{-1}E_{22}+E_{33}+E_{44}$,
		
		$B=\la( \la q_1+q_2)f^{12}_\# -(\la q_2-q_1)f^{13}_\#.$
		\item[-]For $\la q_1+q_2=0$, we have
		\[ \exp(B).(J,R,\si).\exp(-B)=(E_{34}-E_{43},-f_{12}^\#,-f^{12}_\#),\]		
		where $B=\la f^{12}_\#+({\la}^{2}q_1+q_1)f^{13}_\#.$
	\end{itemize}
	
	$\bullet$ $A_{33}\oplus A_1$
	\begin{enumerate}
		\item  The isomorphism from $\A_1$: $f_1=e_3$, $f_2=e_4$, $f_3=-e_1$, $f_4=e_2$.
		
		The generalized complex structure:
		
		$\begin{cases}
		J_1=E_{12}- E_{21}+\la E_{22}+\la E_{44}\\
		R_1=f_{34}^\#\\
		\si_1=(\la^2+1)f^{34}_\#.
		\end{cases} $

		We have
		\[ \exp(B)\phi(T)(J_1,R_1,\si_1)\phi(T^{-1})\exp(-B)=(E_{12}-E_{21},-f_{34}^\#,-f^{34}_\#),\]
		where
		
		$T=E_{11}+E_{22}+E_{33}-E_{44}$, $B=\la f^{34}_\#.$
		
		\item  The isomorphism from $\A_2$:  $f_1=-\frac{b}{x}e_2+e_4$, $f_2=-\frac{a}{x}e_2+e_3$, $f_3=-\frac{1}{x}e_1$, $f_4=e_2$.
		
		The generalized complex structure:
		
		$\begin{cases} J_2=-E_{12}+ E_{21}+\la E_{33}+\frac{-\la b+a}{x}E_{41}-\frac {\la a+b}{x}E_{42}+\la E_{44}\\
		R_2=xf_{34}^\#\\
		\si_2=(\la^2+1)(\frac{b}{x^2}f^{13}_\#+\frac{a}{x^2}f^{23}_\#+\frac{1}{x}f^{34}_\#).
		\end{cases} $

		We have
		\[ \exp(B)\phi(T)(J_2,R_2,\si_2)\phi(T^{-1})\exp(-B)=(E_{12}-E_{21},-f_{34}^\#,-f^{34}_\#),\]
		where
		
		$T=-E_{11}+E_{22}+E_{33}-\frac{1}{x}E_{44}$,
		
		$B=\frac{\la b-a}{x^2}f^{13}_\#-\frac{\la a+b}{x^2}f^{23}_\#+\la f^{34}_\#.$
		
	\end{enumerate}
	$\bullet$  $A^\al_{3,7}\oplus A_1$
	\begin{enumerate}
		\item  The isomorphism from $\A_1$:$f_1=e_3$, $f_2=se_4$, $f_3=-\frac{s}{y_1}e_1$, $f_4=e_2$, $s=\frac{|y_1|}{y_1}$.
		
		The generalized complex structure:
		
		$\begin{cases} J_1=sE_{12}-sE_{21}+\la E_{33}+\la E_{44}\\
		R_1=sy_1f_{34}^\#\\
		\si_1=\frac{s(\la^2+1)}{y_1}f^{34}_\#.
		\end{cases} $

		We have
		\[ \exp(B)\phi(T)(J_1,R_1,\si_1)\phi(T^{-1})\exp(-B)=(sE_{12}-sE_{21},-f_{34}^\#,-f^{34}_\#),\]
		where
		
		$T=E_{11}+E_{22}+E_{33}-\frac{s}{y_1}E_{44}$,
		
		$B=\la f^{34}_\#.$
		\item The isomorphism from $\A_2$: $f_1=s(-\frac{bc+a}{x(c^2+1)}e_2+e_3)$, $f_2=\frac{ac-b}{x(c^2+1)}e_2+e_4$, $f_3=-\frac{s}{cx}e_1$, $f_4=e_2$, $s=\frac{|c|}{c}$.
		
		The generalized complex structure:
		
		$\begin{cases} J_2=sE_{12}-sE_{21}+\la E_{33}-\frac{s(bc\la-ac+a\la+b)}{x(c^2+1)}E_{41}+\frac{ac\la+bc-b\la+a}{x(c^2+1) }E_{24}+\la E_{44}\\
		R_2=scx\,f_{34}^\#\\
		\si_2=\frac{( {\lambda}^{2}+1)(bc+a)}{cx^2(c^2+1)}f^{13}_\#-\frac {s(\la^2+1)(ac-b)}{cx^2(c^2+1)}f^{23}_\#+\frac {{s(\lambda}^{2}+1)}{cx}f^{34}_\#.
		\end{cases} $

		We have
		\[ \exp(B)\phi(T)(J_2,R_2,\si_2)\phi(T^{-1})\exp(-B)=(sE_{12}-sE_{21},-f_{34}^\#,-f^{34}_\#), \]
		where
		
		$T=E_{11}+E_{22}+E_{33}-\frac{s}{cx}E_{44}$,
		
		$B=-\frac {s(bc\la-ac+a\la+b)}{cx^2(c^2+1) }f^{13}_\#-\frac{s(-\la ac-bc+\la b-a)}{cx^2(c^2+1)}f^{23}_\#+f^{34}_\#.$
	\end{enumerate}

	$\bullet$ $A_{4,2}^{-1}$
	\begin{enumerate}
		\item The isomorphism from $\A_3$: $f_1=-\frac{q}{2p}e_1+e_2$,  $f_2=\frac{2pb_1+qb_2}{2p^2}e_1$, $f_3=\frac{b_2}{2p}e_2+e_3$, $f_4=\frac{1}{p}e_4$.
		
		The generalized complex structure:
		
		$\begin{cases} J_1=\la E_{11}+\la E_{22}+\frac{\la b_2}{2p}E_{13}-\frac {b_2}{2p^2}E_{14}+\frac{\la qb_2}{4pb_1+2qb_2}E_{23}-\frac{qb_2}{2p(2pb_1+qb_2) }E_{24}+\frac{1}{p}E_{34}-pE_{43}\\
		R_1=\frac{2p^2}{2pb_1+qb_2}f_{12}^\#\\
		\si_1= (\la^2+1)(\frac{2pb_1+qb_2}{2p^2}f^{12}_\#+\frac{qb_2}{4p^2}f^{13}_\#-\frac{(2pb_1+qb_2)b_2}{4p^3}f^{23}_\#).
		\end{cases} $

		We have
		\[ \exp(B)\phi(T)(J_1,R_1,\si_1)\phi(T^{-1})\exp(-B)=(\la E_{11}+\la E_{22}+E_{34}-E_{43},-f_{12}^\#,-(\la^2+1)f^{12}_\#),\]
		where

		$T=-\frac{2pb_1+qb_2}{2p^3}E_{11}-\frac{\la( 2pb_1+qb_2) b_2}{4p^5}E_{14}+pE_{22}+yE_{24}+pE_{33}+E_{44}$,
		
		$B=\frac{-(4pb_1+2qb_2) y+\la qb_2}{4pb_1+2qb_2}f^{13}_\#-\frac{2\la(2pb_1+qb_2) y+qb_2}{4pb_1+2qb_2}f^{14}_\#-\frac{(2pb_1+qb_2)(\la^2+1)b_2}{4p^5}f^{24}_\#.$

		\item  The isomorphism from $\A_3$: $f_1=\frac{p-1}{r}e_1+e_2$, $f_2=\frac{p+1}{r}e_1+e_2$, $f_4=e_4$, $f_3=\frac{pb_2-rb_1+b_2}{r(pb_2-rb_1-b_2)}e_1-\frac{2}{pb_2-rb_1-b_2}e_3$.
		
		The generalized complex structure:
		
		$\begin{cases} J_2=\la E_{11}-\frac{\la(pb_2-rb_1+b_2)}{2(pb_2-rb_1-b_2)}E_{13}-\frac{pb_2-rb_1+b_2}{4}E_{14}+\la E_{22}+
		\frac{\la(pb_2-rb_1+b_2)}{2(pb_2-rb_1-b_2)}E_{23}\\
		\qquad+\frac{pb_2-rb_1+b_2}{4}E_{24}
		-\frac{pb_{{2}}-rb_{{1}}-b_2}{2}E_{34}+\frac{2}{ pb_2-rb_1-b_2}E_{43}\\
		R_2=\frac{r}{2}f_{34}^\#\\
		\si_2= (\la^2+1)(\frac{2}{r}f^{12}_\#+\frac{( pb_{{2}}-rb_{{1}}+b_{{2}})}{r( pb_{{2}}-rb_{{1}}-b_{{2}})}f^{13}_\#+\frac {( pb_{{2}}-rb_{{1}}+b_{{2}})}{r( pb_{{2}}-rb_{{1}}-b_{{2}})}f^{23}_\#.
		\end{cases} $

		we have
		\[\phi(T) \exp(B)(J_2,R_2,\si_2)\exp(-B)\phi(T^{-1})=(\la E_{11}+\la E_{22}+E_{34}-E_{43},-f_{12}^\#,-(\la^2+1)f^{12}_\#), \]
		where
		
		$T={\frac {pb_{{2}}-rb_{{1}}-b_{{2}}}{r}}E_{11}+{\frac{\la( pb_{{2}}-rb_{{1}}-b_{{2}})( pb_{{2
					}}-rb_{{1}}+b_{{2}})}{r}}E_{14}-\frac{2}{( pb_{{2}}-rb_{{1}}-b_{{2}})}E_{22}+
			{\frac{\la( pb_{{2}}-rb_{{1}}+b_{{2}})}{2(pb_{{2}}-rb_{{1}}-b
					_{{2}})}}E_{23}-\frac{2}{ pb_{{2}}-rb_{{1}}-b_{{2}}
			}E_{33}+E_{44}$,
			
			$B=-{\frac{{\lambda}^{2}pb_{{2}}-{\lambda}^{2}rb_{{1}}+{\lambda}^{2}b_{{2}}+pb_{{2}}-rb_{{1}}+b_{{2}}}{2r}}f^{14}_\#
			-{\frac{{\lambda}^{2}pb_{{2}}-{
						\lambda}^{2}rb_{{1}}+{\lambda}^{2}b_{{2}}+pb_{{2}}-rb_{{1}}+b_{{2}}}{2r}}f^{24}_\#.$

			\item The isomorphism from $\A_3$: $f_1=\frac{p+1}{r}e_1+e_2$, $f_2=\frac{p-1}{r}e_1+e_2$, $f_4=-e_4$
			$f_3=-\frac{pb_2-rb_1-b_2}{r(pb_2-rb_1+b_2)}e_1-\frac{2}{pb_2-rb_1+b_2}e_3$.
			
			The generalized complex structure:

			$\begin{cases} J_3=\la E_{11}-\frac{\la(pb_2-rb_1-b_2)}{2(pb_2-rb_1+b_2)}E_{13}+\frac{pb_2-rb_1-b_2}{4}E_{14}+\la E_{22}+
			\frac{\la(pb_2-rb_1-b_2)}{2(pb_2-rb_1+b_2)}E_{23}\\
			\qquad-\frac{pb_2-rb_1-b_2}{4}E_{24}
			+\frac{pb_{{2}}-rb_{{1}}+b_2}{2}E_{34}-\frac{2}{ pb_2-rb_1+b_2}E_{43}\\
			R_3=-\frac{r}{2}f_{34}^\#\\
			\si_3=- (\la^2+1)(\frac{2}{r}f^{12}_\#+\frac{( pb_{{2}}-rb_{{1}}-b_{{2}})}{r( pb_{{2}}-rb_{{1}}+b_{{2}})}f^{13}_\#+\frac {( pb_{{2}}-rb_{{1}}-b_{{2}})}{r( pb_{{2}}-rb_{{1}}+b_{{2}})}f^{23}_\#.
			\end{cases} $

			We have
			\[\phi(T) \exp(B)(J_3,R_3,\si_3)\exp(-B)\phi(T^{-1})=(\la E_{11}+\la E_{22}+E_{34}-E_{43},-f_{12}^\#,-(\la^2+1)f^{12}_\#),\]
			where
			
			$T={\frac {pb_{{2}}-rb_{{1}}+b_{{2}}}{r}}E_{11}+{\frac{\la( pb_{{2}}-rb_{{1}}-b_{{2}})( pb_{{2
						}}-rb_{{1}}+b_{{2}})}{r}}E_{14}-\frac{2}{( pb_{{2}}-rb_{{1}}+b_{{2}})}E_{22}+
				{\frac{\la( pb_{{2}}-rb_{{1}}+b_{{2}})}{2(pb_{{2}}-rb_{{1}}-b
						_{{2}})}}E_{23}-\frac{2}{ pb_{{2}}-rb_{{1}}3b_{{2}}
				}E_{33}+E_{44}$,
				
				$B=-{\frac{{\lambda}^{2}pb_{{2}}-{\lambda}^{2}rb_{{1}}-{\lambda}^{2}b_{{2}}+pb_{{2}}-rb_{{1}}-b_{{2}}}{2r}}f^{14}_\#
				-{\frac{{\lambda}^{2}pb_{{2}}-{
							\lambda}^{2}rb_{{1}}+{\lambda}^{2}b_{{2}}+pb_{{2}}-rb_{{1}}-b_{{2}}}{2r}}f^{24}_\#.$
			\end{enumerate}
			
			$\bullet$ $A_{4,2}^1$
			\begin{enumerate}
				\item
				The isomorphism from $\B_3$: $f_1=\frac{q_2}{q_1}e_3+e_4$, $f_2=q_1e_1$, $f_3=e_3$, $f_4=e_2$.
				
				The generalized complex structure:
				
				$\begin{cases} J_1=-\frac{q_2}{q_1}E_{11}-E_{13}+E_{22}+\frac{q_1^2+q_2^2}{q_1^2}E_{31}+\frac {q_2}{q_1}E_{33}+\la E_{44}.\\
				R_1=-\frac{1}{q_1}f_{24}^\#\\
				\si_1=-(\la^2+1)q_1f^{24}_\#.
				\end{cases}$

				We have
				\[\exp(B) \phi(T)(J_1,R_1,\si_1)\phi(T^{-1})\exp(-B)=(-E_{13}+E_{31},-\frac{1}{q_1}f_{24}^\#,-q_1f^{24}_\#),\]
				where $T=Id_4+\frac{q_2}{q_1}E_{24}-\frac {q_2}{q_1}E_{31}$, $B=\la q_1 f^{24}_\#.$
				
				\item  The isomorphism from $\B_3$:  $f_1=e_3+\frac{q_1}{q_2}e_4$, $f_2=-q_2e_1$, $f_3=e_4$, $f_4=e_2$.
				
				The generalized complex structure:
				
				$\begin{cases} J_2=\frac{q_1}{q2}E_{11}+E_{13}+\la E_{22}-\frac{q_1^2+q_2^2}{q_2^2}E_{31}-\frac{q_1}{q_2}E_{33}+\la E_{44}\\
				R_2=\frac{1}{q_2}f_{24}^\#\\
				\si_2=(\la^2+1)q_2f^{24}_\#.
				\end{cases}$

				We have
				\[\exp(B) \phi(T)(J_2,R_2,\si_2)\phi(T^{-1})\exp(-B)=(E_{13}-E_{31},-\frac{1}{q_2}f_{24}^\#,-q_2f^{24}_\#),\]
				where $T=Id_4+\frac{q_1}{q_2}E_{24}-\frac {q_1}{q_2}E_{31}$,
				
				$B=-\la q_2 f^{24}_\#.$
				
			\end{enumerate}
			
			$\bullet$ $A_{4,3}$
			
			The isomorphism from $\A_3$: $f_1=\frac{qb_2+b_1b_3}{b_3}e_1+b_2e_2+b_3e_3,$  $f_2=\frac{q}{b_3}e_1$, $f_3=e_2$, $f_4=\frac{1}{b_3}e_4$.
			
			The generalized complex structure:
			
			$\begin{cases} J=\frac{1}{b_3^2}E_{14}+\frac{\la(qb_2+b_1b_3)}{q}E_{21}+\la E_{22}-\frac{qb_2+b_1b_3}{qb_3^2}E_{24}+\la b_2E_{31}+\la E_{33}-\frac{b_2}{b_3^2}E_{34}-\frac{1}{b_3^2}E_{41}\\
			R=-\frac{b_3}{q}f_{23}^\#\\
			\si=\frac{(\la^2+1)}{b_3} (qb_2f^{12}_\#-(qb_2+b_1b_3)f^{13}_\#-qf^{23}_\#).
			\end{cases}$			
			
			We have
			\[ \phi(T)\exp(B)(J,R,\si)\exp(-B)\phi(T^{-1})=(E_{14}-E_{41},-f_{23}^\#,-f^{23}_\#), \]
			where
			
			$T=b_3^2E_{11}+E_{22}+{\frac {\lambda\, \left( qb_{{2}}+b_{{1}}b_{{3}} \right) }{q{b_{{
								3}}}^{2}}}E_{24} +{\frac {q}{b_{{3}}}}E_{33}+{\frac {
					\lambda\,b_{{2}}q}{{b_{{3}}}^{3}}}E_{34}+E_{44}$,
			
			$B=\frac{\la q}{b_3}f^{23}_\#-\frac {qb_2}{b_3^3}f^{24}_\#+
			\frac{qb_2+b_1b_3}{b_3^3}f^{34}_\#.$
			
			
			$\bullet$ $A_{4,5}^{-\al,\al}$, $|\al|<1$
			
			\begin{enumerate}
				\item The isomorphism from $\A_3$: $f_1=-\frac{pb_1+qb_2+b_1b_3}{(b_3+p)(p-b_3)}e_1+\frac{b_2}{b_3+p}e_2+e_3$, $f_2=-\frac{q}{2p}e_1+e_2$, $f_3=e_1$, $f_4=\frac{1}{b_3}e_4$.
				
				The generalized complex structure:
				
				$\begin{cases} J_1=\frac{1}{b_3}E_{14}+\frac{\la b_2}{p+b_3}E_{12}+\la E_{22}-\frac {b_2}{( p+b_3) b_3}E_{24}-\frac{\la(2pb_1+qb_2) }{2p( p-b_3)}E_{13}+\la E_{33}+\frac {2pb_1+qb_2}{2p( p-b_3) b_3}E_{34}-\frac{1}{b_3}E_{14}\\
				R_1=f_{23}^\#\\
				\si_1= (\la^2+1)(\frac{(2pb_1+qb_2)}{2p( p-b_3) }f^{12}_\#+\frac{b_2}{p+b_3}f^{13}_\#+f^{23}_\#).
				\end{cases}$

				We have
				\[ \phi(T)\exp(B)(J_1,R_1,\si_1)\exp(-B)\phi(T^{-1})=(E_{14}-E_{41},-f_{23}^\#,-f^{23}_\#),\]
				where

				$T=b_3E_{11}-E_{22}-\frac{\la b_2}{(p+b_3)b_3}E_{24}+E_{33}-\frac{\la( 2pb_1+qb_2)}{2p(p-b_3)b_3}E_{34}+E_{44}$,
				
				$B=-\la f^{23}_\#-\frac{2pb_1+qb_2}{2p(p-b_3)b_3} f^{24}_\#-\frac {b_2}{(p+b_3)b_3}f^{34}_\#.$

				\item  The isomorphism from $\A_3$: $f_1=-\frac{pb_1+qb_2+b_1b_3}{(b_3+p)(p-b_3)}e_1+\frac{b_2}{b_3+p}e_2+e_3$, $f_2=e_1$, $f_3=-\frac{q}{2p}e_1+e_2$, $f_4=\frac{1}{b_3}e_4$.
				The generalized complex structure:

				$\begin{cases} J_2=\frac{1}{b_3}E_{14}-\frac{(2pb_1+qb_2)\la}{2(p-b_3)p}E_{21}+\la E_{22}+\frac{2pb_1+qb_2}{2pb_3(p-b_3) }E_{24}+\frac{\la b_2}{p+b_3}E_{31}+\la E_{33}-\frac{b_2}{(p+b_3)b_3}E_{34}-b_3E_{41}\\
				R_2=-f_{23}^\#\\
				\si_2=(\la^2+1)(\frac{b_2}{p+b_3}f^{12}_\#+\frac{2pb_1+qb_2}{2p(p-b_3)}f^{13}_\#-f^{23}_\#).\\
				\end{cases}$
				
				We have
				\[ \phi(T)\exp(B)(J_2,R_2,\si_2)\exp(-B)\phi(T^{-1})=(E_{14}-E_{41},-f_{23}^\#,-f^{23}_\#),\]
				where
				
				$T=b_3E_{11}+E_{22}-\frac{(2pb_1+qb_2)\la}{2pb_3(p-b_3)}E_{24}+E_{33}+\frac{\la b_2}{b_3(p+b_3)}E_{34}+E_{44}$,
				
				$B=\la f^{23}_\#-\frac{b_2}{b_3(p+b_3)}f^{24}_\#-\frac{2pb_1+qb_2}{2pb_3(p-b_3)}f^{34}_\#.$

				\item  The isomorphism from $\A_3$: $f_1=-\frac{p^2b_2-prb_1-rb_1b_3-b_2}{r(b_3^2-1)}e_1-\frac{pb_2-rb_1-b_2b_3}{b_3^2-1}e_2+e_3$, $f_2=\frac{p-1}{r}e_1+e_2$, $f_3=\frac{p+1}{r}e_1+e_2$, $f_4=\frac{1}{b_3}e_4$.
				
				The generalized complex structure:
				
				$\begin{cases} J_3=\frac{1}{b_3}E_{14}+\frac{\la(pb_2-rb_1+b_2)}{2(b_3+1)}E_{21}+\la E_{22}-\frac{pb_2-rb_1+b_2}{2b_3(1+b_3)}E_{24}-\frac{\la (pb_2-rb_1-b_2)}{2(b_3-1)}E_{31}+\la E_{33}+\frac{pb_2-rb_1-b_2}{2 b_3(b_3-1)}E_{34}-b_3E_{41}\\
				R_3=\frac{r}{2}f_{23}^\#\\
				\si_3=(\la^{2}+1)(\frac{(pb_2-rb_1-b_2)}{r(b_3-1)}f^{12}_\#+\frac{( pb_2-rb_1+b_2)}{r(b_{{3}}+1)}f^{13}_\#+\frac{2}{r}f^{23}_\#).
				\end{cases}$
				
				We have
				\[ \phi(T)\exp(B)(J_3,R_3,\si_3)\exp(-B)\phi(T^{-1})=(E_{14}-E_{41},-f_{23}^\#,-f^{23}_\#),\]
				where
				
				$T=b_3E_{11}-\frac{2}{r}E_{22}-\frac{\la(pb_2-rb_1+b_2)}{rb_3(1+b_3)}E_{24}+E_{33}-
				\frac{\la ( pb_2-rb_1-b_2)}{2b_3(b_3-1)}E_{34}+E_{44}$,

				$B=-\frac{2\la}{r}f^{23}_\#-\frac{pb_2-rb_1-b_2}{rb_3(b_3-1)}f^{24}_\#-\frac{pb_2-rb_1+b_2}{rb_3(1+b_3)}f^{34}_\#.$
				
				\item  The isomorphism from $\A_3$: $f_1=-\frac{p^2b_2-prb_1-rb_1b_3-b_2}{r(b_3^2-1)}e_1-\frac{pb_2-rb_1-b_2b_3}{b_3^2-1}e_2+e_3$  $f_2=\frac{p+1}{r}e_1+e_2$, $f_3=\frac{p-1}{r}e_1+e_2$, $f_4=\frac{1}{b_3}e_4$.

				The generalized complex structure:
				
				$\begin{cases} J_4=\frac{1}{b_3}E_{14}-\frac{\la(pb_2-rb_1-b_2) }{2(b_3-1)}E_{21}+\la E_{22}+\frac{pb_2-rb_1-b_2}{2b_3(b_3-1)}E_{24}+ \frac {\la( pb_2-rb_1+b_2)}{2(1+b_3)} E_{31}+\la E_{33}-  \frac{pb_2-rb_1+b_2}{2b_3(1+b_3)} E_{34}-b_3E_{41}\\
				R_4=-\frac{r}{2}f_{23}^\#\\
				\si_4=(\la^2+1)(\frac{pb_2-rb_1+b_2}{r(b_3+1)}f^{12}_\#+\frac {pb_2-rb_1-b_2}{r(b_3-1)}f^{13}_\#-\frac{2}{r}f^{23}_\#).
				\end{cases}$
				
				We have
				\[ \phi(T)\exp(B)(J_4,R_4,\si_4)\exp(-B)\phi(T^{-1})=(E_{14}-E_{41},-f_{23}^\#,-f^{23}_\#),\]
				where
				
				$T=b_3E_{11}+\frac{2}{r}E_{22}-\frac{\la( pb_2-rb_1-b_2)}{rb_3(b_3-1)}E_{24}+E_{33}+ \frac{\la( pb_2-rb_1+b_2) }{2b_3(b_3+1)}E_{34}+E_{44}$,
				
				$B=\frac{2\la}{r}f^{23}_\#-\frac{pb_2-rb_1+b_2}{rb_3( 1+b_3)}f^{24}_\#-\frac{pb_2-rb_1-b_2}{rb_3(b_3-1)}f^{34}_\#.$

			\end{enumerate}

			$\bullet$  $A_{4,5}^{-1,\be}$,  $|\be|<1$
			\begin{enumerate}
				\item The isomorphism from $\A_3$: $f_1=\frac{p+1}{r}e_1+e_2$, $f_2=\frac{p-1}{r}e_1+e_2$, $f_3=-\frac{p^2b_2-prb_1-rb_1b_3-b_2}{r(b_3^2-1)}e_1-\frac{pb_2-rb_1-b_2b_3}{b_3^2-1}e_2+e_3$, $f_4=e_4$.
				
				The generalized complex structure:

				$\begin{cases} J_1=\la E_{11}-\frac{\la(pb_2-rb_1-b_2)}{2(b_3-1)}E_{13}+\frac{pb_2-rb_1-b_2}{2(b_3-1)}E_{14}+\la E_{22}+\frac{\la(pb_2-rb_1+b_2)}{2(b_3+1)}E_{23}-\frac{pb_2-rb_1+b_2}{2(b_3+1)}E_{24}+E_{34}-E_{43}\\
				R_1=-\frac{r}{2}f_{12}^\#\\
				\si_1=-(\la^2+1)(\frac{2}{r}f^{12}_\#+\frac{pb_2-rb_1+b_2}{r(b_3+1)}f^{13}_\#+\frac{pb_2-rb_1-b_2}{r(b_3-1)}f^{23}_\#).
				\end{cases}$

				We have
				\[ \phi(T)\exp(B)(J_1,R_1,\si_1)\exp(-B)\phi(T^{-1})=(E_{34}-E_{43},-f_{12}^\#,-f^{12}_\#). \]
				where
				
				$T=\frac{2}{r}E_{11}-\frac{\la( pb_2-rb_1-b_2)}{r(b_3-1)}E_{14}+E_{22}+\frac{\la(pb_2-rb_1+b_2)}{2(1+b_3)}E_{24}+E_{33}+E_{44}$,
				
				$B= \frac{2\la}{r}f^{12}_\#-\frac{pb_2-rb_1+b_2}{r(b_3+1)}f^{14}_\#-\frac{pb_2-rb_1-b_2}{r(+b_3-1)}f^{24}_\#.$
				
				\item The isomorphism from $\A_3$: $f_1=e_1$, $f_2=-\frac {q}{2p}e_1+e_2$,  $f_3=-\frac {pb_1+qb_2+b_1b_3}{p^2-b_3^2}e_1+\frac {b_2}{p+b_3}e_2+e_3$, $f_4=\frac{1}{p}e_4$.
				
				The generalized complex structure:
				
				$\begin{cases} J_2=\la E_{11}-\frac{\la (2pb_1+qb_2)}{2p(p-b_3)}+\frac{2pb_1+qb_2}{2p^2(p-b_3)}E_{14}+\la E_{22}+
				\frac{\la b_2}{b_3+p}-\frac {b_2}{p(p+ b_3)}E_{24}+\frac{1}{p}E_{34}-pE_{43}\\
				R_2=-f_{12}^\#\\
				\si_2=-(\la^2+1)(f^{12}_\#+\frac{b_2}{b_3+p}f^{13}_\#+\frac{2pb_1+qb_2}{2p(p -b_3)}f^{23}_\#).
				\end{cases}$

				We have
				\[\phi(T)\exp(B)(J_2,R_2,\si_2)\exp(-B)\phi(T^{-1})    =(E_{34}-E_{43},-f_{12}^\#,-f^{12}_\#), \]
				where
				
				$T=E_{11}-\frac{\la( 2pb_1+qb_2)}{2p^2(p-b_3)}E_{14}+E_{22}+\frac{\la b_2}{p(b_3+p)}E_{24}+
				pE_{33}+E_{44}$,

				$B=\la f^{12}_\#-\frac{b_2}{p( b_3+p)} f^{14}_\#-\frac{2pb_1+qb_2}{2p^2(p-b_{{3}})} f^{24}_\#.$
				
			\end{enumerate}

			$\bullet$ $A_{45}^{\al,\al}$, $-1<\al\leq1$
			\begin{enumerate}
				\item The isomorphism from $\B_3$: $f_1=e_1$, $f_2=-\frac{q_2}{x+1}e_1+e_4$, $f_3=-\frac{q_1}{x+1}e_1+e_3$, $f_4=e_2$.
				
				The generalized complex structure:
				
				$\begin{cases} J_1=\la E_{11}-\frac{\la q_2-q_1}{x+1}E_{12}-\frac{\la q_1+q_2}{x+1}E_{13}-E_{23}+E_{32}+\la E_{44}\\
				R_1=-f_{14}^\#\\
				\si_=(\la^2+1)(-f^{14}_\#+\frac{q_2}{x+1}f^{24}_\#+\frac{q_1}{x+1}f^{34}_\#).
				\end{cases}$	
				
				We have
				\[ \phi(T) \exp(B)(J_1,R_1,\si_1)\exp(-B)\phi(T^{-1})=(E_{23}-E_{32},-f_{14}^\#,-f^{14}_\#),\]
				where
				
				$T=E_{11}-E_{22}+E_{33}+E_{44}$,
				
				$B= \la f^{14}_\#-\frac {\la q_2-q_1}{x+1}f^{24}_\#-
				\frac {\la q_1+q_2}{x+1}f^{34}_\#.$
				\item The isomorphism from $\B_3$: $f_1=e_1$, $f_2=e_3$, $f_3=e_4$, $f_4=e_2$.
				
				The generalized complex structure:
				
				$\begin{cases} J_2=\la E_{11}+E_{23}-E_{32}+\la E_{44}\\
				R_2=-f_{14}^\#\\
				\si_2=-(\la^2+1)f^{14}_\#.
				\end{cases}$	
				
				We have
				\[  \exp(B)(J_2,R_2,\si_2)\exp(-B)=(E_{23}-E_{32},-f_{14}^\#,-f^{14}_\#), \]
				where $B=f^{14}_\#.$
				
			\end{enumerate}
			
			$\bullet$ $A_{4,5}^{\al,1}$, $|\al|<1$
			
			The isomorphism from $\B_3$: $f_1=-\frac{q_2x}{x+1}e_1+e_4$, $f_2=e_1$, $f_3=-\frac{q_1x}{x+1}e_1+e_3$, $f_4=e_2$ .
			
			The generalized complex structure:
			
			$\begin{cases} J=-E_{13}-\frac{x(\la q_2-q_1) }{x+1}E_{21}+\la E_{22}-\frac{x(\la q_1+q_2)}{x+1}E_{23}+E_{31}+\la E_{44}\\
			R=-f_{24}^\#\\
			\si=(\la^2+1)(f^{14}_\#-\frac{q_2 x}{x+1}f^{24}_\#+\frac{q_1 x}{x+1}f^{34}_\#).
			\end{cases}$

			We have
			\[ \exp(B)(J,R,\si)\exp(-B)=(-E_{13}+E_{31},-f_{24}^\#,-f^{24}_\#),\]
			where
			$B=-\frac{x(\la q_2-q_1 )}{x+1}f^{14}_\#+\la f^{24}_\#-\frac{x(\la q_1+q_2 )}{x+1}f^{34}_\#$.

			$\bullet$ $A_{4,5}^{-1,1}$
			\begin{enumerate}
				\item The isomorphism from $\A_3$: $f_1=\frac{b_2}{2p}e_2+e_3$, $f_2=-\frac{q}{2p}e_1+e_2$, $f_3=e_1$, $f_4=\frac{1}{p}e_4$.
				
				The generalized complex structure:
				
				$\begin{cases} J_1=\frac{1}{p}E_{14}+\frac{\la b_2}{2p}E_{12}+\la E_{22} -\frac{b_2}{2p^2}E_{24}+\frac{\la qb_2}{4p^2}E_{31}+\la E_{33}-\frac{b_2q}{p^3}E_{34}-pE_{14}\\
				R_1=f_{23}^\#\\
				\si_1=(\la^2+1)(-\frac{qb_2}{4p^2}f^{12}_\#+\frac{b_2}{2p}f^{13}_\#+f^{23}_\#).
				\end{cases}$

				We have
				\[\phi(T) \exp(B)(J_1,R_1,\si_1)\exp(-B)\phi(T^{-1})=(E_{41}-E_{14},-f_{23}^\#,-f^{23}_\#),\]
				where

				$T=pE_{11}-E_{22}-\frac{\la b_2}{2p^2}E_{24}+E_{33}+E_{44}$,
				
				$B=\frac {\la qb_2}{4p^2}f^{12}_\#-\la f^{23}_\#+\frac{b_2q}{4p^3}f^{24}_\#-\frac{b_2}{2p^2}f^{34}_\#.
				$
				\item The isomorphism from $\A_3$: $f_1=\frac{p+1}{r}e_1+e_2$, $f_2=\frac{p-1}{r}e_1+e_2$,  $f_3=-\frac{b_2}{r}e_1+e_3$, $f_4=e_4$.
				The generalized complex structure:
				
				$\begin{cases} J_2=\la E_{11}-\frac{\la b_2}{2}E_{13}+\frac{b_2}{2}E_{14}+\la E_{22}+\frac{\la b_2}{2}E_{23}-\frac{ b_2}{2}E_{24}+E_{34}-E_{43}\\
				R_2=-\frac{r}{2}f_{12}^\#\\
				\si_2=-(\la^2+1)(2f^{12}_\#+\frac{b_2}{r}f^{13}_\#+\frac{b_2}{r}f^{23}_\#).
				\end{cases}$

				We have
				\[\phi(T)\exp(B)(J_2,R_2,\si_2)\exp(-B)\phi(T^{-1})=(E_{41}-E_{14},-f_{23}^\#,-f^{23}_\#),\]
				where
				
				$T=E_{13}-\frac{2}{r}E_{22}-\frac{\la b_2}{r}E_{24}+E_{31}+E_{44}
				$,
				
				$B=\frac{2\la}{r}f^{13}_\#-\frac{b_2}{r}f^{14}_\#+\frac{\la b_2}{r}f^{23}_\#-\frac{b_2}{r}f^{24}_\#.
				$
				
				\item The isomorphism from $\A_3$: $f_1=\frac{p-1}{r}e_1+e_2$, $f_2=\frac{p+1}{r}e_1+e_2$, $f_3=\frac{b_2}{r}e_1-e_3$, $f_4=-e_4$.
				The generalized complex structure:
				
				$\begin{cases} J_3=\la E_{11}-\frac{\la b_2}{2}E_{13}+\frac{b_2}{2}E_{14}+\la E_{22}+\frac{\la b_2}{2}E_{23}-\frac{ b_2}{2}E_{24}+E_{34}-E_{43}\\
				R_3=\frac{r}{2}f_{12}^\#\\
				\si_3=(\la^2+1)(2f^{12}_\#+\frac{b_2}{r}f^{13}_\#+\frac{b_2}{r}f^{23}_\#).
				\end{cases}$

				We have
				\[\phi(T)\exp(B)(J_2,R_2,\si_2)\exp(-B)\phi(T^{-1})=(E_{41}-E_{14},-f_{23}^\#,-f^{23}_\#),\]
				where
				
				$T=E_{13}+\frac{2}{r}E_{22}+\frac{\la b_2}{r}E_{24}+E_{31}+E_{44}
				$,
				
				$B=-\frac{2\la}{r}f^{13}_\#+\frac{b_2}{r}f^{14}_\#-\frac{\la b_2}{r}f^{23}_\#+\frac{b_2}{r}f^{24}_\#.
				$

			\end{enumerate}
			
			$\bullet$ $A_{4,6}^{\al,\be}$
			\begin{enumerate}
				\item The isomorphism from $\B_4$: $f_1=e_1$, $f_2=-e_3$, $f_3=e_4$, $f_4=\frac{1}{y}e_2$.
				
				The generalized complex structure:

				$\begin{cases} J_1=\la E_{11}-E_{23}+E_{32}+\la E_{44}\\
				R_1=-yf_{14}^\#\\
				\si_1=-\frac{\la^2+1}{y}f^{14}_\#.
				\end{cases}$

				We have
				\[\phi(T) \exp(B)(J_1,R_1,\si_1)\exp(-B)\phi(T^{-1})=(-E_{23}+E_{32},-f_{14}^\#,-f^{14}_\#),\]
				where  $T=Id_4+\frac{1-y}{y}E_{11}$,  $B=\frac{\la}{y}f^{14}_\#$.

				\item The isomorphism from $\B_4$: $f_1=e_1$, $f_2=e_3$, $f_3=e_4$, $f_4=-\frac{1}{y}e_2$.
				
				The generalized complex structure:

				$\begin{cases} J_2=\la E_{11}+E_{23}-E_{32}+\la E_{44}\\
				R_2=yf_{14}^\#\\
				\si_2=\frac{\la^2+1}{y}f^{14}_\#.
				\end{cases}$
				
				We have
				\[\phi(T) \exp(B)(J_2,R_2,\si_2)\exp(-B)\phi(T^{-1})=(E_{23}-E_{32},-f_{14}^\#,-f^{14}_\#),\]
				where $T=Id_4-\frac{y+1}{y}E_{11}$, $B=-{\frac {\la}{y}}f^{14}_\#.$
				
				\item The isomorphism from $\B_3$:$f_1=e_1$, $f_2=e_1-\frac{xq_1+yq_2+q_1}{q_1^2+q_2^2}e_3-\frac {xq_2-yq_1+q_2}{q_1^2+q_2^2}e_4$, $f_3=-\frac {xq_2-yq_1+q_2}{q_1^2+q_2^2}e_3+\frac{xq_1+yq_2+q_1}{q_1^2+q_2^2}e_4$, $f_4=\frac1y e_2$.
				
				The generalized complex structure:

				$\begin{cases} J_3=\la E_{11}+\la E_{12}+E_{13}-E_{23}+E_{32}+\la E_{44}\\
				R_3=-y f_{14}^\#\\
				\si_3=-\frac{\la^2+1}{y}(f^{14}_\#+f^{24}_\#).
				\end{cases}$

				We have
				\[ \phi(T)\exp(B)(J_3,R_3,\si_3)\exp(-B)\phi(T^{-1})=(-E_{23}+E_{32},-f_{14}^\#,-f^{14}_\#), \]
				where
				
				$T=Id_4+\frac{y-1}{y}E_{11}$,
				
				$B=\frac{\la}{y}f^{14}_\#+\frac{\la}{y}f^{24}_\#+\frac{1}{y}f^{34}_\#.$
				
				\item  The isomorphism from $f_1=e_1$, $f_2=-e_1+\frac {xq_{{1}}+yq_{{2}}+q_{{1}}}{{q_{{1}}}^{2}+{q_{{2}}}^{2}}e_3+\frac {q_{{2}}x-q_{{1}}y+q_{{2}}}{{q_{{1}}}^{2}+{q_{{2}}}^{2}}e_4$, $f_3=-\frac {q_{{2}}x-q_{{1}}y+q_{{2}}}{{q_{{1}}}^{2}+{q_{{2}}}^{2}}e_3+\frac {xq_{{1}}+yq_{{2}}+q_{{1}}}{{q_{{1}}}^{2}+{q_{{2}}}^{2}}e_4$, $f_4=-\frac{1}{y} e_2$.
				
				The generalized complex structure:

				$\begin{cases} J_4=\la E_{11}-\la E_{12}+E_{13}+E_{23}-E_{32}+\la E_{44}\\
				R_4=y f_{14}^\#\\
				\si_4=\frac{\la^2+1}{y}(f^{14}_\#-f^{24}_\#).
				\end{cases}$
				
				We have
				\[ \phi(T)\exp(B)(J_4,R_4,\si_4)\exp(-B)\phi(T^{-1})=(E_{23}-E_{32},-f_{14}^\#,-f^{14}_\#), \]
				where
				
				$T=Id_4-\frac{y+1}{y}E_{11}$,
				
				$B=-\frac{\la}{y}f^{14}_\#+\frac{\la}{y}f^{24}_\#-\frac{1}{y}f^{34}_\#.$
			\end{enumerate}

			$\bullet$ $A_{4,6}^{\al,0}$
			
			The isomorphism from $\A_3$: $f_1=-\frac{p^2b_2-prb_1-rb_1b_3+b_2}{r(b_3^2+1)}e_1-\frac{pb_2-rb_1-b_2b_3}{b_3^2+1}e_2+e_3$,
			$f_2=e_1$, $f_3=pe_1+re_2$, $f_4=-e_4$.
			
			The generalized complex structure:
			
			$\begin{cases} J=-E_{14}-\frac{\la(pb_2b_3-rb_1b_3+b_2)}{r(b_3^{2}+1)}E_{21}+\la E_{22}-\frac{pb_2b_3-rb_1b_3+b_2}{r(b_3^2+1)}E_{24}-\frac{\la( pb_2-rb_1-b_2b_3)}{r(b_3^2+1)}E_{31}+\la E_{33}-\frac {pb_2-rb_1-b_2b_3}{r(b_3^2+1)}E_{34}+E_{41}\\
			R=-\frac{1}{r}f_{23}^\#\\
			\si=(\la^2+1)(-\frac{pb_2-rb_1-b_2b_3}{b_3^2+1}f^{12}_\#+\frac {pb_2b_3-rb_1b_3+b_2}{b_3^2+1}f^{13}_\#-rf^{23}_\#).
			\end{cases}$
			
			We have
			\[\phi(T) \exp(B)(J,R,\si)\exp(-B)\phi(T^{-1})= \frac{|r|}{r}(E_{14}-E_{14},-f_{23}^\#,-f^{23}_\#),\]
			where
			
			$T=-sE_{11}+\sqrt{|r|}E_{22}+\frac{s\la(pb_2b_3-rb_1b_3+b_2)}{\sqrt{|r|}(b_3^2+1)}E_{24}+\sqrt {|r|}E_{33}+{\frac {s\lambda\, (pb_2-rb_1-b_2b_3) }{\sqrt{|r|}(b_3^2+1)}}E_{34} +E_{44}$,

			$B= \la rf^{23}_\#-\frac{pb_2-rb_1-b_2b_3}{b_3^2+1}f^{24}_\#+\frac {pb_{{2}}b_{{3}}-rb_{{1}}b_{{3}}+b_{{2}}}{{b_{{3}}}^{2}+1}f^{34}_\#.$

			$\bullet$ $A_{9,4}^{-\frac{1}{2}}$
			
			The isomorphism from $\A_4$: $f_1=2q_1e_1$, $f_2=-(\frac{1}{3}b_2q_2+b_1)e_1-\frac{1}{3}b_2e_2+e_3$, $f_3=-q_2e_1+2e_2$, $f_4=-\frac{1}{2}e_4$.
			
			The generalized complex structure:
			
			$\begin{cases} J=\la E_{11}-\frac{\la (b_2q_2+2b_1)}{4q_1}E_{12}-\frac{b_2q_2+2b_1}{8q_1}E_{14}-\frac{1}{2}E_{24}-\frac{\la b_2}{6}E_{32}+\la E_{33}-\frac{b_2}{12}E_{34}+2E_{42}\\
			R=-\frac{1}{4q_1}f_{13}^\#\\
			\si=(\la^2+1)(\frac{2q_1b_2}{3}f^{12}_\#-4q_1f^{13}_\#+(b_2q_2+2b_1)f^{23}_\#).
			\end{cases}$

			We have
			\[ \exp(B)\phi(T)(J,R,\si)\phi(T^{-1})\exp(-B)=(-sE_{24}+sE_{42},-f_{13}^\#,-f^{13}_\#),\qquad s=\frac{|q_1|}{q_1}, \]
			where
			
			$T=2s\sqrt{2|q_1|}E_{11}+s\la b_2\frac{\sqrt{2|q_1|}}{3}E_{12}+\frac{\sqrt {2}(3
				\la b_2q_2+6\la b_1-2q_1b_2)
			}{12\sqrt{|q_1|}}E_{14}+2E_{22}+\sqrt{2|q_1|}E_{33}+\la b_2\frac{\sqrt{2|q_1|}}{12}E_{34}+E_{44}$,
			
			$B=\la f^{13}_\#- B_{14}f^{14}_\#-B_{23}f^{23}_\#- B_{34}f^{34}_\#.$
			
			where
			
			$B_{14}=\frac{\sqrt {2}}{12}(\sqrt{|q_1|}{\lambda}^{2}b_{{2}}+\sqrt{|q_1|}b_{{2}})$
			
			$B_{23}=\frac{\sqrt {2}}{6}(\sqrt{|q_1|}{\lambda}^{2}b_{
				{2}}+\sqrt{|q_1|}b_{{2}})$
			
			$B_{34}=-\frac{\sqrt {2}}{4}{\frac{\left( {\lambda}^{2}b_{{2}}q_{{2}}+2\,{\lambda}^{2}b_{{1}}+b_{{2}}q_{
						{2}}+2\,b_{{1}} \right)}{\sqrt{|q_1|}}}$.
			
			$\bullet$ $A_{9,4}^1$
			
			The isomorphism from $\B_3$: $f_1=e_1$, $f_2=-2q_1e_1+e_3$, $f_3=(-2q_1-2q_2)e_1+e_3+e_4$, $f_4=2e_2$.
			
			The generalized complex structure:
			
			$\begin{cases} J=\la E_{11}-2(\la q_1+2q_2)E_{12}-(2q_1(\la-1)+2q_2(\la+1))E_{13}+E_{22}+2E_{23}-E_{32}-E_{33}+\la E_{44}\\
			R=-\frac{1}{2}f_{14}^\#\\
			\si=(\la^2+1)(-2f_{14}^\#+4q_1f_{24}^\#+4(q_1+q_2)f_{34}^\#).
			\end{cases}$

			We have
			\[ \phi(T)\exp(B)(J,R,\si)\exp(-B)\phi(T^{-1})=(-E_{23}+E_{32},-f_{14}^\#,-f^{14}_\#), \]
			where
			
			$T=2E_{11}+\frac{\sqrt{3}+1}{2}E_{22}+E_{23}+\frac{\sqrt{3}-1}{2}E_{32}+\sqrt{3}E_{33}+E_{44}$,
			
			$B=2\la f^{14}_\#+\la  f^{23}_\#-(4\,\lambda\,q_{{1}}+4q_{{2}}) f^{24}_\#-(4q_1(\la-1)+4q_2(\la+1)) f^{34}_\#.
			$

			$\bullet$ $A^\al_{4,11}$
			
			The isomorphism from $\B_4$: $f_1=-se_1$, $f_2=\frac{4yq_2+2q_1}{4y^2 + 1}e_1-e_3$, $f_3=s(\frac{4yq_1-2q_2}{4y^2+1}e_1+e_4)$, $f_4=\frac{s}{y}e_2$, $s=\frac{|y|}{y}$.
			
			The generalized complex structure:
			
			$\begin{cases} J=\la E_{11}-s{\frac {4\la yq_2+2\la q_1-4yq_1+2q_2}{4{y}^{2}+1}}E_{12}-{\frac{4\la yq_1-2\la q_2+4yq_2+2q_1}{4{y}^{
						2}+1}}E_{23}-sE_{23}+sE_{32}+\la E_{44}\\
			R=y f_{14}^\#\\
			\si=(\la^2+1)(\frac{1}{y}f^{14}_\#-\frac{s(4yq_2+2q_1)}{y(4y^2+1)}f^{24}_\#-\frac{4yq_1-2q_2}{y(4y^2+1)})f^{34}_\#.
			\end{cases}$

			We have
			\[ \exp(B)(J,R,\si)\exp(-B)=(sE_{23}-sE_{32},yf_{14}^\#,\frac{1}{y}f^{14}_\#),\]
			where
			
			$B=-{\frac{\la}{y}}f^{14}_\#+{\frac {s(4\la yq_{{2}}+2\la q_{{1}}-4yq_{{1}}+2q_{{2}})}{y
					(4{y}^{2}+1)}}f^{24}_\# +{\frac {4\lambda\,yq_{{1}}-2\lambda\,q_{{2
						}}+4\,yq_{{2}}+2\,q_{{1}}}{y \left( 4\,{y}^{2}+1 \right) }}f^{34}_\#.$

				$\bullet$ $A_{4,12}$
				\begin{enumerate}
					\item
					
					The isomorphism from $\A_1$: $f_1=-e_3$, $f_2=e_4$, $f_3=-e_1+\frac{y_1}{y_2}e_2$, $f_4=\frac{1}{y_2}e_2$.
					
					The generalized complex structure:

					$\begin{cases} J_1=E_{12}-E_{21}+\la E_{33}+\la E_{44}\\
					R_1=-f_{34}^\#\\
					\si_1=-\frac{\la^2+1}{y_2}f^{34}_\#.
					\end{cases}$

					We have
					\[ \exp(B)(J_1,R_1,\si_1)\exp(-B)=(E_{12}-E_{21},-y_2f_{34}^\#,-\frac{1}{y_2}f^{34}_\#),\]
					where $B=\frac{\la}{y_2}f^{34}_\#$.

					\item The isomorphism from $\A_5$: $f_1=pe_2+e_3+e_4$, $f_2=-pe_1+e_3-e_4$, $f_3=-\frac12e_1-\frac12e_2$, $f_4=\frac12e_1-\frac12e_2$.
					
					The generalized complex structure:
					
					$\begin{cases} J_2=-E_{12}+E_{21}-p(\la+1)E_{31}+p(\la-1)E_{32}+\la E_{33}-p(\la-1)E_{41}-p(\la+1)E_{42}+\la E_{44}\\
					R_2=-2f_{34}^\#\\
					\si_2=-(\la^2+1)(p^2f^{12}_\#+\frac{p}{2}f^{13}_\#-\frac{p}{2}f^{14}_\#+\frac{p}{2}f^{23}_\#+\frac{p}{2}f^{24}_\#+\frac{1}{2}f^{34}_\#).
					\end{cases}$

					we have

					\[ \exp(B)(J_2,R_2,\si_2)\exp(-B)=(-E_{12}+E_{21},-2f_{34}^\#,-\frac{1}{2}f^{34}_\#),\]
					where
					
					$B=\frac{\la p-p}{2}f^{13}_\#-\frac{\la p+p}{2}f^{14}_\#+\frac{\la p+p}{2}f^{23}_\#
					-\frac{\la p-p}{2}f^{24}_\#+\frac{\la}{2}f^{34}_\#.$
				\end{enumerate}

\end{document}